%% file: zempan_frechetProcrustes.tex
\documentclass[noinfoline]{imsart}

\RequirePackage[OT1]{fontenc}
\RequirePackage{amsthm,amsmath}
\RequirePackage[numbers]{natbib}

\usepackage{verbatim}

\allowdisplaybreaks

\usepackage{fullpage}

\usepackage{algorithm,algorithmic}
\usepackage{subfigure}
\usepackage{layout}

\usepackage{dsfont}

\usepackage[mathscr]{eucal}
\usepackage{mathrsfs}
\usepackage{color}
\usepackage{pifont}
\usepackage{bm}
\usepackage{latexsym}
\usepackage{amsmath}
\usepackage{amsthm}
\usepackage{amsfonts}
\usepackage{amssymb}
\usepackage{epsfig}
\usepackage{graphicx}
\newtheorem{theorem}{Theorem}
\newtheorem{proposition}{Proposition}
\newtheorem{corollary}{Corollary}
\newtheorem{lemma}{Lemma}

\newtheorem{remark}{Remark}

\input xy
\xyoption{all}

\usepackage[normalem]{ulem}

\input{macros}

\newcommand{\Tan}{\mathrm{Tan}}

\newcommand{\tocless}[2]{\bgroup\let\addcontentsline=\nocontentsline#1{#2}\egroup}

\startlocaldefs
\numberwithin{equation}{section}
\theoremstyle{plain}

\endlocaldefs

\begin{document}

\begin{frontmatter}

\title{{ Fr\'echet Means and Procrustes Analysis\\in Wasserstein Space\thanksref{T1}}}

\runtitle{Fr\'echet Means and Procrustes Analysis in Wasserstein Space}

\begin{aug}
\author{\fnms{Yoav} \snm{Zemel}\ead[label=e1]{yoav.zemel@epfl.ch}} \and
\author{\fnms{Victor M.} \snm{Panaretos}\ead[label=e2]{victor.panaretos@epfl.ch}}

\thankstext{T1}{Research supported by an ERC Starting Grant Award to Victor M. Panaretos.}

\runauthor{Y. Zemel \& V.M. Panaretos }

\affiliation{Ecole Polytechnique F\'ed\'erale de Lausanne}

\address{Institut de Math\'ematiques\\
Ecole Polytechnique F\'ed\'erale de Lausanne\\
1015 Lausanne, Switzerland\\
\printead{e1}, \printead*{e2}}

\end{aug}

\begin{abstract}
We consider two statistical problems at the intersection of functional and non-Euclidean data analysis: the determination of a Fr\'echet mean in the Wasserstein space of multivariate distributions; and the optimal registration of deformed random measures and point processes. We elucidate how the two problems are linked, each being in a sense dual to the other. We first study the finite sample version of the problem in the continuum. Exploiting the tangent bundle structure of Wasserstein space, we deduce the Fr\'echet mean via gradient descent. We show that this is equivalent to a Procrustes analysis for the registration maps, thus only requiring successive solutions to pairwise optimal coupling problems. We then study the population version of the problem, focussing on inference and stability: in practice, the data are i.i.d. realisations from a law on Wasserstein space, and indeed their observation is discrete, where one observes a proxy finite sample or point process. We construct regularised nonparametric estimators, and prove their consistency for the population mean, and uniform consistency for the population Procrustes registration maps.
\end{abstract}

\begin{keyword}[class=AMS]
\kwd[Primary ]{62M30; 60D05}
\kwd[; secondary ]{62G07; 60G55}
\end{keyword}

\begin{keyword}
\kwd{Functional Data Analysis}
\kwd{Manifold Statistics}
\kwd{Optimal Transportation}
\kwd{Phase Variation}
\kwd{Point Process}
\kwd{Random Measure}
\kwd{Registration}
\kwd{Warping}
\end{keyword}

\end{frontmatter}

\vspace{-0.4cm}

\tableofcontents

\newpage

\section{Introduction}
Functional data analysis (e.g. Hsing \& Eubank \cite{hsing-book}) and non-Euclidean statistics (e.g. Patrangenaru \& Ellingson \cite{patrangenaru2015nonparametric}) represent modern areas of statistical research, whose key challenges arise from the intrinsic complexity of the data and the peculiarities of their ambient space. In the first case, the data are random elements in a separable Hilbert space of functions (typically $L^2[0,1]$), and resulting challenges are linked to infinite dimensionality (e.g. ill-posed studentisation, Munk et al. \cite{munk2008one}, and discrete measurements of continuum random objects, Zhang \& Wang \cite{zhang2016sparse}). In the second case, the data are seen as random elements of a finite-dimensional Riemannian manifold (often a shape space), and resulting challenges are linked to the non-linear structure of the space (e.g. existence/uniqueness of Fr\'echet means, Le \cite{le1998consistency} and Kendall \cite{kendall2010survey}, and analysis of manifold variation, Huckemann, Munk \& Hotz \cite{huckemann2010intrinsic}).

At the intersection of these two domains, with manifestations in neurophysiology, imaging, and environmetrics, one finds data objects that are best modelled as \emph{distributions} over $\mathbb{R}^d$, that is, \emph{random measures} (Stoyan, Kendall \& Mecke \cite{chiu2013stochastic}, Kallenberg \cite{kallenberg1986random}). Such random measures carry the infinite dimensional traits of functional data, but at the same time are characterised by intrinsic non-linearities due to their positivity and integrability constraints, requiring a non-Euclidean point of view. Indeed, despite their functional nature, their dominating variational feature is not due to additive amplitude fluctuations (as can be seen in the Karhunen-Lo\`eve expansion of functional data), but rather to \emph{random deformation} of a structural mean (as in Freitag \& Munk \cite{freitag2005hadamard}) or template (as in \emph{morphometrics}, Bookstein \cite{bookstein1997morphometric}). Still, being infinite dimensional, their observation is typically done discretely, for example noisily over a grid (e.g.\ Amit et al.\ \cite{amit1991structural}, Allassonni{\`e}re et al.\ \cite{allassonniere2007towards}) or via random sampling (e.g. Panaretos \& Zemel \cite{panzem15}), requiring tools and techniques from nonparametric statistics, as used in functional data analysis.

In this setting, the typical statistical objective is to estimate the underlying template that gives rise to the data by random deformation. This can often be modelled as a Fr\'echet mean with respect to some metric structure; dual to this problem is the recovery the deformation maps themselves, in order to \emph{register} the individual realisations in a common coordinate system, given by \emph{registration maps}. These problems are interwoven in \emph{shape theory}, where the template and registration maps are the two ingredients of Procrustes analysis (Gower \cite{gower1975generalized}; Dryden \& Mardia \cite{dryden1998statistical}) and non-Euclidean PCA (Huckemann, Munk \& Hotz \cite{huckemann2010intrinsic}; Huckemann \& Ziezold \cite{huckemann2006principal}). Obviously, the methods and algorithms for estimating a mean and carrying out a registration/Procrustes analysis are inextricably linked with the geometry of the sample space, which can be a matter of modelling choice or of first principles.

In this paper, we choose to study the problem of \emph{Fr\'echet averaging} and \emph{Procrustes registration} when the data are viewed as elements of the $L^2$-Wasserstein space of multivariate measures on $\mathbb{R}^d$. We choose this setting since it has a long history in assessing compatibility and fit of distributions related via deformations (Munk \& Czado \cite{munk1998nonparametric}; Freitag \& Munk \cite{freitag2005hadamard}), and as it can be seen to be a natural analogue of using $L^2$, in the case of measures\footnote{In the sense that the Wasserstein space is topologically homeomorphic to a convex subset of $L^2([0,1]^d)$;  when $d=1$, this homeomorphism is an isometry, whereas for $d>1$, it is a local isometry.} (Panaretos \& Zemel \cite{panzem15}; Bigot \& Klein \cite{bigot2012consistent}). We work at both a sample level and a population level, as well as both at the level of continuum and discrete observation: our object of study is the determination of the Fr\'echet mean and registration maps at the level of a sample, as well as at their estimation when the observed measures are discretely observed realisations from a population of random measures.  When $d=1$, the problem is well understood, owing to the flat geometry of Wasserstein space (Panaretos \& Zemel \cite{panzem15}). When $d>1$, however, the Wasserstein space has non-negative curvature, and one encounters the classical difficulties of non-Euclidean statistics, augmented by the infinite dimensionality and discrete measurement of the problem (see Anderes et al.\ \cite{anderes2015discrete}, Sommerfeld \& Munk \cite{sommerfeld2016inference} and Tameling et al.\ \cite{tameling2017empirical} for challenges involved in the discrete setting).

\smallskip
\noindent In more detail, our contributions are:
\begin{itemize}

\item[(A)] \underline{At the sample level:}  we illustrate how knowledge of the Fr\'echet mean (template) gives an explicit solution to the optimal registration/multicoupling problem (Section \ref{population_to_sample}, Proposition \ref{multicoupling}). We study the tangent space geometry,  using it to determine the gradient of the Fr\'echet functional (Section \ref{section:gradient}, Theorem \ref{gradient_definition}), and characterise Karcher means via its zeroes (Corollary \ref{thm:minisstationary}, Section \ref{characterisation_mean}). We give criteria for determining when a Karcher mean (local optimum) is a Fr\'echet mean (global optimum; Theorem \ref{thm:optimalKarcher}). We construct a gradient descent algorithm (Algorithm \ref{algo}), and find its optimal stepsize (Lemma \ref{lem:iterationbound}) illustrating the algorithm structurally equivalent to a Procrustes algorithm (Section \ref{algorithm_section}), reducing the determination of the mean to the successive solution of pairwise optimal transport problems. We prove that the gradient iterate converges to a Karcher mean in the Wasserstein metric (Subsection~\ref{convergence_section}, Theorem \ref{thm:convtosta}); and that the induced transportation maps converge uniformly to the Procrustes maps (required for optimal mutlicoupling; Theorem \ref{thm:convprocrustes}, Section \ref{convergence_procrustes}). The latter is particularly involved and requires techniques from the geometry of monotone operators on $\mathbb{R}^d$.  {As a noteworthy corollary, we deduce convergence of the multicouplings (Corollary~\ref{cor:convMulti}).}

\item[(B)] \underline{At the population level:} we consider a population level model linking Fr\'echet means and optimal registration and give conditions for model identifiability (Section \ref{sec:discrete_intro}, Theorem~\ref{thm:meanTid});  We then tackle the problem of point estimation of the population mean and registration maps in a functional data analysis setup, where instead of observing an i.i.d. sample $\{\mu^1,\dots,\mu^N\}$ from the population, we observe samples or point processes with these measures as distributions/intensities. In this setting, we construct regularised nonparametric estimators of the Fr\'echet means and Procrustes maps, and prove that they are consistent in Wasserstein distance and uniform norm, respectively (Theorems \ref{thm:consistency} and \ref{thm:consistencyWarp}).
\end{itemize}

\noindent Before presenting our main results, we first provide a short introduction to Wasserstein space in Section \ref{wasser_intro}.  Section \ref{proofs} gathers the main proofs, for the sake of tidiness, and Section \ref{SEC:EX} presents several interesting special examples as an illustration.  Section~\ref{zempan16supp} supplements the main article, providing further technical details.

\noindent In reviewing an earlier version of our paper (\cite{zempan-report}, February 2016), a referee brought to our attention independent parallel work by \'Alvarez-Esteban et al., that had concurrently (January 2016) been submitted for publication in an analysis journal (and has now appeared, see  \cite{alvarez2016fixed}). Their work overlaps with part of ours in (A) above (Subsections \ref{procrustes_analogy} and \ref{convergence_section}). In particular, they too arrive at a (structurally) same algorithm (Algorithm \ref{algo}). Their motivation, construction, and convergence proof differ substantially from ours (theirs is a fixed point iteration heuristically motivated by the Gaussian case, while their proof uses almost sure representations). Indeed, our geometrical framework and proof techniques is what allows us to study the problem of optimal registration (Procrustes analysis), requiring a careful study of the stochastic convergence of monotone operators on $\mathbb{R}^d$ (Section \ref{appendix}).

\section{Optimal Transportation and Wasserstein Space}\label{wasser_intro}

The reason the Wasserstein space arises as the natural space to capture deformation-based variation of random measures lies in its deep connection with the problem of \emph{optimal transportation of measure}. This consists in solving the \emph{Monge problem} (Villani \cite{vil}): given a pair of measures $(\mu,\nu)$, find a mapping $\mathbf{t}_{\mu}^{\nu}:\mathbb{R}^d\mapsto\mathbb{R}^d$ such that $\mathbf{t}_{\mu}^{\nu}\#\mu=\nu$, and
\[
\ownint{\mathbb{R}^d}{}{\left\|\mathbf{t}_{\mu}^{\nu}(x)-x\right\|^2}{\mu(x)}
\le \ownint{\mathbb{R}^d}{}{\left\|\mathbf{q}(x)-x\right\|^2}{\mu(x)},
\]
for any other $\mathbf{q}$ such that $\mathbf{q}\#\mu=\nu$. Here, ``$\#$" denotes the push-forward operation, where $[\mathbf{t}\#\mu](A)=\mu(\mathbf{t}^{-1}(A))$ for all Borel sets $A$ of $\mathbb{R}^d$. The map $\mathbf{t}_{\mu}^{\nu}$ is called an optimal transport plan, and a solution to this problem yields an optimal deformation of $\mu$ into $\nu$ with respect to the \emph{transport cost} given by squared Euclidean distance.

An optimal transport map may fail to exist, and instead, one may need to solve the relaxed Monge problem, known as the \emph{Kantorovich} problem (Villani \cite{vil}). Here instead of seeking a map $\mathbf{t}_{\mu}^{\nu}\#\mu=\nu$, one seeks a distribution $\xi$ on $\mathbb{R}^d\times \mathbb{R}^d$ with marginals $\mu$ and $\nu$, minimising the functional
\[
\ownint{\mathbb{R}^d\times\mathbb{R}^d}{}{\|x-y\|^2}{\xi(x,y)}
\]
over all measures $\xi$ on $\mathbb{R}^d\times \mathbb{R}^d$ with marginals $\mu$ and $\nu$. In probabilistic terms, $\xi$ yields a coupling of random variables $X\sim\mu$ and $Y\sim \nu$ that minimises the quantity
$$\mathbb{E}\|X-Y\|^2,$$ over all possible couplings of $X$ and $Y$. It can be shown that when the measure $\mu$ is regular (absolutely continuous with respect to Lebesgue measure), the Kantorovich problem reduces to the Monge problem, and the optimal coupling $\xi$ is supported on the graph of the function. That is, the optimal coupling exists, is unique, and can be realised by a proper transport map $\mathbf{t}_{\mu}^{\nu}$.

One may consider the space $\mathcal{P}_2(\mathbb{R}^2)$ of all probability measures $\mu$ on $\mathbb{R}^d$ with finite variance (that is, $\ownint{\mathbb{R}^d}{}{\|x\|^2}{\mu(x)}<\infty$) as a metric space, endowed with the $L^2$-Wasserstein distance
\[
d(\mu,\nu)=\inf_{\xi\in \Gamma(\mu,\nu)}\sqrt{\ownint{\mathbb{R}^d\times\mathbb{R}^d}{}{\|x-y\|^2}{\xi(x,y)}},
\]
where $\Gamma(\mu,\nu)$ is the set of probability measures on $\mathbb{R}^d\times \mathbb{R}^d$ with marginals $\mu$ and $\nu$. The induced metric space is colloquially called \emph{Wasserstein space} and will form the geometrical context for our study of \emph{deformation-based} variation of random measures. This space has been used extensively in statistics, as it metrises the topology of weak convergence, and convergence with respect to the metric yields both convergence in law, as well as convergence of the first two moments (for instance, in applications to the bootstrap, e.g. Bickel \& Freedman \cite{bickel1981some}, and to goodness-of-fit, e.g. Rippl, Munk \& Sturm \cite{Rippl201690}).

The appropriateness of this distance when modeling deformations of measures becomes clear based on our previous remark concerning regularity: one can imagine an initial regular template $\mu$, that is \emph{deformed} according to maps $\mathbf{q}_i$ to yield new measures $\mu^i=(\mathbf{q}_i)\#\mu$. It is then natural to quantify the distance of the template to its perturbations by means of the minimal transportation (or deformation) cost
\[
d(\mu,\mu^i)
=\sqrt\ownint{\mathbb{R}^d}{}{\left\|\mathbf{t}_{\mu}^{\mu^i}(x)-x\right\|^2}{\mu(x)}.
\]
That the distance can be expressed via a proper map, is due to the assumed regularity of $\mu$. Note that the maps $\mathbf{q}_i$ themselves will, in general, not be identifiable (many Borel maps can push $\mu$ forward to $\mu^i$). But they can be assumed to be exactly optimal, i.e.\ $\mathbf{q}_i=\mathbf{t}_{\mu}^{\mu^i}$ as a matter of parsimony, and in any case without loss of generality, leading to identifiability. These maps will also solve the registration problem: a map of the form $\mathbf{t}_{\mu}^{\mu^i}-\mathbf{i}$, with $\mathbf{i}$ the identity mapping, shows how the coordinate system of $\mu$ should be deformed to be registered to the coordinate system of $\mu^i$.

This raises the question of how to \emph{characterise} the optimal transportation maps. For instance, in the one-dimensional case, if $\mu$ and $\nu$ are probability measures on $\mathbb{R}$, and $\mu$ is diffuse we may write
\begin{equation}\label{eq:optimal1d}
\mathbf{t}_{\mu}^{\nu}
=G_\nu^{-1}\circ G_{\mu},
\end{equation}
where $G_{\mu}(t)=\ownint {-\infty} t{}{\mu(x)}$, $G_{\nu}(t)=\ownint{-\infty}t{}{\nu(x)}$ are their distribution functions and $G^{-1}_{\nu}$ is the quantile function of $\nu$. This characterises optimal maps in one dimension as non-decreasing functions. More generally, when one has measures on $\mathbb{R}^d$, the class of optimal maps can be seen to be that of \emph{monotone maps} (see Section \ref{appendix}), defined as fields $\mathbf{t}:\mathbb{R}^d\rightarrow\mathbb{R}^d$ that are obtained as gradients of convex functions $\varphi:\mathbb{R}^d\rightarrow \mathbb{R}$,
$$\mathbf{t}=\nabla\varphi.$$
This is known as Brenier's characterisation (Villani \cite[Theorem~2.12]{vil}). With these basic definitions in place, we are now ready to consider the problem of finding a Fr\'echet mean of a collection of measures -- the latter viewed as the common template measure that was deformed to give rise to these measures.

\section{Sample Setting}\label{sample_setting}

\subsection{Fr\'echet Means and Optimal Registration}\label{population_to_sample}

The notion of a Fr\'echet mean (Fr\'echet \cite{frechet1948elements}) generalises that of the mean in a normed vector space to a general metric space. Though it has primarily been studied on Riemannian manifolds, the generality of its definition allows it to be used very broadly: it replaces the usual ``sum of squares", with a ``sum of squared distances", the \emph{Fr\'echet functional}. A closely related notion is that of a \emph{Karcher mean} (Karcher \cite{karcher1977riemannian}; Le \cite{le1995mean}), a term that describes stationary points of the sum of squares functional, when the latter is differentiable. See Kendall \cite{kendall2010survey}, and Kendall \& Le \cite{kendall2011limit} for an overview and a detailed review, respectively. In the context of Wasserstein space,  a Fr\'echet mean of a collection of measures $\{\mu^1,\ldots,\mu^N\}$, is a minimiser of the Fr\'echet functional
\begin{equation}\label{frechet_functional}
F(\gamma):=\frac{1}{2N}\sum_{i=1}^{N}d^2(\mu^i,\gamma)
\end{equation}
over elements $\gamma$ in the Wasserstein space $\mathcal{P}_2(\mathbb{R}^d)$, and a Karcher mean is a stationary point of $F$. The functional will be finite for any $\gamma\in \mathcal{P}_2(\mathbb{R}^d)$, provided that it is so for some $\gamma_0$. Population versions, assuming $\mathcal{P}_2(\mathbb{R}^d)$ is endowed with a probability measure, can also be defined, replacing summation by expectation with respect to that law. Interestingly, Fr\'echet himself \cite{frechet1957distance} considered the Wasserstein metric between probability measures on $\mathbb{R}$, and some refer to this as the \emph{Fr\'echet distance} (e.g.\ Dowson \& Landau \cite{dowson1982frechet}). In general, existence and uniqueness of a sample Fr\'echet mean can be subtle, but Agueh \& Carlier \cite{bary} have shown that it \emph{will uniquely exist} in the Wasserstein space, provided that some regularity is asserted\footnote{For a population version, one needs to tackle measurability and identifiability issues, see Section \ref{sec:discrete_intro}}. Here and in the following, we call a measure \textit{regular} if it is absolutely continuous with respect to Lebesgue measure (this condition can be slightly weakened \cite{bary}).
\begin{proposition}[Agueh \& Carlier \cite{bary}]
Let $\{\mu^1,\ldots,\mu^N\}$ be a collection in the Wasserstein space of measures $\mathcal{P}_2(\mathbb{R}^d)$.  If at least one of the measures is regular with bounded density, then their Fr\'echet mean exists, is unique, and is regular.
\end{proposition}

We will now show that, once the Fr\'echet mean $\bar{\mu}$ of $\{\mu^1,\ldots,\mu^N\}$ has been determined, it may be used to optimally multi-couple the measures $\{\mu^1,\ldots,\mu^n\}$ in $\mathbb{R}^{d\times N}$, in terms of pairwise mean square distances, thus providing a solution to the \emph{multidimensional Monge--Kantorovich problem} considered by Gangbo \& \'Swi\c{e}ch \cite{gangbo1998optimal}. That is, $\bar{\mu}$ can be used to construct a random vector whose marginals are as concentrated as possible in terms of pairwise mean-square distance, subject to the constraint of having laws $\{\mu^1,\ldots,\mu^N\}$.

\smallskip
\noindent Our first result combines results of \cite{bary} and \cite{gangbo1998optimal} to illustrate precisely how (also see Pass \cite[Theorem~4.2.2]{pass2013optimal} for an analogous result when considering continuous flows of measures).
\begin{proposition}[Optimal Multicoupling via Fr\'echet Means]\label{multicoupling}
Let $\{\mu^1,\ldots,\mu^N\}$ be regular probability measures in $\mathcal{P}_2(\mathbb{R}^d)$, one with bounded density, and let $\bar{\mu}$ be their (unique) Fr\'echet mean with respect to the Wasserstein metric. Let $Z\sim \bar{\mu}$ and define
$$\bm{X}=(X_1,\ldots,X_N),\qquad X_i=\mathbf{t}_{\bar{\mu}}^{\mu^i}(Z),\qquad i=1\ldots,N,$$
where $\mathbf{t}_{\bar{\mu}}^{\mu^i}$ is the optimal transport plan pushing $\bar{\mu}$ forward to $\mu^i$. Then $X_i\sim\mu^i$ for $i=1,\ldots,N$ and furthermore,
\[
\sum_{i=1}^N\sum_{j=i+1}^N\mathbb E\|X_i - X_j\|^2
\le\sum_{i=1}^N\sum_{j=i+1}^N\mathbb E\|Y_i - Y_j\|^2
\]
for any other $\bm{Y}=(Y_1,\ldots,Y_N)$ such that $Y_i\sim \mu^i$, $i=1,\ldots,N$.
\end{proposition}

In the language of shape theory, the Fr\'echet mean $\bar{\mu}$ may be used as a \emph{template} to \emph{jointly register} the collection of measures, just as Euclidean configurations can be registered to their Procrustes mean by a Procrustes analysis (Goodall \cite{goodall1991procrustes}). Only in this case, instead of the similarity group of shape theory, registration is \emph{deformation based}, by means of the collection of maps $\{\mathbf{t}_{\bar{\mu}}^{\mu^i}\}_{i=1}^{N}$, where $\mathbf{t}_{\bar{\mu}}^{\mu^i}$ is the optimal transport map
\[
\mathbf{t}_{\bar{\mu}}^{\mu^i}\#\bar\mu
=\mu^i.
\]
By analogy to shape theory, we shall refer to these as \emph{Procrustes maps}. These yield a common coordinate system (corresponding to $\bar{\mu}$) where one can best compare samples from each measure, similarly to ``quantile renormalisation" in one dimension, e.g.\ Bolstad et al.\ \cite{bolstad2003comparison}, Gallon et al.\ \cite{gallon2013statistical}. The Procrustes maps can also be used in order to produce a Principal Component Analysis, capturing the main modes of deformation-based variation (Bigot et al.\ \cite{bigot2013geodesic}, Panaretos \& Zemel \cite{panzem15}; Huckemann, Munk \& Hotz \cite{huckemann2010intrinsic}, Wang et al.\ \cite{wang2013linear}).

\subsection{Wasserstein Geometry and the Gradient of the Fr\'echet Functional}\label{gradient_section}

In this section, we determine the conditions for the Fr\'echet derivative of the Fr\'echet functional \eqref{frechet_functional} to be well defined, and determine its functional form. Furthermore, we characterise Karcher means and give criteria for their optimality, opening the way for the determination of the Fr\'echet mean. The key to our analysis will be to exploit the tangent bundle over the Wasserstein space of regular measures.

\subsubsection{The Tangent Bundle}\label{tangent_bundle}

  Let $\mathcal P_2(\mathbb R^d)$ be the Wasserstein space of probability measures $\mu$ on $\mathbb{R}^d$ such that $\ownint{\mathbb R^d}{}{\|x\|^2}{\mu(x)}$ is finite, as defined in Section \ref{wasser_intro}.  An absolutely continuous measure on $\mathbb R^d$ will be called \emph{regular}. When $\mu^0\in\mathcal P_2(\mathbb{R}^d)$ is regular and $\mu^1\in\mathcal P_2(\mathbb{R}^d)$, the transportation map $\mathbf t_{\mu^0}^{\mu^1}$ uniquely exists, in which case there is a unique geodesic curve between $\mu^0$ and $\mu^1$.  Using again the notation $\mathbf i$ for the identity map, this geodesic is given by
\[
\mu_t=\left[\mathbf i+t(\mathbf t_{\mu^0}^{\mu^1} - \mathbf i)\right]
\#\mu^0,\qquad t\in[0,1].
\]
This curve is known as McCann's interpolation (McCann \cite{mccann1997convexity}, Villani \cite{vil}). 
The tangent space at an arbitrary $\mu\in\mathcal P_2(\mathbb{R}^d)$ is then (Ambrosio et al.\ \cite[Definition 8.4.1, p.\ 189]{ambgigsav})
\[
\Tan_{\mu}
=\Tan_\mu\mathcal P_2(\mathbb R^d)
=\overline{\{\nabla\varphi:\varphi\in C_c^\infty(\mathbb R^d)\}}^{L^2(\mu)},
\]
where $C_c^\infty(\mathbb R^d)$ denotes infinitely differentiable functions $\varphi:\mathbb R^d\to\mathbb R$ with compact support, and the closure operation is taken with respect to the space $L^2(\mu)$.  Note the interesting fact that the closure operation is the only aspect of the tangent space that directly involves the measure $\mu$. An equivalent definition, which is more useful to us, is given by Ambrosio et al.\ \cite[Definition 8.5.1, p.\ 195]{ambgigsav}:
\[
\Tan_\mu
=\overline{\{\lambda(\mathbf r-\mathbf i):\mathbf r\textrm{ optimal between }\mu \textrm{ and }\mathbf r\#\mu; \lambda>0\}}^{L^2(\mu)},
\]
that is, we take the collection of $\mathbf r$'s that are optimal maps from $\mu$ to $\mathbf r\#\mu$; i.e.\ the gradients of convex functions. This is a linear space (not just a cone) by the first definition, even though it is not obvious from the second. The definitions are equivalent by Theorem~8.5.1 of Ambrosio et al.\ \cite[p.\ 195]{ambgigsav}. As was mentioned above, when $\mu^0\in \mathcal P_2(\mathbb R^d)$ is regular, every measure $\mu^1\in\mathcal P_2(\mathbb R^d)$ admits a unique optimal map $\mathbf t_{\mu^0}^{\mu^1}$ that pushes $\mu^0$ forward to $\mu^1$.  Thus, the exponential map
\[
{\rm exp}_{\mu^0}(\mathbf r - \mathbf i)
= \mathbf r\#\mu^0
\]
is surjective, and its inverse, the log map
\[
\log_{\mu^0}(\mu^1)
=\mathbf t_{\mu^0}^{\mu^1} - \mathbf i,
\]
is well-defined throughout $\mathcal P_2(\mathbb R^d)$.  In particular, the geodesic $\left[\mathbf i+t(\mathbf t_{\mu^0}^{\mu^1} - \mathbf i)\right]\#\mu^0$ is mapped bijectively to the line segment $t(\mathbf t_{\mu^0}^{\mu^1} - \mathbf i)\in \Tan_{\mu^0}$ through the log map.

\subsubsection{Gradient of the Fr\'echet functional}\label{section:gradient}

We will now exploit the tangent bundle structure described in the previous section in order to determine the gradient of the empirical Fr\'echet functional. Fix $\mu^0\in\mathcal P_2(\mathbb{R}^d)$ and consider the function
\[
F_0:\mathcal P_2(\mathbb{R}^d)\to\mathbb{R},
\qquad F_0(\mu)=\frac12d^2(\mu,\mu^0).
\]
When $\mu$ is regular, we have that (\cite[Corollary~10.2.7, p.\ 239]{ambgigsav}), for any $\mu^0$
\[
\lim_{\nu\to\mu}\frac{F_0(\nu)- F_0(\mu) +\displaystyle \ownint {\mathbb R^d}{}{\innprod{\mathbf t_\mu^{\mu^0}(x)-x}{\mathbf t_\mu^\nu(x)-x}}{\mu(x)}}
{d(\nu,\mu)}
=0,
\]
where the convergence $\nu\to\mu$ is with respect to the Wasserstein distance. The integral above can be seen as the inner product
\[
\innprod{\mathbf t_\mu^{\mu^0}-\mathbf i}{\mathbf t_\mu^\nu-\mathbf i}
\]
in the space $L^2(\mu)$ that includes as a (closed) subspace the tangent space $\Tan_{\mu}$. In terms of this inner product and the log map, we can write
\[
F_0(\nu) - F_0(\mu)
=-\innprod{\log_{\mu}(\mu^0)}{\log_{\mu}(\nu)}
+o(d(\nu,\mu)),
\qquad \nu\to\mu,
\]
so that $F_0$ is Fr\'echet-differentiable at $\mu$ with derivative
\[
F_0'(\mu)
=- \log_{\mu}(\mu^0)
=-\left(\mathbf t_\mu^{\mu^0} - \mathbf i\right)
\in \Tan_{\mu}.
\]

\noindent We have proven:

\begin{theorem}[Gradient of the Fr\'echet Functional]\label{gradient_definition}
Fix a collection of measures $\mu^1,\dots,\mu^N\in\mathcal P_2(\mathbb{R}^d)$. When $\gamma$ is regular, the Fr\'{e}chet functional
\begin{equation}\label{eq:frechetfun}
F(\gamma)
=\frac1{2N} \sum_{i=1}^Nd^2(\gamma,\mu^i),
\qquad \gamma\in \mathcal P_2(\mathbb{R}^d).
\end{equation}
is Fr\'echet-differentiable, and its gradient satisfies
\begin{equation}\label{eq:frechetfungradientequation}
F'(\gamma)
=- \frac1N \sum_{i=1}^N \log_{\gamma}(\mu^i)
=- \frac1N \sum_{i=1}^N \left(\mathbf t_{\gamma}^{\mu^i}-\mathbf i\right).
\end{equation}
\end{theorem}

\subsubsection{Karcher and Fr\'echet Means}\label{characterisation_mean}

We can now characterise Karcher means, and also show that the empirical Fr\'echet mean must be sought amongst them, by an immediate corollary to Theorem \ref{gradient_definition}:

\begin{corollary}\label{thm:minisstationary}
Let $\mu^1,\dots,\mu^N\in\mathcal P_2(\mathbb R^d)$ be regular measures, one of which with bounded density.  A measure $\mu$ is a Karcher mean of $\{\mu^i\}$ if and only if
$$ \frac1N \sum_{i=1}^N \left(\mathbf t_{\mu}^{\mu^i}-\mathbf i\right) =0,\qquad\mu-\mbox{almost everywhere}.$$
Furthermore, the Fr\'echet mean of $\{\mu^i\}$ is itself a Karcher mean.
\end{corollary}
In fact, the corollary suggests that a Karcher mean is ``almost" a Fr\'echet mean: Agueh and Carlier \cite{bary} show by convex optimisation methods that if $\sum_{i=1}^N \left(\mathbf t_{\mu}^{\mu^i}-\mathbf i\right) =0$ \emph{everywhere} on $\mathbb R^d$ (rather than just $\mu$-almost everywhere), then $\mu$ is in fact the unique Fr\'echet mean. Thus one hopes that this ``gap of measure zero" can be bridged: that a sufficiently regular Karcher mean should in fact be a Fr\'echet mean. We now show that this is indeed the case; if $\mu^1,\dots,\mu^N\in\mathcal P_2(\mathbb R^d)$ are smooth measures with convex support, then a smooth Karcher mean of same support \emph{must be the unique Fr\'echet mean}:

\begin{theorem}[Optimality Criterion for Karcher Means]\label{thm:optimalKarcher}
Let $\mu^i$ for $i=1,\dots,N$ be probability measures on an open convex $X\subseteq \mathbb R^d$ whose densities $g^i$ are bounded and strictly positive  on $X$ and let $\mu$ be a regular Karcher mean of $\{\mu^i\}$ with density $f$.  Then $\mu$ is the unique Fr\'echet mean of $\{\mu^i\}$, provided one of the following holds:

\begin{enumerate}
\item $X=\mathbb R^d$, $f$ is bounded and strictly positive, and the densities $f,g^1,\dots,g^N$ are of class $C^1$;
\item $X$ is bounded, $\mu(X)=1$, $f$ is bounded, and the densities $f,g^1,\dots,g^N$ are bounded from below on $X$.
\end{enumerate}

\end{theorem}

\begin{remark}
In the first condition, the $C^1$ assumption can be weakened to H\"older continuity of the densities for some exponent $\alpha\in(0,1]$.
\end{remark}

\begin{remark}{We conjecture that a stronger result should be valid: specifically, if $\mu^1,\dots,\mu^N$ satisfy the conditions of Theorem~\ref{thm:optimalKarcher}, then we conjecture the Fr\'echet functional $F$ to in fact have a unique Karcher mean, coinciding with the Fr\'echet mean.} \end{remark}

\subsection{Gradient Descent and Procrustes Analysis}\label{algorithm_section}

\subsubsection{Elements of the Algorithm}\label{procrustes_analogy}

Let $\mu^1,\dots,\mu^N\in \mathcal P_2(\mathbb R^d)$ be regular and let $\gamma_j\in \mathcal P_2(\mathbb R^d)$ be a regular measure, representing our current estimate of the Fr\'echet mean of $\mu^1,\dots,\mu^N$ at step $j$.  Following the discussion above, it makes sense to introduce a step size $\tau_j>0$, and to carry out a steepest descent in the space of measures (e.g. Molchanov \& Zuyev \cite{molchanov2002steepest}),  following the negative of the gradient:
\[
\gamma_{j+1}
=\exp_{\gamma_j} \left(-\tau_j F'(\gamma_j)\right)
= \left[\mathbf i   +   \tau_j\frac1N \sum_{i=1}^N \log_{\gamma}(\mu^i)\right]\#\gamma_j=\left[\mathbf i   +   \tau_j\frac1N\sum_{i=1}^N  (\mathbf t_{\gamma_j}^{\mu^i} - \mathbf i)\right]\#\gamma_j.
\]
In order to guarantee that the descent is well-defined, we must make sure that the gradient itself will remain well-defined as we iterate over $j$. In view of Theorem~\ref{gradient_definition}, this requires showing that $\gamma_{j+1}$ remains regular whenever $\gamma_j$ is regular. This is indeed the case, at least if the step size is contained in $[0,1]$:

\begin{lemma}[Regularity of the iterates]\label{iterate_regularity}
If $\gamma_0$ is regular and $\tau_0\in[0,1]$ then so is $\gamma_{1}$.
\end{lemma}

Lemma \ref{iterate_regularity} suggests that the step size must be restricted to $[0,1]$. The next result suggests that the objective function essentially tells us that the \emph{optimal} step size, achieving the maximal reduction of the objective function (thus corresponding to an approximate line search), is exactly equal to 1:

\begin{lemma}[Optimal Stepsize]\label{lem:iterationbound}
If $\gamma_0\in \mathcal P_2(\mathbb R^d)$ is regular then
\begin{equation*}
F(\gamma_{1})-F(\gamma_0)
\le
- \|F'(\gamma_0)\|^2\left[\tau  -  \frac{\tau^2}2 \right].
\end{equation*}
and the bound on the right-hand side of the last display is minimised when $\tau=1$.
\end{lemma}

In light of the results in Lemmas \ref{iterate_regularity} and \ref{lem:iterationbound}, one needs only concentrate on the case $\tau_j=1$. This has an interesting ramification: when $\tau=1$, the gradient descent iteration is structurally equivalent to a Procrustes analysis. Specifically, the gradient descent algorithm proceeds by iterating the two steps of a Procrustes analysis (Gower \cite{gower1975generalized}; Dryden \& Mardia \cite[p. 90]{dryden1998statistical}):
\begin{enumerate}
\item[(1)] \textbf{Registration}: Each of the measures $\{\mu^1,\ldots,\mu^N\}$ is \emph{registered} to the current template $\gamma_j$, via the optimal transportation (registration) maps $\mathbf{t}_{\gamma_j}^{\mu^i}$. In geometrical terms, the measures $\{\mu^1,\ldots,\mu^N\}$ are lifted to the tangent space at $\gamma_j$ (via the log map), and their linear representation on the tangent space is expressed in local coordinates which coincide with the maps $\mathbf t_{\gamma_j}^{\mu^i} - \mathbf i=\log_{\gamma_j}(\mu^i)$. These can be seen as a common coordinate system for $\{\mu^1,\ldots,\mu^N\}$, i.e.\ a registration.

\item[(2)] \textbf{Averaging}: The registered measures are \emph{averaged coordinate-wise}, using the common coordinates system by the registration step (1). In geometrical terms, the linear representation of $\{\mu^1,\ldots,\mu^N\}$ afforded by their local coordinates $\mathbf t_{\gamma_j}^{\mu^i} - \mathbf i=\log_{\gamma_j}(\mu^i)$ is averaged linearly.  The linear average is then retracted back onto the manifold via the exponential map to yield the estimate at the $(j+1)$-step.

\end{enumerate}

That the gradient descent reduces to Procrustes analysis is not simply of aesthetic value. It is of the essence, as it shows that the algorithm relies entirely on solving a succession of \emph{pairwise optimal transportation problems}, thus reducing the determination of the Fr\'echet mean to the classical Monge problem of optimal transportation (e.g.\ Benamou and Brenier \cite{benamou2000computational}, Haber et al.\ \cite{haber2010efficient}, Chartrand et al.\ \cite{chartrand2009gradient}). After all, this is precisely the point of a Procrustes algorithm: exploiting the (easier) problem of pairwise registration to solve the (harder problem) of multi-registration.  We note that, further to requiring the ability to solve the pairwise optimal transportation problem, and the regularity conditions on the measures, the algorithm does not require additional structural assumptions/workarounds to reduce the problem to the one-dimensional case (as in, for example the ``admissibility" approach of Boissard et al.\ \cite{boissard2015distribution}). An additional practical advantage is that Procrustes algorithms are easily parallelisable, since one can distribute the solution of the pairwise transport problems at each step $j$.  Any regular measure can serve as an initial point for the algorithm, for instance one of the $\mu^i$. We should mention at this point that, if one is content with obtaining an \emph{approximate} or \emph{regularised} Fr\'echet mean, then there are several numerical strategies available, and there is a rapidly growing literature for the efficient computation of such schemes  -- we briefly summarise some such approaches in the concluding remarks section (Section \ref{sec:conclusions}).

\smallskip
\noindent The gradient/Procrustes iteration is presented succinctly as Algorithm~\ref{algo}. 

\begin{algorithm}[t]
\caption{Gradient Descent via Procrustes Analysis}
\label{algo}
\begin{enumerate}
\item[(A)] Set a tolerance threshold $\epsilon>0$.

\item[(B)] For $j=0$, let $\gamma_j$ be an arbitrary regular measure.

\item[(C)]  For $i=1,\ldots,N$ solve the (pairwise) Monge problem and find the optimal transport map $\mathbf t_{\gamma_j}^{\mu^i}$ from $\gamma_j$ to $\mu^i$.

\item[(D)]  Define the map $T_{j}=N^{-1}\sum_{i=1}^N  \mathbf t_{\gamma_j}^{\mu^i}$.

\item[(E)] Set $\gamma_{j+1}=T_{j}\#\gamma_j$, i.e.\ push-forward $\gamma_j$ via $T_{j}$ to obtain $\gamma_{j+1}$.

\item[(F)]  If $\|F'(\gamma_{j+1})\|<\epsilon$, stop, and output $\gamma_{j+1}$ as the approximation of $\bar{\mu}$ and $\mathbf{t}_{\gamma_{j+1}}^{\mu^i}$ as the approximation of $\mathbf{t}_{\bar{\mu}}^{\mu^i}$,  $i=1,\dots,N$.  Otherwise, return to step (C).

\end{enumerate}
\end{algorithm}

\subsubsection{Convergence of the Algorithm}\label{convergence_section}
In order to tackle the issue of convergence, we will use an approach that is specific to the nature of optimal transportation. The reason is that Hessian type arguments that are used to prove similar convergence results for gradient descent on Riemmanian manifolds (Afsari et al.\ \cite{afsari2013convergence}) or Procrustes algorithms (Le \cite{le2001locating}, Groisser \cite{groisser2005convergence}) do not apply here, since the Fr\'echet functional may very well fail to be twice differentiable. Still, this specific geometry of Wasserstein space affords some advantages; for instance, we will place no restriction on the starting point for the iteration, except that it be regular:

\begin{theorem}[Limit Points are Karcher Means]\label{thm:convtosta}
Let $\mu^1,\dots,\mu^N\in \mathcal P_2(\mathbb R^d)$ be absolutely continuous probability measures, one of which with bounded density.  Then, the sequence generated by Algorithm \ref{algo} stays in a compact set of the Wasserstein space $\mathcal P_2(\mathbb R^d)$, and any limit point of the sequence is a Karcher mean of $(\mu^1,\dots,\mu^N)$.
\end{theorem}

In view of Corollary~\ref{thm:minisstationary}, this immediately implies:
\begin{corollary}[Wasserstein Convergence of Gradient Descent]\label{cor:conv}
Under the conditions of Theorem~\ref{thm:convtosta}, if $F$ has a unique stationary point, then the sequence $\{\gamma_j\}$ generated by Algorithm \ref{algo} converges to the Fr\'echet mean of $\{\mu^1,\dots,\mu^N\}$ in the Wasserstein metric,
$$d(\gamma_j,\bar{\mu})\stackrel{j\rightarrow\infty}{\longrightarrow}0.$$
\end{corollary}

{Of course, combining Theorem~\ref{thm:convtosta} with Theorem~\ref{thm:optimalKarcher} shows that the conclusion of Corollary~\ref{cor:conv} holds when the appropriate assumptions on $\{\mu^i\}$ and the Karcher mean $\mu$ are satisfied.} The proof of Theorem \ref{thm:convtosta} is elaborate, and is constructed via a series of intermediate results in a separate section (Section \ref{proof_algorithm}) in the interest of tidiness.  The main challenge is that the standard condition used for convergence of gradient descent algorithms, that gradients be Lipschitz, fails to hold in this setup.  Indeed, $F$ is not differentiable on discrete measures, and these constitute a dense subset of the Wasserstein space.

\subsubsection{Uniform Convergence of Procrustes Maps and Multicoupling}\label{convergence_procrustes}
We conclude our analysis of the algorithm by turning to the Procrustes maps $\mathbf{t}^{\bar{\mu}}_{\mu^i}$, which optimally couple each sample observation $\mu^i$ to their Fr\'echet mean $\bar{\mu}$. These are the key objects required for the solution of the multicoupling problem (as established in Proposition \ref{multicoupling}), and one would use the limit of $\mathbf{t}^{\gamma_j}_{\mu^i}$ in $j$ as their approximation. However,  the fact that $d(\gamma_j,\bar{\mu}){\rightarrow}0$ does not immediately imply the convergence of $\mathbf{t}^{\gamma_j}_{\mu^i}$ to $\mathbf{t}^{\bar{\mu}}_{\mu^i}$:  the Wasserstein convergence only means that certain integrals of the warp maps converge.  Still, convergence of the warp maps \emph{does} hold, indeed uniformly so on compacta, $\bar{\mu}$-almost everywhere:

\begin{theorem}[Uniform Convergence of Procrustes Maps] \label{thm:convprocrustes}
Under the conditions of Corollary~\ref{cor:conv}, there exist sets $A,B^1,\dots,B^N\subseteq \mathbb R^d$ such that $\bar{\mu}(A)=1=\mu^1(B^1)=\dots=\mu^N(B^N)$ and
\[
\sup_{\Omega_1}\,\left\|\mathbf{t}_{\gamma_j}^{\mu^i} - \mathbf{t}_{\bar{\mu}}^{\mu^i}\right\|\stackrel{j\rightarrow\infty}{\longrightarrow}0,
\qquad
\sup_{\Omega_2}\,\left\|\mathbf{t}^{\gamma_j}_{\mu^i} - \mathbf{t}^{\bar{\mu}}_{\mu^i}\right\|\stackrel{j\rightarrow\infty}{\longrightarrow}0,
\qquad i=1,\dots,N,
\]
for any pair of compacta $\Omega_1\subseteq A$,  $\Omega_2\subseteq B^i$,  where the sequence $\mathbf{t}^{\gamma_j}_{\mu^i}$ and $\mathbf{t}_{\gamma_j}^{\mu^i}=\left(\mathbf{t}^{\gamma_j}_{\mu^i}\right)^{-1}$ are the Procrustes maps generated by Algorithm \ref{algo}.  If in addition all the measures $\{\mu^1,\dots,\mu^N\}$ have the same support, then one can choose the sets so that $B^1=\dots=B^N$.
\end{theorem}
\noindent {With both ingredients of the registration problem in hand, we deduce a solution to the latter:
\begin{corollary}
[Convergence of Multicouplings]
\label{cor:convMulti}
Under the conditions of Corollary~\ref{cor:conv}, the sequence of multicouplings
\[
\left(\mathbf t_{\gamma_j}^{\mu^1},\dots\mathbf t_{\gamma_j}^{\mu^n}\right)\#\gamma_j
\]
of $\{\mu^1,\dots,\mu^N\}$ converges (in Wasserstein distance on $(\R^d)^N$) to the optimal multicoupling $(\mathbf t_{\bar\mu}^{\mu^1},\dots\mathbf t_{\bar\mu}^{\mu^n})\#\bar{\mu}$.
\end{corollary}
}

\section{Population Setting}\label{sec:population}
In order to carry out \emph{inference}, we must relate the sample collection of measures to a population, and show that the relevant quantities are identifiable parameters. Furthermore, in practice the sample measures will only be discretely observed, and this must be taken into account. We now formulate such a model, and study its nonparametric estimation from discrete observations.

\subsection{Deformation Models and Discrete Observation}\label{sec:discrete_intro}
Let $\lambda$ be a regular probability measure with a strictly positive density on a convex compact $K\subset\mathbb R^d$ of positive Lebesgue measure\footnote{In applied settings, the point processes will be observed on a bounded \textit{observation window} $K$.  For this reason as well as the sake of simplicity, we restrict our discussion to a given compact set (but remark that it could be extended to unbounded observation windows subject to further conditions). }, and let $\{\Pi_1,\ldots,\Pi_N\}$ be i.i.d point processes with intensity measure $\lambda$,
\[
\mathbb{E}[\Pi_i(A)]
=\lambda(A),
\]
for all Borel subsets $A\subseteq K$. Instead of observing the true processes $\{\Pi_1,\ldots,\Pi_N\}$, we are able to observe \emph{warped} versions
\[
\widetilde{\Pi_i}:
=T_i\#\Pi_i,
\qquad i=1,\dots,N,
\]
with conditional warped mean measures
\[
\mathbb{E}[\widetilde{\Pi_i}|T_i]
=\mathbb{E}[T_i\#\Pi_i|T_i]
=\Lambda_i
=T_i\#\lambda,
\]
where the $\{T_i:\mathbb{R}^d\rightarrow\mathbb{R}^d\}$ are i.i.d random homeomorphisms on $K$, satisfying the properties of
\begin{enumerate}
\item Unbiasedness:  the Fr\'echet mean of $\Lambda_i=T_i\#\lambda$ is $\lambda$.

\item Regularity:  $T_i$ is a gradient of a convex function on $K$.
\end{enumerate}
The conditional mean measures $\{\Lambda_i=T_i\#\lambda\}_{i=1}^{N}$ play the role of the unobservable sample of random measures generated from a population law constructed via random deformations of the template $\lambda$. The processes $\{\widetilde{\Pi}_i\}_{i=1}^{N}$ play the role of the discretely observed versions of the $\{\Lambda_i\}_{i=1}^{N}$. Conditions (1) and (2) state that the deformations $\{T_i\}$ are identifiable. They can also be motivated from first principles: (1) states that the maps do not deform the template $\lambda$ on average (otherwise this ``average deformation" would be by definition the template); and (2) states that among all possible deformations that could have mapped $\lambda$ to $\Lambda_i$, we take the parsimonious choice of the optimal deformation. The importance and canonicity of these two assumptions has been discussed in depth in Panaretos \& Zemel \cite[Section 3.3]{panzem15}, who study a one-dimensional version of the above problems (which is qualitatively very different, given the flat nature of 1d Wasserstein space, and the availability of explicit closed form expressions).

The connection of this deformation model to Fr\'echet means, via the optimal maps, is now given as follows (in a general setup, encompassing our model setup).  Let $C_b(K,\mathbb R^d)$ be the space of continuous bounded functions $f:K\to\mathbb R^d$ endowed with the supremum norm $\|f\|_\infty=\sup_{x\in K}\|f(x)\|$.

\begin{theorem}[mean identity warp functions and Fr\'echet means]\label{thm:meanTid}
Let $K\subset\R^d$ be a compact convex set of positive Lebesgue measure, and let $\lambda\in \mathcal{P}_2(K)$ be regular.  Consider the random measure $\Lambda=T\#\lambda$, where $T:K\to K$ is a random deformation (viewed as a random element in $C_b(K,\mathbb R^d)$), almost surely injective, and satisfying
\begin{enumerate}
\item almost surely there exists a convex function $\phi$ such that $T=\nabla\phi$ on the interior of $K$;
\item $\E [T(x)]=x$ for all $x\in K$ (or on a dense subset of $K$);
\item almost surely $T$ is differentiable with a nonsingular derivative for almost all $x\in K$.
\end{enumerate}
Then $\lambda$ is the unique Fr\'echet mean of $\Lambda$, i.e., the unique minimiser of the population Fr\'echet functional $\gamma\mapsto \mathbb Ed^2(\Lambda,\gamma)$.
\end{theorem}
An important requirement for the statement and proof of Theorem~\ref{thm:meanTid} is that $\phi$, $\phi^*$ and $\Lambda$ are measurable as random elements in the appropriate spaces; this is not a priori obvious, but is established as part of the proof.

The statistical problem will now be to estimate the unknown structural mean measure $\lambda$, and the registration maps $T_i$ non-parametrically,  by smoothing the observed point processes $\{\widetilde{\Pi}_1,\ldots,\widetilde{\Pi}_N\}$. Once $\lambda$ and $\{T_i\}$  have been estimated, the processes $\{\widetilde{\Pi}_1,\ldots,\widetilde{\Pi}_N\}$ can be registered by applying the inverses of the estimated maps $T_i$, allowing for further analysis of the point processes in a functional data context.  Theorem~\ref{thm:meanTid} guarantees that the estimands considered are identifiable.

\subsection{Regularised Nonparametric Estimation}

In order to estimate the $\lambda$ and the $\{\Lambda_i,T_i\}$, we will follow the steps below:

\begin{enumerate}
\item \emph{Regularisation:} Estimate $\Lambda_i=T_i\# \lambda$ by a regular kernel estimator $\widehat{\Lambda}_i$ restricted on $K$,
\begin{equation}\label{eq:lambdahat}
\widehat{\Lambda}_i
=\frac1m  \sum_{j=1}^m\frac{\delta\{x_j\}\ast \psi_\sigma}  {[\delta\{x_j\}\ast \psi_\sigma](K)}\bigg|_K,
\end{equation}
where $\psi:\mathbb R^d\to(0,\infty)$ is a unit-variance isotropic density function, $\psi_\sigma(x)=\sigma^{-d}\psi(x/\sigma)$ for $\sigma>0$ (more generally, $\psi$ could be non-isotropic, having a bandwidth matrix, but we focus on the isotropic case for simplicity), and $\widetilde\Pi_i$ is the sum of dirac masses $\sum_{j=1}^m\delta\{x_i\}$.  If $\widetilde{\Pi}_i$ contains no points (that is, $m=0$), define $\hat\Lambda_i$ to be the (normalised) Lebesgue measure on $K$.

\item \emph{Fr\'echet Mean Estimation:} Estimate $\lambda$ by the empirical Fr\'echet mean $\hat{\lambda}$ of $\widehat{\Lambda}_1,\dots,\widehat{\Lambda}_N$, using the Procrustes Algorithm \ref{algo}.

\item \emph{Procrustes Analysis:} Estimate $T_i$ by the optimal transportation map of $\widehat{\lambda}$ onto $\widehat{\Lambda}_i$, as given by the final step in the iteration of Algorithm \ref{algo}. Estimate the map $T_i^{-1}$ by $\widehat{T^{-1}_i}=\widehat{T}^{-1}_i$.

\item \emph{Registration:} Register the observed point processes to a common coordinate system by defining $\widehat{\Pi}_i=\widehat{{T^{-1}_i}}{\#}\widetilde{\Pi}_i$.
\end{enumerate}

 In the next section, we will prove that our estimates are consistent for their population version, as the number of observed processes, and the number of points per process diverge.

\subsection{Asymptotic Theory}
To establish consistency, we will use the \emph{dense} asymptotics regime of functional data analysis, adapted to the current setting.  We will consider a setup where the number of observed point processes $n$ diverges, and the (mean) number of points in each observed process, $\mathbb{E}[\widetilde{\Pi}_i(K)]$, diverge too.  Here we use the index notation ``$n$'' rather than ``$N$" to emphasize that the index is no longer held fixed.  Specifically, let $(\Pi_1^{(n)},\Pi_2^{(n)},\dots,\Pi_n^{(n)})_{n=1}^\infty$ be a triangular array of row-independent and identically distributed point processes on $K$ following the same infinitely divisible distribution and having mean measure $\tau_n\lambda$, where $\tau_n>0$ are constants. Let $T_1,\dots,T_n$ be independent and identically distributed realisations of a random homeomorphism $T$ of $K$ satisfying the unbiasedness and regularity assumptions of Section \ref{sec:discrete_intro}. Let $\widetilde\Pi_i^{(n)}=T_i\#\Pi_i^{(n)}$ and set $\Lambda_i=T_i\#\lambda=\tau_n^{-1}\mathbb E[\widetilde\Pi_i^{(n)}|T_i]$. Suppose that $\widehat\Lambda_i$ is an estimator of $\Lambda_i$, constructed by kernel smoothing of $\Pi_i^{(n)}$ using a (possibly random) bandwidth $\sigma_i^{(n)}$, as described in the previous section. Correspondingly, let $\widetilde\Pi_i^{(N)}=T_i\#\Pi_i^{(n)}$ and set $\Lambda_i=T_i\#\lambda=\tau_n^{-1}\mathbb E[\widetilde\Pi_i^{(n)}|T_i]$.

\begin{theorem}[Consistency  of the regularised Fr\'echet Mean]\label{thm:consistency}
If $\tau_n/\log n\to\infty$ and $\sigma_n=\max_i\sigma^{(n)}_i\stackrel{\rm{p}}\to0$ then
\begin{enumerate}
\item For any $i$,
\[
d(\widehat\Lambda_i,\Lambda_i)\stackrel{\rm{p}}\to0;
\]
\item The estimator $\widehat\lambda_n$ is strongly consistent
\[
d(\widehat\lambda_n,\lambda)\stackrel{\rm{as}}\to0.
\]
\end{enumerate}
If the smoothing is carried out independently across trains, that is, $\sigma_i^{(n)}$ depends only on $\widetilde\Pi_i^{(n)}$, then the result still holds if merely $\tau_n\to\infty$.

If $\mathbb E\left[\Pi_1^{(1)}\right]^4<\infty$, $\sum_n\tau_n^{-2}<\infty$ and $\sigma_n\stackrel{\rm{as}}\to0$ then convergence almost surely holds.
\end{theorem}

\begin{remark}
There is no lower bound on $\sigma_n$, and it can vanish at any rate, provided it is strictly positive.  In practice, however, if $\sigma_n$ is very small, then the densities of $\widehat\Lambda_i$ will have very high peaks, and the constant $C_\mu$ in Proposition~\ref{prop:Cmu} (with $\mu^i=\widehat\Lambda_i$) will be large (essentially proportional to $1/\sigma_n$).  The proof of Proposition~\ref{prop:Fdercont} suggests that this may slow down the convergence of Algorithm~\ref{algo}.
\end{remark}

\begin{remark}
It is worth remarking that Le Gouic \& Loubes \cite[Theorem~3]{legouic2016existence} consider the stability of Fr\'echet means in a rather general setting; verification of their assumptions in our particular setting, however, is quite involved and in fact essentially amounts to directly proving Theorem \ref{thm:consistency}.
\end{remark}

Our next two results concern the (uniform) consistency of the Procrustes registration procedure. Though the results themselves parallel their one-dimensional counterparts (see Panaretos \& Zemel \cite{panzem15}), their proofs are entirely different, and substantially more involved (because the geometry of monotone mappings in $\mathbb{R}^d$ is far more rich than the geometry of monotone maps on $\mathbb{R}$). In particular, we have:

\begin{theorem}[Consistency of Procrustes Maps]\label{thm:consistencyWarp}
Under the same conditions of Theorem~\ref{thm:consistency},  for any $i$ and any compact set $\Omega\subseteq\mathrm{int}(K)$,
\[
\sup_{x\in \Omega}\|\widehat T_i^{-1}(x) - T_i^{-1}(x)\|\stackrel{\rm{p}}\to0,
\qquad
\sup_{x\in \Omega}\|\widehat T_i(x) - T_i(x)\|\stackrel{\rm{p}}\to0.
\]
The same remarks at the end of the statement of Theorem~\ref{thm:consistency} apply here as well.
\end{theorem}

\begin{corollary}[Consistency of Procrustes Registration]\label{cor:consistencyPi}
Under the same conditions of Theorem~\ref{thm:consistency}, the registration procedure is consistent:  for any $i$
\[
d\left( \frac{\widehat \Pi_i} {\widehat\Pi_i(K)}, \frac{\Pi_i} {\Pi_i(K)}    \right)\stackrel{\rm{p}}\to0
,\qquad n\to\infty,
\]
provided one of the following conditions holds:
\begin{enumerate}
\item Every point of the boundary of $K$ is \textit{exposed}, that is, for any $y\in\partial K$ there exists $\alpha\in\mathbb R^d$ such that
\[
\innprod y\alpha
>\innprod {y'}\alpha
,\qquad y'\in K\setminus\{y\}.
\]

\item The warp map $T_i$ is \emph{strictly monotone}
\[
\innprod {T_i(x') - T_i(x)}  {x'  - x}>0,
\qquad x,x'\in\mathrm{int}(K)
,\quad x\ne x'.
\]
\end{enumerate}
\end{corollary}
The first condition is satisfied by any ellipsoid in $\mathbb R^d$ and more generally if the boundary of $K$ can be written as
$
\partial K
=\{x : \varphi_K(x) = 0\},
$
for a strictly convex function $\varphi_K$.  Indeed, if $\alpha$ creates a supporting hyperplane to $K$ at $y$ and $\innprod\alpha y=\innprod \alpha {y'}$ for $y\ne y'$, then as $\varphi_K$ is strictly convex on the line segment $[y,y']$, it is impossible that $y'\in K$ without the hyperplane intersecting the interior of $K$.  Although this condition excludes some interesting cases, perhaps most prominently polyhedral sets such as $K=[0,1]^d$, such sets can be approximated by convex sets that do satisfy it (Krantz \cite[Proposition~1.12]{krantz2014convex}).

As for the second condition, in general it will hold almost surely. Indeed, as $T_i\#\lambda=\Lambda_i$ and both measures are absolutely continuous,  there exists a $\lambda$-null set $\mathcal N$ such that $T_i$ is strictly monotone outside $\mathcal N$ \cite[Proposition 6.2.12]{ambgigsav}.  By assumption $\lambda$ has a strictly positive density on $K$, so that $\lambda$-null subsets of $K$ are precisely the Lebesgue null subsets of $K$.  In that sense, this condition is not overly restrictive, and will most likely be satisfied under additional regularity assumptions on the warp maps $T_i$ and, possibly, $K$.

\section{Proofs of Formal Statements}\label{proofs}

Our proofs will require us to establish some analytical results that are intrinsic to the optimal transportation problem. These are essential for the proofs, especially of our main results, and some are non-trivial. For tidiness, we will state and prove these results separately at the end of this section (Section \ref{appendix}), developing our main results first, and referring to the analytical background when necessary.

\subsection{Proofs of Statements in Section \ref{population_to_sample}}

\begin{proof}[Proof of Proposition \ref{multicoupling}]
The optimisation problem
\[
\min_{Y_i\sim \mu^i}
\mathbb E  \sum_{i=1}^N\sum_{j=i+1}^N
\|Y_i - Y_j\|^2
=\min_{\xi\in \Gamma(\mu^1,\dots,\mu^N)}
\ownint {\mathbb R^{Nd}}{}{\sum_{i=1}^N\sum_{j=i+1}^N\|t_i - t_j\|^2}{\xi(t_1,\dots,t_N)}
\]
is equivalent to minimising
\[
G(\xi)
=\frac1{2N}
\ownint {\mathbb R^{Nd}}{}{\sum_{i=1}^N\left\|t_i - \frac1N\sum_{j=1}^Nt_j\right\|^2}{\xi(t_1,\dots,t_N)},
\qquad \xi\in \Gamma(\mu^1,\dots,\mu^N),
\]
and Agueh \& Carlier \cite[Proposition~4.2]{bary} show that $\min_\mu F(\mu)=\min_{\xi} G(\xi)$.

Since $\bar\mu$ is regular \cite[Proposition~5.1]{bary}, $\bm X$ is well-defined and has joint distribution
\[
\xi' = h\# \bar\mu,
\quad h : \R^d \to \R^{Nd},
\quad h = \left(  \mathbf t_{\bar\mu}^{\mu^1}, \dots ,\mathbf t_{\bar\mu}^{\mu^N} \right).
\]
Since the coordinates of $h$ have mean identity (see \cite[Equation~(3.9)]{bary} or Corollary~\ref{thm:minisstationary}),
\[
G(\xi')
=\frac1{2N}
\ownint {\mathbb R^d}{}{\sum_{i=1}^N\|\mathbf t_{\bar\mu}^{\mu^i} - \mathbf i\|^2}{\bar\mu}
=\frac1{2N}\sum_{i=1}^N d^2(\bar\mu,\mu^i)
=F(\bar\mu)
=\inf_\mu F(\mu).
\]
Thus $\xi'$ is optimal.
\end{proof}

\subsection{Proofs of Statements in Section \ref{gradient_section}}

\begin{proof}[Proof of Corollary~\ref{thm:minisstationary}]
The characterisation of Karcher means is immediate from Theorem \ref{gradient_definition}. The fact that the Fr\'echet mean $\mu$ satisfies $\sum_{i=1}^N \left(\mathbf t_{\mu}^{\mu^i}-\mathbf i\right) =0$  $\mu$-almost everywhere follows by a result of Agueh \& Carlier \cite{bary}.  For an alternative proof using the tangent bundle, see the supplementary material (Section~\ref{zempan16supp}).
\end{proof}

\begin{proof}[Proof of Theorem~\ref{thm:optimalKarcher}]
The result exploits Caffarelli's regularity theory for Monge--Amp\`ere equations.  In the first case, by Theorem~4.14(iii) in Villani \cite{vil} there exist $C^1$ (in fact, $C^{2,\alpha}$) convex potentials $\varphi_i$ on $\mathbb R^d$ with $\mathbf t_\mu^{\mu^i}=\nabla\varphi_i$, so that $\mathbf t_\mu^{\mu^i}(x)$ is a singleton for all $x\in\mathbb R^d$.  The set $\{x\in\mathbb R^d:\sum \mathbf t_\mu^{\mu^i}(x)/N\ne x\}$ is $\mu$-negligible (and hence Lebesgue-negligible) and open by continuity.  It is therefore empty, so $F'(\mu)=0$ everywhere, and $\mu$ is the Fr\'echet mean (see the discussion after Corollary~\ref{thm:minisstationary}).

In the second case, by the main theorem in Caffarelli \cite[p.\ 99]{caffarelli1992regularity}, and the same argument, we have $\sum\mathbf t_\mu^{\mu^i}(x)/N=x$ for all $x\in X$.  Since $X$ is convex, there must exist a constant $C$ such that $\sum \varphi_i(x)=C+N\|x\|^2/2$ for all $x\in X$.  Hence Equation~(3.9) in \cite{bary} holds with $\mathbb R^d$ replaced by $X$.  Repeating the proof of Proposition~3.8 in \cite{bary}, we see that $\mu$ minimises $F$ on $\mathcal P_2(X)$, the set of measures supported on $X$. (All the integrals that appear in the proof can be taken on $X$, where we know the inequality holds).  Again by convexity of $X$, the minimiser of $F$ must be\footnote{We know that the minimiser must be in $\mathcal P_2(\overline X)$, but minimising on $\mathcal P_2(X)$ suffices by continuity of $F$.} in $\mathcal P_2(X)$ (see the existence proof at the beginning of the proof of Theorem~\ref{thm:meanTid} in the supplementary material, Section~\ref{zempan16supp}).
\end{proof}

\subsection{Proofs of Statements in Section \ref{algorithm_section}}

\begin{proof}[Proof of Lemma \ref{iterate_regularity}] By \cite[Proposition 6.2.12]{ambgigsav} there exists a $\gamma_0$-null set $A_i$ such that on $\mathbb R^d\setminus A_i$, $\mathbf t_{\gamma_0}^{\mu^i}$ is differentiable, $\nabla \mathbf t_{\gamma_0}^{\mu^i}>0$ (positive definite), and $\mathbf t_{\gamma_0}^{\mu^i}$ is strictly monotone
\[
\innprod{\mathbf t_{\gamma_0}^{\mu^i}(x) - \mathbf t_{\gamma_0}^{\mu^i}(x')}{x-x'}
>0,
\qquad x,x'\notin A_i,
\quad x\ne x'.
\]
Since $\mathbf t_{\gamma_0}^{\gamma_1}=(1-\tau)\mathbf i+\tau N^{-1}\sum_{i=1}^N  \mathbf t_{\gamma_0}^{\mu^i}$, it stays strictly monotone (hence injective) and $\nabla \mathbf t_{\gamma_0}^{\gamma_1}>0$ outside $A=\cup A_i$, which is a $\gamma_0$-null set.

Let $h_0$ denote the density of $\gamma_0$ and set $\Sigma=\mathbb R^d\setminus A$.  Then $\mathbf t_{\gamma_0}^{\gamma_1}|_\Sigma$ is injective and $\{h_0>0\}\setminus\Sigma$ is Lebesgue negligible because
\[
0 = \gamma_0(A)
= \gamma_0(\mathbb R^d\setminus \Sigma)
=\ownint{\mathbb R^d\setminus \Sigma} {}{h_0(x)} x
=\ownint{\{h_0>0\}\setminus \Sigma}{}{h_0(x)} x,
\]
and the integrand is strictly positive.  Since $|\mathrm{det}\nabla \mathbf t_{\gamma_0}^{\mu^i}|>0$ on $\Sigma$ we obtain that $\gamma_1=\mathbf t_{\gamma_0}^{\mu^i}\#\gamma_0$ is absolutely continuous by \cite[Lemma~5.5.3]{ambgigsav}.
\end{proof}

\begin{proof}[Proof of Lemma \ref{lem:iterationbound}]
Let $S_i=\mathbf t_{\gamma_0}^{\mu^i}$ be the optimal map from $\gamma_0$ to $\mu^i$, and set $W_i=S_i-\mathbf i$. Then
\begin{equation}\label{eq:costbyW}
2NF(\gamma_0)
=\sum_{i=1}^Nd^2(\gamma_0,\mu^i)
=\sum_{i=1}^N\ownint {\mathbb R^d}{}{\|S_i-\mathbf i\|^2}{\gamma_0}
=\sum_{i=1}^N\innprod{W_i}{W_i}
=\sum_{i=1}^N\|W_i\|^2,
\end{equation}
with the inner product being in $L^2(\gamma_0)$.  By definition
\[
\gamma_1
=\left[(1-\tau)\mathbf i+\frac\tau N\sum_{j=1}^NS_j\right]\#\gamma_0
=\left[(1-\tau)S_i^{-1}+\frac\tau N\sum_{j=1}^NS_j\circ S_i^{-1}\right]\#\mu^i.
\]
This is a map that pushes forward $\mu^i$ to $\gamma_1$ (not necessarily optimally). Hence
\[
d^2(\gamma_1,\mu^i)
\le \ownint {\mathbb R^d}{}{\left\|\left[(1-\tau)S^{-1}_ i+\frac\tau N\sum_{j=1}^NS_j\circ S_i^{-1}\right]-\mathbf i\right\|_{\R^d}^2}{\mu^i}.
\]
Now $\mu^i=S_i\#\gamma_0$, which means that $\ownint{}{} f{\mu^i}=\ownint{}{}{ (f\circ S_i)}{\gamma_0}$ for any measurable $f$.  This change of variables gives
\[
d^2(\gamma_1,\mu^i)
\le
\ownint {\mathbb R^d}{}{\left\|\left[(1-\tau)\mathbf i+\frac\tau N\sum_{j=1}^NS_j\right] - S_i\right\|_{\mathbb R^d}^2}{\gamma_0}
= \left\| -W_i + \frac \tau N\sum_{j=1}^NW_j\right\|^2_{L^2(\gamma_0)}.
\]
The norm is always in $L^2(\gamma_0)$, regardless of $i$. Developing the squares, summing over $i=1,\dots,N$ and using \eqref{eq:costbyW} gives
\begin{align*}
2NF(\gamma_1)
&\le \sum_{i=1}^N \|W_i\|^2
- 2\frac\tau N \sum_{i,j=1}^N\innprod{W_i}{W_j}
+\frac{\tau^2}{N^2} \sum_{i,j,k=1}^N  \innprod{W_j}{W_k}\\
& = 2NF(\gamma_0)
- 2N\tau  \left\|\sum_{i=1}^N\frac 1NW_i\right\|^2
+ N\tau^2   \left\|\sum_{i=1}^N\frac 1NW_i\right\|^2,
\end{align*}
and recalling that $W_i=S_i-\mathbf i$ yields
\[
F(\gamma_1) - F(\gamma_0)
\le \frac{\tau^2 - 2\tau}2\left\|\frac 1N\sum_{i=1}^NW_i\right\|^2
= -\|F'(\gamma_0)\|^2\left[\tau - \frac{\tau^2}2\right].
\]
Since $\tau  -  \tau^2/2$ is clearly maximised at $\tau=1$, the proof is complete.
\end{proof}

\subsubsection{Proof of Theorem \ref{thm:convtosta}}\label{proof_algorithm}

We will prove the theorem by establishing the following facts:
\begin{enumerate}

\item The sequence $\|F'(\gamma_j)\|$ converge to zero as $j\to\infty$.

\item The sequence $\{\gamma_j\}$ is stays in a compact subset of $\mathcal P_2(\mathbb R^d)$.

\item The mapping $\gamma\mapsto \|F'(\gamma)\|^2$ is continuous.

\end{enumerate}

The first two are relatively straightforward, and are proven in the form of the following two Lemmas.

\begin{lemma}
The objective value of the Fr\'echet functional decreases at each step of Algorithm \ref{algo}, and $\|F'(\gamma_j)\|$ vanishes as $j\to\infty$.
\end{lemma}
\begin{proof} The first statement is clear from Lemma~\ref{lem:iterationbound}, from which it also follows that
\[
\frac12\sum_{j=0}^k   \|F'(\gamma_j)\|^2
\le \sum_{j=0}^k  F(\gamma_{j})  - F(\gamma_{j+1})
= F(\gamma_0) - F(\gamma_{k+1})
\le F(\gamma_0).
\]
Consequently, the series at the left-hand side converges whence $\|F'(\gamma_j)\|^2\to 0$.
\end{proof}

\begin{lemma}\label{lem:seqcompact}
The sequence generated by Algorithm~\ref{algo} stays in a compact subset of the Wasserstein space $\mathcal P_2(\mathbb R^d)$.
\end{lemma}

\begin{proof}
For any $\epsilon>0$ there exists a compact convex set $K_\epsilon$ such that $\mu^i(K_\epsilon)>1-\epsilon/N$ for $i=1,\dots,N$.  Let $A^i=(\mathbf t_{\gamma_j}^{\mu^i})^{-1}(K_\epsilon)$, $A=\cap_{i=1}^NA^i$.  Then $\gamma_j(A^i)>1-\epsilon/N$, so that $\gamma_j(A)>1-\epsilon$.  Since $K_\epsilon$ is convex, $T_j(x)\in K_\epsilon$ for any $x\in A$, so that
\[
\gamma_{j+1}(K_\epsilon)
=\gamma_j(T_j^{-1}(K_\epsilon))
\ge \gamma_j(A)
>1-\epsilon
,\qquad
j=0,1,\dots.
\]
We shall now show that any weakly convergent subsequence of $\{\gamma_j\}$ is in fact convergent in the Wasserstein space.  By Theorem~7.12 in Villani \cite{vil}, it suffices to show that
\begin{equation}\label{eq:uniformboundedness}
\lim_{R\to\infty}
\sup_{j\in \mathbb N}\ownint{\{x:\|x\|>R\}}{}{\|x\|^2}{\gamma_j(x)}
=0.
\end{equation}
For simplicity, we shall show this under the stronger assumption that the measures $\mu^1,\dots,\mu^N$ have a finite third moment
\begin{equation}\label{eq:3mom}
\ownint {\R^d}{}{\|x\|^3}{\mu^i(x)}
\le M(3)
,\qquad i=1,\dots,N.
\end{equation}
In Section~\ref{zempan16supp} we show that \eqref{eq:uniformboundedness} holds even if \eqref{eq:3mom} does not.

For any $j\ge1$ it holds that
\begin{align*}
\ownint {\R^d}{}{\|x\|^3}{\gamma_j(x)}
=\ownint {\R^d}{}{\left\|\frac1N \sum_{i=1}^N \mathbf t_{\gamma_{j-1}}^{\mu^i} (x)\right\|^3}{\gamma_{j-1}(x)}
&\le \frac1N\sum_{i=1}^N \ownint {\R^d}{}{\|\mathbf t_{\gamma_{j-1}}^{\mu^i}(x)\|^3}{\gamma_{j-1}(x)}\\
&=\frac1N\sum_{i=1}^N \ownint {\R^d}{}{\|x\|^3}{\mu^i(x)}
\le M(3).
\end{align*}
This implies that for any $R>0$ and any $j>0$,
\[
\ownint{\{x:\|x\|>R\}}{}{\|x\|^2}{\gamma_j(x)}
\le\frac1R\ownint{\{x:\|x\|>R\}}{}{\|x\|^3}{\gamma_j(x)}
\le\frac1RM(3),
\]
and \eqref{eq:uniformboundedness} follows.
\end{proof}
\noindent The third statement (continuity of the gradient) is much more subtle to establish. We will prove it in two steps: first we establish a Proposition, giving sufficient conditions for the third statement to hold true. Then, we will verify that the conditions of the Proposition are satisfied in the setting of Theorem \ref{algorithm_section}, in the form of a Lemma and a Corollary. We start with the proposition.

\begin{proposition}[Continuity of $F'$]\label{prop:Fdercont}
Let $\mu^1,\dots,\mu^N\in \mathcal P_2(\R^d)$ be given regular measures, and consider a sequence $\gamma_n$ of regular measures that converges in $\mathcal P_2(\mathbb R^d)$ to a regular measure $\gamma$.  If the densities of $\gamma_n$ are uniformly bounded, then $\|F'(\gamma_n)\|^2\to \|F'(\gamma)\|^2$.
\end{proposition}

\begin{proof}  The regularity of $\gamma_n$ and $\gamma$ implies that $F$ is indeed differentiable there, and so it needs to be shown that
\[
\left\| \frac1N\sum_{i=1}^N \mathbf t_{\gamma_n}^{\mu^i} - \mathbf i\right \|^2_{L^2(\gamma_n)}
\longrightarrow\left\| \frac1N\sum_{i=1}^N \mathbf t_{\gamma}^{\mu^i} - \mathbf i\right \|^2_{L^2(\gamma)},
\qquad n\to\infty.
\]
Denote the integrands by $g_n$ and $g$ respectively.  At a given $x\in\mathbb R^d$, $g_n(x)$ can be undefined, either because some $\mathbf t_{\gamma_n}^{\mu^i}(x)$ is empty, or because they can be multivalued.  Redefine $g_n(x)$ at such points by setting it to 0 in the former case and choosing an arbitrary representative otherwise.  Since the set of these ambiguity points is a $\gamma_n$-null set (because $\gamma_n$ is absolutely continuous), this modification does not affect the value of the integral $\ownint{}{}{g_n}{\gamma_n}$.  Apply the same procedure to $g$.  Then $g_n$ and $g$ are finite and nonnegative throughout $\mathbb R^d$.  Absolute continuity of $\gamma$, Remark~2.3 in \cite{albamb} and Proposition~\ref{prop:cont} imply together that the set of points where $g$ is not continuous is a $\gamma$-null set.

Next, we approximate $g_n$ and $g$ by bounded functions as follows.  Since $\gamma_n$ converge in the Wasserstein space, they satisfy \eqref{eq:uniformboundedness} by \cite[Theorem~7.12]{vil}.  It is easy to see that this implies the uniform absolute continuity
\begin{equation}\label{eq:unifabscont}
\forall\epsilon>0
\exists\delta>0
\forall j\ge1
\forall A\subseteq\mathbb R^d
\textrm{ Borel}:
\quad \gamma_j(A)\le \delta
\quad \Longrightarrow
\ownint A{}{\|x\|^2}{\gamma_j(x)}
<\epsilon.
\end{equation}
The $\delta$'s can be chosen in such a way that \eqref{eq:unifabscont} holds true for the finite collection $\{\mu^1,\dots,\mu^N\}$ as well.  Fix $\epsilon>0$, set $\delta=\delta_\epsilon$ as in \eqref{eq:unifabscont}, and let $A_n=\{x:g_n(x)\ge 4R\}$, where $R=R_\epsilon\ge1$ is such that (using \eqref{eq:uniformboundedness})
\[
\forall i\ \forall n:
\quad
\ownint{\{\|x\|^2>R\}}{}{\|x\|^2}{\gamma_n(x)}
+
\ownint{\{\|x\|^2>R\}}{}{\|x\|^2}{\mu^i(x)}
<\frac \delta {2N}.
\]
The bound
\[
g_n(x)
\le 2\|x\|^2
+\frac2N\sum_{i=1}^N  \|\mathbf t_{\gamma_n}^{\mu^i}(x)\|^2,
\]
implies that
\[
A_n
\subseteq \{x:\|x\|^2>R\}\cup
\bigcup_{i=1}^N \{x:\|\mathbf t_{\gamma_n}^{\mu^i}(x)\|^2>R\}.
\]
To deal with the sets in the union observe that (since $\mathbf t_{\gamma_n}^{\mu^i}$ is $\gamma_n$-almost surely injective),
\[
\gamma_n(\{x:\|\mathbf t_{\gamma_n}^{\mu^i}(x)\|^2>R\})
=\mu^i(\{x:\|x\|^2>R\})
<\frac \delta {2N},
\]
so that $\gamma_n(A_n)<\delta$.  We use this in conjunction with \eqref{eq:unifabscont} to bound
\begin{align*}
\ownint {A_n}{}{g_n(x)}{\gamma_n(x)}
&\le 2\ownint {A_n}{}{\|x\|^2}{\gamma_n(x)}
+\frac 2N\sum_{i=1}^N\ownint {A_n}{}{\|\mathbf t_{\gamma_n}^{\mu^i}(x)\|^2}{\gamma_n(x)}\\
&\le 2\epsilon
+\frac 2N\sum_{i=1}^N\ownint {\mathbf t_{\gamma_n}^{\mu^i}(A_n)}{}{\|x\|^2}{\mu^i(x)}
\le 4\epsilon,
\end{align*}
where we have used the measure-preservation property $\mu^i(\mathbf t_{\gamma_n}^{\mu^i}(A_n))=\gamma_n(A_n)<\delta$.

Define the truncation $g_{n,R}(x)=\min(g_n(x),4R)$.  Then $0\le g_n-g_{n,R}\le g_n\mathbf1\{g_n>4R\}$, so
\[
\ownint {}{}{[g_n(x) - g_{n,R}(x)]}{\gamma_n(x)}
\le
\ownint {A_n}{}{g_n(x)}{\gamma_n(x)}
\le 4\epsilon
,\qquad n=1,2,\dots.
\]
The analogous truncated function $g_R$ satisfies
\begin{equation}\label{eq:gprop}
0\le g_R(x)\le 4R
\quad \forall x\in \mathbb R^d
\quad\textrm{and}\quad
\{x:g_R \textrm{ is continuous }\}\textrm{ is of }\gamma\textrm{-full measure}.
\end{equation}

Let $E=\mathrm{supp}(\gamma)$.  Proposition~\ref{prop:unifcompacts} (Section \ref{appendix}) implies pointwise convergence of $\mathbf t_{\gamma_n}^{\mu^i}(x)$ to $\mathbf t_{\gamma}^{\mu^i}(x)$ for any $i=1,\dots,N$ and any $x\in E\setminus \mathcal N$, where $\mathcal N=\cup_{i=1}^N\mathcal N^i$ and
\[
\mathcal N^i
= (E\setminus E^{\mathrm{den}}) \cup \{x:\mathbf t_{\gamma}^{\mu^i}(x) \textrm{ contains more than one element}\}.
\]
Thus, $g_n$ and $g$ are univalued functions defined throughout $\mathbb R^d$,  and $g_n\to g$ pointwise on $x\in E\setminus \mathcal N$ (for whatever choice of representatives selected to define $g_n$);  consequently, $g_{n,R}\to g_R$ on $E\setminus \mathcal N$.

In order to restrict the integrands to a bounded set we invoke the tightness of the sequence $(\gamma_n)$ and introduce a compact set $K_\epsilon$ such that $\gamma_n(\mathbb R^d\setminus K_\epsilon)<\epsilon/R$ for all $n$.  Clearly, $g_{n,R}\to g_R$ on $E'=K_\epsilon\cap E\setminus \mathcal N$, and by Egorov's theorem (valid as $\mathrm{Leb}(E')\le\mathrm{Leb}(K_\epsilon)<\infty$), there exists a Borel set $\Omega=\Omega_\epsilon\subseteq E'$ on which the convergence is uniform, and $\mathrm{Leb}(E'\setminus \Omega)<\epsilon/R$.  Let us write
\[
\ownint {}{}{g_{n,R}}{\gamma_n} - \ownint {}{}{g_R}\gamma
=\ownint {}{}{g_R}{(\gamma_n-\gamma)}+\ownint \Omega{}{(g_{n,R} - g_R)}{\gamma_n} + \ownint {\mathbb R^d\setminus\Omega}{}{(g_{n,R} - g_R)}{\gamma_n},
\]
and bound each of the three integrals at the right-hand side as $n\to\infty$.

The first integral vanishes as $n\to\infty$, by \eqref{eq:gprop} and the Portmanteau lemma (Lemma~\ref{lem:portmanteau}, Section \ref{appendix}). For a given $\Omega$, the second integral vanishes as $n\to\infty$, since $g_{n,R}$ converge to $g_R$ uniformly. The third integral is bounded by $8R\gamma_n(\mathbb R^d\setminus \Omega)$.  The latter set is a subset of $\mathcal N\cup(E'\setminus \Omega)\cup(\mathbb R^d\setminus E)\cup (\mathbb R^d\setminus K_\epsilon)$, where the first set is Lebesgue-negligible and the second has Lebesgue measure smaller than $\epsilon/R$.  The hypothesis of the densities of $\gamma_n$ implies that $\gamma_n(A)\le C\mathrm{Leb}(A)$ for any Borel set $A\subseteq\mathbb R^d$ and any $n\in\mathbb N$;  it follows from this and $\gamma_n(\mathbb R^d\setminus K_\epsilon)<\epsilon/R$ that
\[
\left|\ownint {\mathbb R^d\setminus\Omega}{}{(g_{n,R} - g_R)}{\gamma_n}\right|
\le 8R(C\epsilon/R + \gamma_n(\mathbb R^d\setminus E)+\epsilon/R)
= 8\left(R\gamma_n(\mathbb R^d\setminus E) + C\epsilon+\epsilon\right).
\]
Write the open set $E_1=\mathbb R^d\setminus E$ as a countable union of closed sets $A_k$ with $\mathrm{Leb}(E_1\setminus A_k)<1/k$, and conclude that
\[
\limsup_{n\to\infty}\gamma_n(E_1)
\le \limsup_{n\to\infty}\gamma_n(A_k)
+ \limsup_{n\to\infty}\gamma_n(E_1\setminus A_k)
\le \gamma(A_k)+\frac Ck
=\frac Ck,
\]
where we have used the Portmanteau lemma again, $A_k\cap \mathrm{supp}(\gamma)=\emptyset$ and $\gamma_n(A)\le C\mathrm{Leb}(A)$.  Consequently, for all $k$
\[
\limsup_{n\to\infty}\left|\ownint {}{}{g_{n,R}}{\gamma_n} - \ownint {}{}{g_R}\gamma\right|
\le\limsup_{n\to\infty}\left|\ownint {\mathbb R^d\setminus\Omega}{}{(g_{n,R} - g_R)}{\gamma_n}\right|
\le \frac{8R_\epsilon C}k +8(C+1)\epsilon.
\]
Letting $k\to\infty$, then incorporating the truncation error yields
\[
\limsup_{n\to\infty}\left|\ownint {}{}{g_{n}}{\gamma_n} - \ownint {}{}{g}\gamma\right|
\le 8(C+1)\epsilon + 8\epsilon.
\]
The proof is complete upon noticing that $\epsilon$ is arbitrary.
\end{proof}

Our proof will now be complete if we show that the sequence $\gamma_k$ generated by the algorithm satisfies the assumptions of the last Proposition.  First we show that limits of the sequence are indeed regular.
\begin{proposition}[Sequence has bounded density]\label{prop:Cmu}
Let $\mu^i$ have density $g^i$ for $i=1,\dots,N$ and let $\gamma_0$ be a regular probability measure.  Then the density of $\gamma_1$ is bounded by a constant $C_\mu=\min\{N^{d-1}\max_i \|g^i\|_\infty, N^d\min_i \|g^i\|_\infty\}$ that depends only on $\{\mu^1,\dots,\mu^N\}$.
\end{proposition}

\begin{proof}  Let $h_i$ be the density of $\gamma_i$.  By the change of variables formula, for $\gamma_0$-almost any $x$
\[
h_1(\mathbf t_{\gamma_0}^{\gamma_1}(x))
=\frac{h_0(x)} {\mathrm{det}\nabla \mathbf t_{\gamma_0}^{\gamma_1}(x)}
;\qquad
g^i(\mathbf t_{\gamma_0}^{\mu^i}(x))
=\frac{h_0(x)} {\mathrm{det}\nabla \mathbf t_{\gamma_0}^{\mu^i}(x)}.
\]
Fiedler \cite{fiedler1971bounds} shows that if $B_1$ and $B_2$ are $d\times d$ positive semidefinite matrices with eigenvalues $0\le\alpha_i,\beta_i$, then
\[
\mathrm{det}(B_1+B_2)
\ge \prod_{i=1}^d(\alpha_i + \beta_i).
\]
The right-hand side contains $2^d$ nonnegative summands of which two are $\mathrm{det}B_1$ and $\mathrm{det}B_2$, and so we see that $\mathrm{det}(B_1+B_2)\ge \mathrm{det}B_1+\mathrm{det}B_2$.  (One can show the stronger result $\sqrt[d]{\mathrm{det}(B_1+B_2)}\ge\sqrt[d]{ \mathrm{det}B_1}+\sqrt[d]{\mathrm{det}B_2}$.)  Since $\nabla \mathbf t_{\gamma_0}^{\gamma_1}$ is an average of $N$ $d\times d$ positive semidefinite matrices, we obtain
\[
h_1(\mathbf t_{\gamma_0}^{\gamma_1}(x))
=\frac{N^d h_0(x)} {\mathrm{det}\sum \nabla \mathbf t_{\gamma_0}^{\mu^i}(x)}
\le \frac{N^dh_0(x)} {\sum \mathrm{det}\nabla \mathbf t_{\gamma_0}^{\mu^i}(x)}
= N^d \left[ \sum_{i=1}^N \frac 1{g^i(\mathbf t_{\gamma_0}^{\mu^i}(x))} \right]^{-1}
\le N^d\left[  \sum_{i=1}^N \frac 1{\|g^i\|_\infty} \right]^{-1}.
\]
Let $\Sigma$ be the set of points where this inequality holds;  then $\gamma_0(\Sigma)=1$.  Hence
\[
\gamma_1(\mathbf t_{\gamma_0}^{\gamma_1}(\Sigma))
= \gamma_0[(\mathbf t_{\gamma_0}^{\gamma_1})^{-1}(\mathbf t_{\gamma_0}^{\gamma_1}(\Sigma))]
\ge \gamma_0(\Sigma)=1.
\]

Thus $\gamma_1$-almost surely,
\[
h_1
\le N^d\left[ \sum_{i=1}^N \frac 1{\|g^i\|_\infty}\right]^{-1}
\le \min\left\{N^{d-1}\max_i\|g^i\|_\infty
,N^d\min_i \|g^i\|_\infty\right\}
=C_\mu.
\]
For $C_\mu$ to be finite it suffices that $\|g^i\|_\infty$ be finite for some $i$.
\end{proof}

\noindent Our task is now essentially complete. All that remains is to show:

\begin{corollary}[Limits are regular]
Every limit of the sequence generated by Algorithm~\ref{algo} is absolutely continuous provided the density of $\mu^i$ is bounded for some $i$.
\end{corollary}
\begin{proof} Each $\gamma_k$ ($k=1,2,\dots$) has a density that is bounded by the finite constant $C_\mu$.  For any open set $O$, $\liminf \gamma_k(O)\le C_\mu\mathrm{Leb}(O)$, so any limit point $\gamma$ of $(\gamma_k)$ is such that $\gamma(O)\le C_\mu\mathrm{Leb}(O)$ by the Portmanteau lemma.  It follows that $\gamma$ is absolutely continuous with density bounded by $C_\mu$.  We note that Agueh and Carlier \cite{bary} show that the density of the Fr\'echet mean is bounded by $N^d\min_i\|g^i\|_\infty\ge C_\mu$, a slightly weaker bound.
\end{proof}

\begin{proof}[Proof of Theorem~\ref{thm:convprocrustes}]
Let $E=\mathrm{supp}(\bar\mu)$ and set $A^i=E^{\mathrm{den}}\cap \{x:\mathbf t_{\bar\mu}^{\mu^i}(x)\textrm{ is multivalued}\}$.  By Corollary~\ref{cor:pointwise} $\bar\mu(A^i)=1$.  Choose $A=\cap_{i=1}^NA^i$ and apply Proposition~\ref{prop:unifcompacts}.  This proves the first assertion.

Now let $E^i=\mathrm{supp}(\mu^i)$ and set $B^i=(E^i)^{\mathrm{den}}\cap \{x:\mathbf t_{\mu^i}^{\bar\mu}(x)\textrm{ is univalued}\}$.   Since $\mu^i$ is regular, $\mu^i(B^i)=1$.  Apply Proposition~\ref{prop:unifcompacts}.  If in addition $E^1=\dots=E^N$ then $\mu^i(B)=1$ for $B=\cap B^i$.
\end{proof}

\begin{proof}[Proof of Corollary~\ref{cor:convMulti}]

The proof is very similar to that of Proposition~\ref{prop:Fdercont}.  Define $\eta_j,\eta\in \mathcal P_2((\mathbb R^{d})^{N+1})$ by
\[
\eta_j
=\left(\mathbf t_{\gamma_j}^{\mu^1},\dots\mathbf t_{\gamma_j}^{\mu^n},\mathbf i\right)\#\gamma_j
,\qquad
\eta
=\left(\mathbf t_{\gamma}^{\mu^1},\dots\mathbf t_{\gamma}^{\mu^n},\mathbf i\right)\#\gamma.
\]
We establish convergence of $\eta_j$ to $\eta$.  Since the optimal multicouplings are marginals of $\eta_j$ and $\eta$ their convergence follow.  Let $h:(\R^d)^{N+1}\to\R$ be any continuous function such that
\[
|h(t_1,\dots,t_N,y)|
\le
\frac1N\sum_{i=1}^N \|t_i\|^2 + \|y\|^2.
\]
Define $g_j:\mathbb R^d\to\mathbb R$ by $g_j(x)=h(\mathbf t_{\gamma_j}^{\mu^1},\dots\mathbf t_{\gamma_j}^{\mu^n},x)$ and analogously define $g$.  By \cite[Theorem~7.12]{vil} it suffices to show that (if this holds for $h$, it also holds for $a+bh$ with $a,b$ scalars)
\[
\ownint {\R^{d(N+1)}}{}{h}{\eta_j}
=\ownint {\R^d}{}{g_j}{\gamma_j(x)}
\to
\ownint {\R^d}{}{g}{\gamma(x)}
=\ownint {\R^{dN}}{}{h}{\eta}.
\]
(In Proposition~\ref{prop:Fdercont} we had $h=\|y - \bar t\|^2$.)  Since $h$ is continuous, we can modify $g_n$ and $g$ to be well-defined, finite and so that $g$ be continuous $\gamma$-almost surely.  Define $R$ as in the proof of Proposition~\ref{prop:Fdercont}, $A_j=\{x:|g_j(x)|\ge 4R\}$, invoke \eqref{eq:unifabscont} and translate the bound on $h$ to a bound on $|g_j|$ to conclude that $\ownint {A_j}{}{|g_j|}{\gamma_j}\le4\epsilon$.  Carry out the same (now two-sided) truncation $g_{j,R}(x)=\max(-4R,\min(g_j(x),4R))$ to obtain $|g_j - g_{j,R}|\le |g_j|\mathbf 1\{|g_j|>4R\}$, $|g_R|\le 4R$ and $g_R$ is continuous $\gamma$-almost surely (see \eqref{eq:gprop}).  The rest can be done as in the proof of Proposition~\ref{prop:Fdercont}, since it did not depend on the precise form of $g$.
\end{proof}

\subsection{Proofs of Statements in Section \ref{sec:discrete_intro}}
\begin{proof}[Proof of Theorem~\ref{thm:meanTid}]
Since $\lambda$ is regular and $T$ is injective with nonsingular derivative, $\Lambda=T\#\lambda$ is also regular by Lemma~5.5.3 in \cite{ambgigsav}.  Moreover, $\Lambda$ is supported on $K$ because $T$ takes values there.  Consequentely, the Fr\'echet mean of $\Lambda$ is unique and supported itself on $K$;  this is essentially a consequence of Corollary~2.9 in \cite{alvarez2011}. For tidiness, we provide the full details in Section~\ref{zempan16supp}.

In view of the preceding paragraph, it suffices to show that
\[
\mathbb Ed^2(\lambda,\Lambda)
\le \mathbb E d^2(\theta,\Lambda),
\qquad \theta\in \mathcal{P}_2(K).
\]
As a gradient of a convex function, $T=\mathbf t_\lambda^\Lambda$ is optimal.  Let $\phi$ be the convex potential of $T$, and define $\phi^*$ its Legendre transform.  Then the pair $(\|x\|^2/2-\phi,\|y\|^2/2 - \phi^*)$ is dual optimal.  Invoking strong duality for $\lambda$ and weak duality for $\theta$, we find
\begin{align*}
d^2(\lambda,\Lambda)
&= \ownint {\mathbb R^d}{}{\left(\frac12\|x\|^2 - \phi(x)\right)}{\lambda(x)}
+ \ownint {\mathbb R^d}{}{\left(\frac12\|y\|^2 - \phi^*(y)\right)}{\Lambda(y)};\\
d^2(\theta,\Lambda)
&\ge \ownint {\mathbb R^d}{}{\left(\frac12\|x\|^2 - \phi(x)\right)}{\theta(x)}
+ \ownint {\mathbb R^d}{}{\left(\frac12\|y\|^2 - \phi^*(y)\right)}{\Lambda(y)}.
\end{align*}
By Fubini's theorem (see the supplementary material for a justification), we have
\begin{align*}
\E d^2(\lambda,\Lambda)
&= \ownint {\mathbb R^d}{}{\left(\frac12\|x\|^2 - \E\phi(x)\right)}{\lambda(x)}
+ \E\ownint {\mathbb R^d}{}{\left(\frac12\|y\|^2 - \phi^*(y)\right)}{\Lambda(y)};\\
\E d^2(\theta,\Lambda)
&\ge \ownint {\mathbb R^d}{}{\left(\frac12\|x\|^2 - \E\phi(x)\right)}{\theta(x)}
+ \E\ownint {\mathbb R^d}{}{\left(\frac12\|y\|^2 - \phi^*(y)\right)}{\Lambda(y)}.
\end{align*}
The function $\E T$ is continuous (by the bounded convergence theorem and boundedness of $K$), so equals the identity for all $x\in K$.  Again by Fubini's theorem (see the supplementary material), it follows that $\E\phi(x)=\|x\|^2/2$ for all $x\in K$, perhaps up to an additive constant.  Since $\theta$ and $\lambda$ are both supported on $K$, the integrals with respect to $\lambda$ and $\theta$ vanish, and this completes the proof.
\end{proof}

\noindent As part of our proofs, we will need to control the Wasserstein distance between the regularised measures and their true counterparts:

\begin{lemma}\label{lem:smoothing}
The smooth measure $\widehat\Lambda_i$ defined by \eqref{eq:lambdahat} satisfies
\begin{equation}\label{eq:smoothbound}
d^2\left(\widehat\Lambda_i,\frac{\widetilde\Pi_i}{\widetilde\Pi_i(K)}\right)
\le C_{\psi,K}\sigma^2
\quad\textrm{ if }
\sigma\le 1
\quad\textrm{ and } \quad
\widetilde\Pi_i(K)>0,
\end{equation}
where $C_{\psi,K}$ is a (finite) constant that depends only on $\psi$ and $K$.
\end{lemma}
We prove the lemma in the supplementary material (Section~\ref{zempan16supp}).
\begin{remark}\label{rem:nonisotropic}
There is no need for $\psi$ to be isotropic:  it is sufficient that merely
\[
\delta_\psi(r)
=\inf_{\|x\|\le r}  \psi(x)
> 0
,\qquad r>0,
\]
which is satisfied as long as $\psi$ is continuous and strictly positive.
\end{remark}

We now remark that a trivial extension of \cite[Lemma~3]{panzem15} yields:
\begin{lemma}[Number of points per process is $O(\tau_n)$]\label{lem:numpoints}
If $\tau_n/\log n\to\infty$, then there exists a constant $C_\Pi>0$, depending only on the distribution of the $\Pi$'s, such that
\[
\liminf_{n\to\infty}  \frac{\min_{1\le i\le n}  \Pi_i^{(n)}(K)}   {\tau_n}
 \ge C_\Pi
\quad \textrm{almost surely}.
\]
In particular, there are no empty point processes, so the normalisation is well-defined.
\end{lemma}

\begin{proof}[Proof of Theorem~\ref{thm:consistency}]
The proof is very similar to that of Theorem~1 in Panaretos \& Zemel \cite{panzem15}, and we give the details in the supplementary material (Section~\ref{zempan16supp}).
\end{proof}

\begin{proof}[Proof of Theorem~\ref{thm:consistencyWarp}]
The argument is considerably different than the case $d=1$ considered in \cite{panzem15}, and brings into play the geometry of convex functions in $\mathbb R^d$.  Let $i$ be a fixed integer and for $n\ge i$ set
\[
\mu_n = \widehat\Lambda_i;
\qquad \nu_n = \widehat\lambda_n;
\qquad\mu = \Lambda_i;
\qquad\nu = \lambda;
\qquad u_n = \widehat T_i^{-1};
\qquad u = T_i^{-1}.
\]
We wish to show that $u_n\to u$ uniformly on compact sets, using our knowledge that
\[
\begin{cases}
\mu_n\to \mu;\\
\nu_n\to \nu;
\end{cases}
\quad u_n\#\mu_n=\nu_n;
\quad u\#\mu = \nu;
\qquad u_n, u\textrm{ optimal}.
\]
This follows from Proposition~\ref{prop:unifcompacts} below.  To verify the conditions, notice that all the measures are supported on $K=E$, a compact and convex set.  Furthermore $\mu_n$, $\mu$ and $\nu$ all have strictly positive densities there, so their support is exactly $K$.  Continuity of $u$ on $\mathrm{int}(K)$ follows from the assumptions that $T_i$ and $T_i^{-1}$ are continuous.  The finiteness in \eqref{eq:measuressetting} follows from the compactness of $K$, and the uniqueness follows from the regularity of $\mu$.

The same proposition can be applied to show convergence of $\widehat T_i$ to $T_i$ uniformly on $\Omega\subseteq \mathrm{int}(K)$:  one needs to reverse the roles of $\mu_n$ and $\nu_n$ and of $\mu$ to $\nu$, and notice that $\nu$ too is regular, which guarantees the uniqueness in \eqref{eq:measuressetting}.
\end{proof}

\begin{proof}[Proof of Corollary~\ref{cor:consistencyPi}]   The square of the distance is
\[
\ownint K{}{\|\widehat T_i^{-1}(T_i(x)) - x\|^2}{\frac{\Pi_i}{\Pi_i(K)}},
\]
and this is well-defined (that is, $\Pi_i(K)>0$) almost surely for $n$ large enough by Lemma \ref{lem:numpoints}.  Since $\lambda(\partial K)=0$, almost surely there are no points on the boundary and the integral can be taken on the interior of $K$.  Let $\Omega\subseteq\mathrm{int}(K)$ be compact and split the integral to $\Omega$ and its complement.  Then
\[
\ownint{\mathrm{int}(K)\setminus \Omega}{}{\|\widehat T_i^{-1}(T_i(x)) - x\|^2}{\frac{\Pi_i}{\Pi_i(K)}}
\le d_K^2\frac{\Pi_i(\mathrm{int}(K)\setminus\Omega)} {\tau_n}  \frac{\tau_n} {\Pi_i(K)}
\stackrel{\rm{as}}\to d_K^2\lambda(\mathrm{int}(K)\setminus\Omega),
\]
by the law of large numbers.  Since the interior of $K$ can be written as a countable union of compact sets, the right-hand side can be made arbitrarily small by selection of $\Omega$.

Let us now consider the integral on $\Omega$.  Since
\[
\ownint \Omega {}{\|\widehat T_i^{-1}(T_i(x)) - x\|^2}{\frac{\Pi_i}{\Pi_i(K)}}
\le \sup_{x\in\Omega} \|\widehat T_i^{-1}(T_i(x)) - x\|^2
= \sup_{y\in T_i(\Omega)} \|\widehat T_i^{-1}(y) - T_i^{-1}(y)\|^2
\]
and $T_i(\Omega)$ is compact, we only need to show that it is included in $\mathrm{int}(K)$ in order to apply Theorem~\ref{thm:consistencyWarp}.  Suppose towards contradiction that $y=T_i(x)\in\partial K$ for $x\in \mathrm{int}(K)$.  Let $\alpha\in\mathbb R^d\setminus\{0\}$ with $\langle y,\alpha\rangle\ge\sup\langle K,\alpha\rangle$.  Let $x'=x+t\alpha$ for $t>0$ small enough such that $x'\in\mathrm{int}(K)$.  Then $y'=T_i(x')\in K$, so that
\[
0\le\innprod{y' - y}{x' - x}
=t\innprod{y' - y}\alpha.
\]
Either condition in the statement of the corollary imply that $y'=y$, in contradiction to $T_i$ being injective.
\end{proof}

\subsection{Monotone Operators, Optimal Transportation, Stochastic Convergence}\label{appendix}
This section contains the statements and proofs of analytical results needed in our proofs, culminating in Proposition~\ref{prop:unifcompacts}.  The latter is the backbone result needed for the proofs of Theorem~\ref{thm:consistencyWarp}, Theorem~\ref{thm:convtosta} (more precisely, Proposition~\ref{prop:Fdercont}) and Theorem~\ref{thm:convprocrustes}.  Rather than start with all the background definitions we will define the necessary objects en route.

We shall follow the notation and terminology of Alberti \& Ambrosio \cite{albamb}.  Let $u$ be a set-valued function (or multifunction) on $\mathbb R^d$, that is, $u:\mathbb R^d\to 2^{\mathbb R^d}$. It is said that $u$ is \emph{monotone} if
\[
\innprod{y_2-y_1}{x_2-x_1}
\ge0
\qquad\textrm{whenever }y_i\in u(x_i)
\quad(i=1,2).
\]
When $d=1$, the definition reduces to $u$ being a nondecreasing (set-valued) function.  It is said that $u$ is \emph{maximal} if no points can be added to its graph while preserving monotonicity:
\[
\left\{
\innprod{y'-y}{x'-x}
\ge0
\quad\textrm{whenever }y\in u(x)
\right\}
\quad\Longrightarrow\quad y'\in u(x').
\]
We sometimes use the notation $(x,y)\in u$ to mean $y\in u(x)$. Note that $u(x)$ can be empty, even when $u$ is maximal.

The relevance of monotonicity stems from the fact that subdifferentials of convex functions are monotone.  That is, if $\varphi:\mathbb R^d\to\mathbb R\cup\{\infty\}$ is lower semicontinuous and convex (and not identically infinite), then $u=\partial\varphi$ is maximally monotone \cite[Section~7]{albamb}, where
\[
\partial\varphi(x)
= \{ y :  \varphi(z)  \ge  \varphi(x) + \innprod y{z-x} \textrm{ for any }z\}
\]
is the \emph{subdifferential} of $\varphi$ at $x$.  Here $u(x)=\emptyset$ if $\varphi(x)=\infty$.

We will use extensively the continuity of $u$ at points where it is univalued.
\begin{proposition}[Continuity at Singletons]\label{prop:cont}
Let $u$ be a maximal monotone function, and suppose that $u(x)=\{y\}$ is a singleton. Then $u$ is nonempty on some neighbourhood of $x$ and it is continuous at $x$: if $x_n\to x$ and $y_n\in u(x_n)$, then $y_n\to y$.
\end{proposition}
\begin{proof} See \cite[Corollary~1.3(4)]{albamb}.  Notice that this result implies that differentiable convex functions are continuously differentiable \cite[Corollary~25.5.1]{roc}.\end{proof}

It turns out that when $u$ is univalued, monotonicity is a local property.  To state the result in the general form that we shall use, we need to introduce the notion of points of Lebesgue density.

Let $B_r(y)=\{x:\|x-y\|<r\}$ for $r\ge0$ and $y\in\mathbb R^d$.  A point $x_0$ is of \emph{Lebesgue density} of a measurable set $G\subseteq \mathbb R^d$ if for any $\epsilon>0$ there exists $t_\epsilon>0$ such that
\[
\frac{\mathrm{Leb}(B_t(x_0) \cap G)}  {\mathrm{Leb}(B_t(x_0))}
> 1 - \epsilon,
\qquad 0<t<t_\epsilon.
\]
We denote the set of points of Lebesgue density of $G$ by $G^{\rm{den}}$.  Clearly, $G^{\rm{den}}$ lies between $\mathrm{int}(G)$ and $\overline G$.  Stein and Shakarchi \cite[Chapter~3, Corollary~1.5]{steinmeasure} show that almost any point of $G$ is in $G^{\rm{den}}$.  By the Hahn--Banach theorem, $G^{\rm{den}}\subseteq \mathrm{int}(\mathrm{conv}(G))$.

\begin{lemma}[Density Points and Distance]\label{lem:densityodist}
Let $x_0$ be a point of Lebesgue density of a measurable set $G\subseteq \mathbb R^d$.  Then
\[
\delta(z)
=\inf_{x\in G} \|z - x\|
=o(\|z-x_0\|),
\qquad \textrm{as }z\to x_0.
\]
\end{lemma}
This result was given as an exercise in \cite{steinmeasure};  for completeness we provide a full proof in the supplementary material (Section~\ref{zempan16supp}).
\begin{lemma}[Local Monotonicity]\label{lem:monislocal}
Let $u$ be a maximal monotone function such that $u(x_0)=\{y_0\}$. Suppose that $x_0$ is a point of Lebesgue density of a set $G$ satisfying
\[
\langle y-y^*,x-x_0\rangle
\ge0
\qquad\forall x\in G
\ \forall y\in u(x).
\]
Then $y^*=y_0$.  In particular, the result is true if the inequality holds on $G=O\setminus \mathcal N$ with $\emptyset\ne O$ open and $\mathcal N$ Lebesgue negligible.
\end{lemma}

\begin{proof}  Set $z_t=x_0+t(y^*-y_0)$ for $t>0$ small.  It may be that $z_t\notin G$;  but Lemma~\ref{lem:densityodist} guarantees existence of $x_t\in G$ with $\|x_t - z_t\| /t \to0$.  By Proposition~\ref{prop:cont} $u(x_t)$ is nonempty for $t$ small enough.  For $y_t\in u(x_t)$,
\begin{align*}
0\le \innprod {y_t - y^*} {x_t - x_0}
&= \innprod{y_t - y^*}{x_t - z_t} + \innprod{y_t - y^*}{z_t - x_0}\\
&= \innprod{y_t - y^*}{x_t - z_t} + t\innprod{y_t - y_0}{y^* - y_0} - t\|y^* - y_0\|^2.
\end{align*}
Rearrangement, division by $t>0$ and application of the Cauchy--Schwartz inequality gives
\[
\|y^* - y_0\|^2
\le  \|y_t - y_0\| \|y^* - y_0\| + t^{-1}\|x_t - z_t\| \left(\|y_t - y_0\| + \|y^* - y_0\|\right).
\]
As $t\searrow 0$ the right-hand side vanishes, since $y_t \to y_0$ (Proposition~\ref{prop:cont}) and $\|x_t - z_t\|/t\to0$.  It follows that $y^*=y_0$.
\end{proof}

This concludes the necessary discussion on monotone operators. We will now state some necessary results on optimal transportation maps, and specifically their convergence properties. Consider the following setting:  let $\{\mu_n\}$, $\{\nu_n\}$ be two sequences of probability measures on $\mathbb R^d$ that converge weakly to $\mu$ and $\nu$ respectively.  Let $\pi_n$ be an optimal coupling between $\mu_n$ and $\nu_n$ having finite cost, which is supported on the graph of a subdifferential of a proper (not identically infinite) convex lower semicontinuous function $\varphi_n$ \cite[Chapter~2]{vil}. The set-valued function $u_n=\partial\varphi_n$ that maps $x$ to the subdifferential of $\varphi_n$ at $x$ is maximally monotone \cite[Section~7]{albamb}. The appropriate functions for $\mu$ and $\nu$ will be denoted by $\varphi$ and $u=\partial\varphi$ and the optimal coupling by $\pi$.  This setting will be succinctly referred to by the equation
\begin{equation}\label{eq:measuressetting}
\begin{array}{l}
\mu_n\to \mu\\
\nu_n\to \nu
\end{array}
\quad
\begin{array}{lll}
\pi_n \textrm{ finite} & \textrm{optimal for }\mu_n,\nu_n & (u_n = \partial\varphi_n)\#\mu_n = \nu_n\\
\pi \textrm{ unique} & \textrm{optimal for }\mu,\nu &(u = \partial\varphi)\#\mu = \nu.
\end{array}
\end{equation}
We notice now that uniqueness of $\pi$ and the stability of optimal transportation imply that $\pi_n$ converge weakly to $\pi$ (even if $\pi_n$ is not unique); see Schachermayer \& Teichmann \cite[Theorem~3]{tei} or Cuesta-Albertos et al. \cite[Theorem~3.2]{cuesta1997optimal}.  This weak convergence will be used in the following form:

\begin{lemma}[Portmanteau]\label{lem:portmanteau}
Weak convergence of Borel probability measures $\mu_k$ to $\mu$ on $\mathbb R^d$ is equivalent to any of the following conditions:
\begin{itemize}
\item[(I)] for any open set $G$, $\liminf \mu_k(G)\ge \mu(G)$;
\item[(II)] for any closed set $F$, $\limsup \mu_k(F)\le \mu(F)$;
\item[(III)] $\ownint{}{}h{\mu_k}\to \ownint{}{}h\mu$ for any bounded measurable $h$ whose set of discontinuity points is a $\mu$-null set.
\end{itemize}
\end{lemma}
\begin{proof}
The equivalence with the first two conditions is classical and can be found in Billingsley \cite[Theorem~2.1]{bil};  for the third, see Pollard \cite[Section~III.2]{pollard2012convergence}.
\end{proof}

We shall now translate this into convergence of $u_n$ to $u$ under certain regularity conditions.

\begin{proposition}[Uniform Convergence of Optimal Maps]\label{prop:unifcompacts}
In the setting of Display \eqref{eq:measuressetting}, denote $E=\mathrm{supp}(\mu)$.

Let $\Omega$ be a compact subset of $E^{\mathrm{den}}$ on which $u$ is univalued, where $E^{\mathrm{den}}$ is the set of points of Lebesgue density of $E$. Then $u_n$ converges to $u$ uniformly on $\Omega$: $u_n(x)$ is nonempty for all $x\in \Omega$ and all $n>N_\Omega$, and
\[
\sup_{x\in \Omega}\sup_{y\in u_n(x)}    \|y - u(x)\|   \to   0
,\qquad n\to\infty.
\]
In particular, if $u$ is univalued throughout $\mathrm{int}(E)$ (so that $\varphi\in C^1$ there), then uniform convergence holds for any compact $\Omega\subset \mathrm{int}(E)$.
\end{proposition}

\begin{corollary}[Pointwise convergence $\mu$-almost surely]\label{cor:pointwise}
If in addition $\mu$ is absolutely continuous then $u_n(x)\to u(x)$ $\mu$-almost surely.
\end{corollary}
\begin{proof}  The set of points $x\in E$ for which $\Omega=\{x\}$ fails to satisfy the conditions of Proposition~\ref{prop:unifcompacts} is included in
\[
(E\setminus E^{\mathrm{den}})
\cup \{x\in \mathrm{int}(\mathrm{conv}(E)):u(x)\textrm{ contains more than one point}\}.
\]
(Since $u$ is nonempty on $\mathrm{int}(\mathrm{conv}(E))$ by \cite[Corollary~1.3(2)]{albamb}.)  Both sets are Lebesgue-negligible (see \cite[Remark~2.3]{albamb} for the latter), and $\mu$ is absolutely continuous.
\end{proof}

\begin{remark}
In the setting of Theorem~\ref{thm:consistencyWarp}, $E$ is convex, $\mu$ is absolutely continuous, and $u$ is univalued on $\mathrm{int}(E)$, so one can take any $\Omega\subseteq\mathrm{int}(E)$, without the need to introduce Lebesgue density.  The more general statement of the proposition is used in the proof of Proposition~\ref{prop:Fdercont}, where we have no control on the support of $\gamma$ or the regularity of the transport maps.
\end{remark}

We split the proof of Proposition \ref{prop:unifcompacts} into two steps: (1) Limit points of the graphs of $u_n$ are in the graph of $u$ (Lemma~\ref{lem:liminT}); (2) Points in the graphs of $u_n$ stay in a bounded set (Proposition~\ref{lem:yinkstrong}). Each of these points will be proven using one intermediate lemma.

\begin{lemma}[Points in the limit graph are limit points]\label{lem:exists}
Assume \eqref{eq:measuressetting}.  For any $x_0\in \mathrm{supp}(\mu)$ such that $u(x_0)=\{y_0\}$ is a singleton there exists a subsequence $(x_{n_k},y_{n_k})\in u_{n_k}$ that converges to $(x_0,y_0)$.
\end{lemma}

\begin{proof}
Since $u=\partial\varphi$ is a maximal monotone function \cite[Section~7]{albamb} that is univalued at $x_0$, it is continuous there (Proposition~\ref{prop:cont}). This means that for any $\epsilon>0$ there exists $\delta>0$ such that if $x\in B_\delta(x_0)=\{x:\|x-x_0\|<\delta\}$ then $u(x)$ is nonempty and if $y\in u(x)$, then $\|y-y_0\|<\epsilon$. Take $\epsilon_k\to0$ and corresponding $\delta_k\to0$, and set $B_k=B_{\delta_k}(x_0)$, $V_k=B_{\epsilon_k}(y_0)$. Then $u(B_k)\subseteq V_k$, so
\[
\pi(B_k\times V_k)
=\pi\{(x,y):x\in B_k,y\in u(x)\cap V_k\}
=\pi\{(x,y):x\in B_k,y\in u(x)\}
=\mu(B_k)
>0,
\]
because $B_k$ is a neighbourhood of $x_0\in \mathrm{supp}(\mu)$. Since $B_k\times V_k$ is open, we have by the Portmanteau lemma that $\pi_n(B_k\times V_k)>0$ for $n$ large. Consequently, there exists $n_k$ such that
\[
\pi_{n_k}(B_k\times V_k)
>0
\qquad\textrm{and }n_k\to\infty
\quad\textrm{as }k\to\infty.
\]
Since $\pi_{n_k}$ is concentrated on the graph of $u_{n_k}$, it follows that there exist $(x_{n_k},y_{n_k})\in u_{n_k}$ with $\|x_{n_k}-x_0\|<\delta_k$ and $\|y_{n_k}-y_0\|<\epsilon_k$.  Hence $(x_{n_k},y_{n_k})\to(x_0,y_0)$.
\end{proof}

\begin{lemma}[Limit points are in the limit graph]\label{lem:liminT}
Assume that \eqref{eq:measuressetting} holds and denote $E=\mathrm{supp}(\mu)$.  If a subsequence $(x_{n_k},y_{n_k})\in u_{n_k}$ converges to $(x_0,y^*)$, where $x_0$ is a point of Lebesgue density of $E$, and $u(x_0)$ is a singleton, then $y^*=u(x_0)$.  In particular, the statement is true if $x_0\in \mathrm{int}(E)$ and $u(x_0)$ is a singleton.
\end{lemma}

\begin{proof}  The set $\mathcal N\subseteq\mathbb R^d$ of points where $u$ contains more than one element is Lebesgue negligible \cite[Remark~2.3]{albamb}.  There exists a neighbourhood $V$ of $x_0$ on which $u$ is nonempty (Proposition~\ref{prop:cont}).  Thus, $x_0$ is a point of Lebesgue density of $G=(E\cap V)\setminus \mathcal N$, and $u(x)$ is a singleton for every $x\in G$.  Fix such an $x$ and set $y=u(x)$.  By Lemma~\ref{lem:exists} (applied to $\{u_{n_k}\}_{k=1}^\infty$ at $x$) there exist sequences $x_{n_{k_l}}'\to x$ and $y_{n_{k_l}}'\to y$ with $(x_{n_{k_l}}',y_{n_{k_l}}')\in u_{n_{k_l}}$. Consequently,
\[
\langle y-y^*,x-x_0\rangle
=\lim_{l\to\infty}\langle y_{n_{k_l}}'-y_{n_{k_l}},    x_{n_{k_l}}'-x_{n_{k_l}}\rangle
\ge0.
\]
This holds for any $(x,y)\in u$ such that $x\in G$. Since $x_0$ is a point of Lebesgue density of $G$ (and $u$ is maximal), it follows from Lemma~\ref{lem:monislocal} that $y^*=u(x_0)$.
\end{proof}

Let $B_\epsilon^\infty(x_0)=\{x:\|x-x_0\|_\infty<\epsilon\}$ be the $\ell_\infty$ ball around $x_0$ and $\overline B_\epsilon^\infty(x_0)$ its closure.

\begin{lemma}[Continuity of Convex Hulls]\label{lem:linfchull}
Let $Z=\{z_i\}\subseteq\mathbb R^d$ be a set of points whose convex hull, $\mathrm{conv}(Z)$, includes $B_\rho^\infty(x_0)$ and let $\tilde Z=\{\tilde z_i\}$ be a set of points such that $\|\tilde z_i - z_i\|_\infty\le\epsilon$.  Then the convex hull of $\tilde Z$ includes $B_{\rho-\epsilon}^\infty(x_0)$.
\end{lemma}
For a proof, see Section~\ref{zempan16supp}.

\begin{proposition}[Boundedness]\label{lem:yinkstrong}
Suppose that \eqref{eq:measuressetting} holds, and fix a compact $\Omega\subseteq\mathrm{int}(\mathrm{conv}(\mathrm{supp}(\mu)))$.  Then for $n>N(\Omega)$ sufficiently large, $u_n(x)$ is nonempty for all $x\in\Omega$ and $u_n(\Omega)$ is bounded uniformly.
\end{proposition}

\begin{proof}  Denote $E=\mathrm{supp}(\mu)$ and its convex hull by $F=\mathrm{conv}(E)$. There exists $\delta=\delta(\Omega)>0$ such that the closed $\ell_\infty$-ball, $\overline B^\infty_{3\delta}(\Omega)$, is included in $\mathrm{int}(F)$.  Cover $\Omega$ by a finite union of $B_{\delta}^\infty(\omega_j)$, and denote by $Q$ be the finite set of vertices of $\cup_j\overline B^\infty_{3\delta}(\omega_j)$.  Since $Q$ is included in the convex hull of $E$, each point in $Q$ can be written as a convex combination of elements of $E$.  We conclude that there exists a finite set $Z=\{z_1,\dots,z_m\}$ of points in $E$ whose convex hull includes $B_{3\delta}^\infty(\omega_j)$ for any $j$.

Let $B_i=B_\delta^\infty(z_i)$.  Since $B_i$ is an open neighbourhood of $z_i\in E=\mathrm{supp}(\mu)$, the Portmanteau lemma implies that when $n$ is large, $\mu_n(B_i)>\epsilon_i=\mu(B_i)/2$ for any $i=1,\dots,m$.  Let $\epsilon=\min_i\epsilon_i>0$.  Since $\{\nu_n\}$ is a tight sequence, there exists a compact set $K_\epsilon$ such that $\nu_n(K_\epsilon)>1-\epsilon$ for any integer $n$.  In particular, there exist $x_{ni}\in B_i$ and $y_{ni}\in u_n(x_{ni})$ such that $y_{ni}\in  K_\epsilon$.  Application of Lemma~\ref{lem:linfchull} to
\[
\tilde Z = X_n = \{x_{n1},\dots,x_{nm}\}
\]
and noticing that by definition $\|x_{ni} - z_i\|_\infty\le \delta$ yields
\[
\mathrm{conv}(X_n) = \mathrm{conv}(\{x_{n1},\dots,x_{nm}\})
\supseteq B_{3\delta-\delta}^\infty(\omega_j)
= B_{2\delta}^\infty(\omega_j)
\qquad \textrm{for all }j.
\]
For each $\omega\in \Omega$ there exists $j$ such that $\|\omega-\omega_j\|_\infty\le\delta$, so that $\mathrm{conv}(X_n)\supseteq B_\delta^\infty(\omega)\supseteq B_\delta(\omega)$, since $\ell_2$-balls are smaller than $\ell_\infty$-balls.  Summarising: $\mathrm{conv}(X_n)\supseteq B_\delta(\Omega)$.

By \cite[Lemma~1.2(4)]{albamb} it follows that for any $\omega\in\Omega$ and any $y_0\in u_n(\omega)$,
\[
\|y_0\|
\le \frac{[\sup_{x,z\in X_n}\|x - z\|  ]  [  \max_{x\in X_n}\inf_{y\in u_n(x)}\|y\|   ]}   {d(\omega, \mathbb R^d\setminus \mathrm{conv}(X_n))}
\le \frac1\delta \left[\sup_{k,l}\|x_{nk} - x_{nl}\|  \right]  \left[  \max_i\inf_{y\in u_n(x_{ni})}\|y\|  \right].
\]
Now observe that the infimum at the right-hand side is bounded by $\|y_{ni}\|\le\sup_{y\in K_\epsilon}\|y\|$.  Furthermore, $\|x_{nk} - x_{nl}\|\le 2\sqrt d\delta + \|z_k - z_l\|$.  Hence
\[
\forall\omega\in\Omega
\quad \forall y_0\in u_n(\omega):
\qquad
\|y_0\|
\le \frac1\delta \left(2\sqrt d\delta + \max_{k,l}\|z_k - z_l\|\right)  \sup_{y\in K_\epsilon}\|y\|,
\]
and the right-hand side is independent of $n$.  We may therefore conclude that for $n$ large enough, $u_n(\Omega$) stays in a compact set;  it is nonempty by \cite[Corollary~1.3(2)]{albamb}.
\end{proof}

\begin{proof}[Proof of Proposition~\ref{prop:unifcompacts}]  By Proposition~\ref{lem:yinkstrong} when $n>N_\Omega$ is large, $u_n(x)\ne\emptyset$ for all $x\in \Omega$ and
\[
\sup_{x\in \Omega}\sup_{y\in u_n(x)}\|y\|
\le C_{\Omega,d}
<\infty
,\qquad n>N_\Omega,
\]
where $C_{\Omega,d}$ is a constant that depends only on $\Omega$ (and the dimension $d$).

Suppose that the converse is true, and uniform convergence does not hold.  Then there exist $\epsilon>0$ and subsequences $y_{n_k}\in u_{n_k}(x_{n_k})$ such that $x_{n_k}\in \Omega$ and
\[
\|y_{n_k} - u(x_{n_k})\|
>\epsilon
,\qquad k=1,2,\dots.
\]
The $x_{n_k}$'s lie in the compact set $\Omega$, whereas by Proposition~\ref{lem:yinkstrong} the $y_{n_k}$'s lie in the ball of radius $C_{\Omega,d}$ centred at the origin.  Therefore, up to the extraction of a subsequence, we have $x_{n_k}\to x\in\Omega$ and $y_{n_k}\to y$. By Lemma~\ref{lem:liminT}, $y=u(x)$. But $u$ is continuous at $x$ (Proposition~\ref{prop:cont}), whence
\[
\epsilon
<\|y_{n_k} - u(x_{n_k})\|
\le\|y_{n_k}-y\|  +  \|y - u(x)\|  +  \|u(x) - u(x_{n_k})\|
\to0
,\qquad k\to\infty,
\]
a contradiction.
\end{proof}

\section{Some Examples}\label{SEC:EX}
As an illustration, we implement Algorithm~\ref{algo} in several settings for which pairwise optimal maps can be calculated explicitly at every iteration, allowing for fast computation without error propagation. Indeed, these settings allow for stronger convergence statements to be made on a case-by-case basis. More details on the calculations and properties of each individual scenario are given in Section~\ref{zempan16supp}.

\subsection{The case $d=1$}  When the measures are supported on the real line, the optimal maps have the explicit expression given in Equation~\eqref{eq:optimal1d} and one may apply Algorithm~\ref{algo} starting from one of these measures.  Figure~\ref{fig:F-8S80den} plots $N=4$ univariate densities and the Fr\'echet mean yielded by the algorithm in two different scenarios.  At the left, the densities were generated as
\begin{equation}\label{eq:bigaussden}
f^i(x)
=\frac12\phi\left(\frac{x - m^i_1}{\sigma^i_1}\right)
+\frac12\phi\left(\frac{x - m^i_2}{\sigma^i_2}\right)
,\qquad
\end{equation}
with $\phi$ the standard normal density, and the parameters generated independently as
\[
m^i_1\sim U[-13, -3],\quad
m^i_2\sim U[3, 13],\quad
\sigma^i_1,\sigma^i_2\sim Gamma(4, 4).
\]
At the right of Figure~\ref{fig:F-8S80den}, we used a mixture of a shifted gamma and a Gaussian:
\begin{equation}\label{eq:bigaussgammaden}
f^i(x)
=\frac 35\frac{\beta_i^3}{\Gamma(3)}(x - m^i_3)^2e^{-\beta_i(x-3)}
+\frac 25\phi(x - m^i_4),
\end{equation}
with
\[
\beta^i\sim Gamma(4, 1)
,\quad m^i_3\sim U[1,4]
,\quad m^i_4\sim U[-4,-1].
\]
The resulting Fr\'echet mean density for both settings is shown in thick light blue, and can be seen to capture the bimodal nature of the data.  Even though the Fr\'echet mean of Gaussian mixtures is not a Gaussian mixture itself, it is approximately so, provided that the peaks are separated enough. Figure~\ref{vector_fields}(a) shows the Procrustes maps pushing the Fr\'echet mean $\bar\mu$ to the measures $\mu^1,\dots,\mu^N$ in each case.  If one ignores the ``middle part" of the $x$ axis, the maps appear (approximately) affine for small values of $x$ and for large values of $x$, indicating how the peaks are shifted. In the middle region, the maps need to ``bridge the gap" between the different slopes and intercepts of these affine maps.

\begin{figure}
\includegraphics[trim=10mm 16mm 10mm 20mm, clip, scale = 0.42]{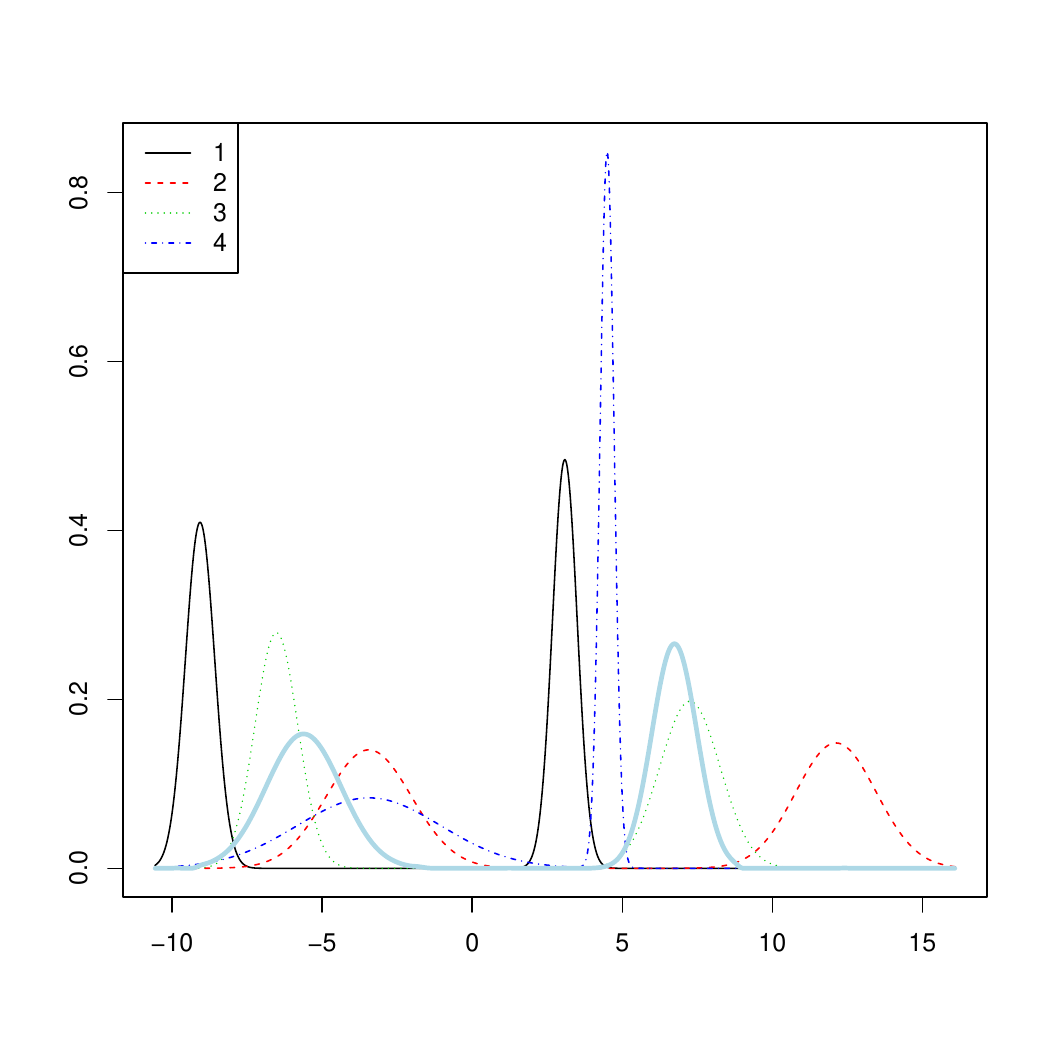}
\includegraphics[trim=10mm 16mm 10mm 20mm, clip, scale = 0.42]{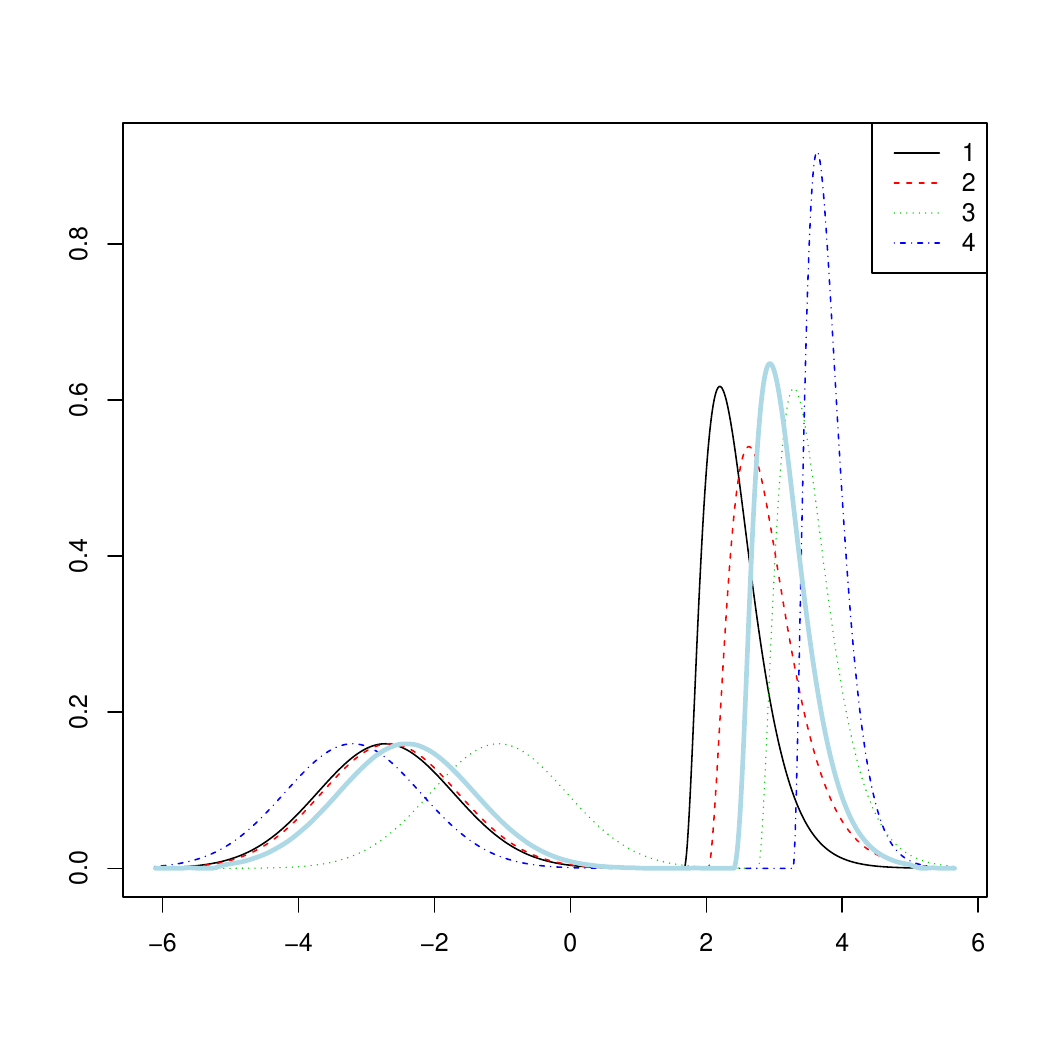}
\caption{Densities of bimodal Gaussian mixture (left) and a mixture of Gaussian with gamma (right), with the Fr\'echet mean density in light blue.}
\label{fig:F-8S80den}
\end{figure}

\subsection{Independence}  We next take measures $\mu^i$ on $\mathbb R^2$, having independent marginal densities $f_X^i$ as in \eqref{eq:bigaussden}, and $f_Y^i$ as in \eqref{eq:bigaussgammaden}.  Figure~\ref{fig:indS80} shows the density plot of $N=4$ such measures, constructed as the product of the measures from Figure~\ref{fig:F-8S80den}.  One can distinguish the independence by the ``parallel" structure of the figures:  for every pair $(y_1,y_2)$, the ratio $g(x,y_1)/g(x,y_2)$ does not depend on $x$ (and vice versa, interchanging $x$ and $y$).  Figure~\ref{fig:indS80bar} plots the density of the resulting Fr\'echet mean.  We observe that the Fr\'echet mean captures the four peaks, and their location.  Furthermore, the parallel nature of the figure is preserved in the Fr\'echet mean.  Indeed, we prove in the supplement (Section~\ref{zempan16supp}) that, unsurprisingly, the Fr\'echet mean is a product measure.

\begin{figure}
\includegraphics[trim=10mm 37mm 15mm 31mm, clip, scale = 0.25]{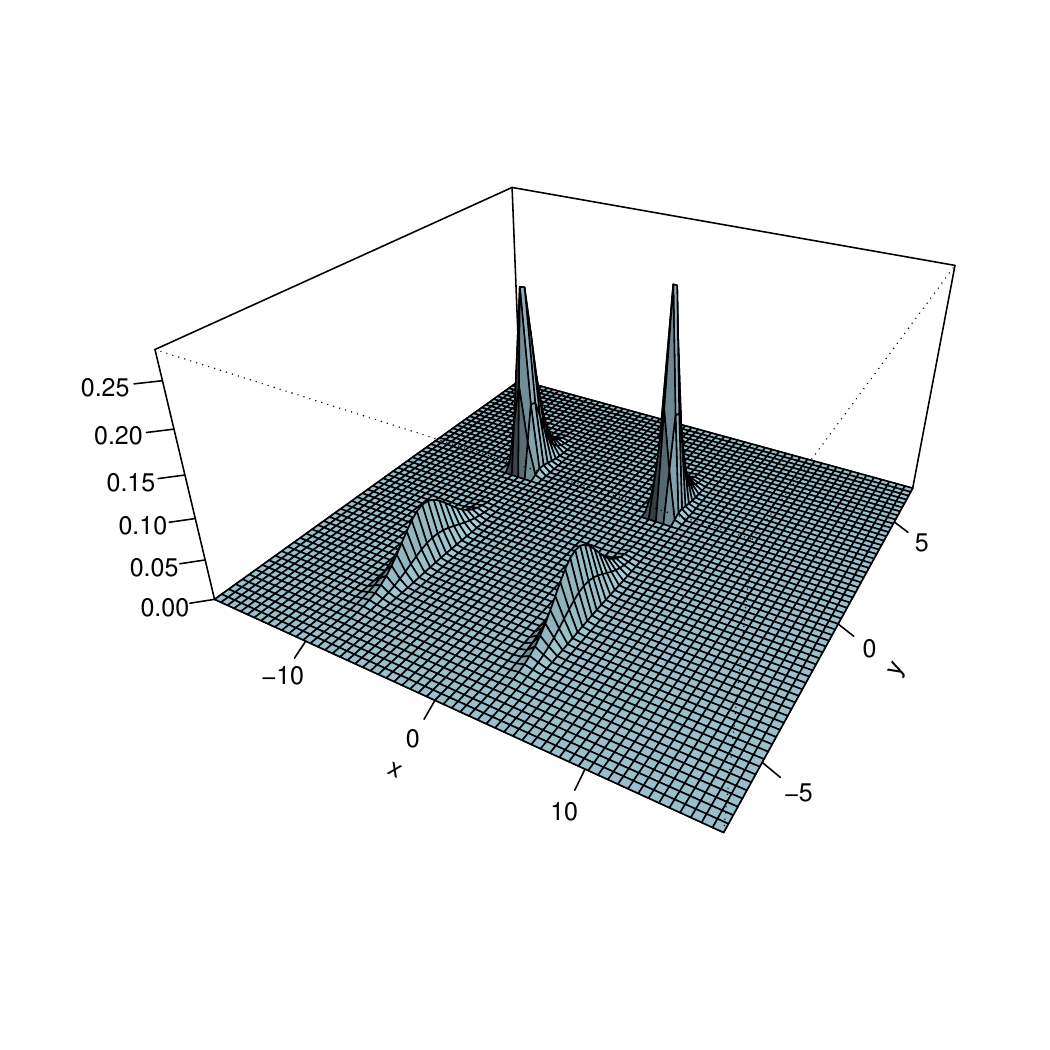}
\includegraphics[trim=10mm 37mm 15mm 31mm, clip, scale = 0.25]{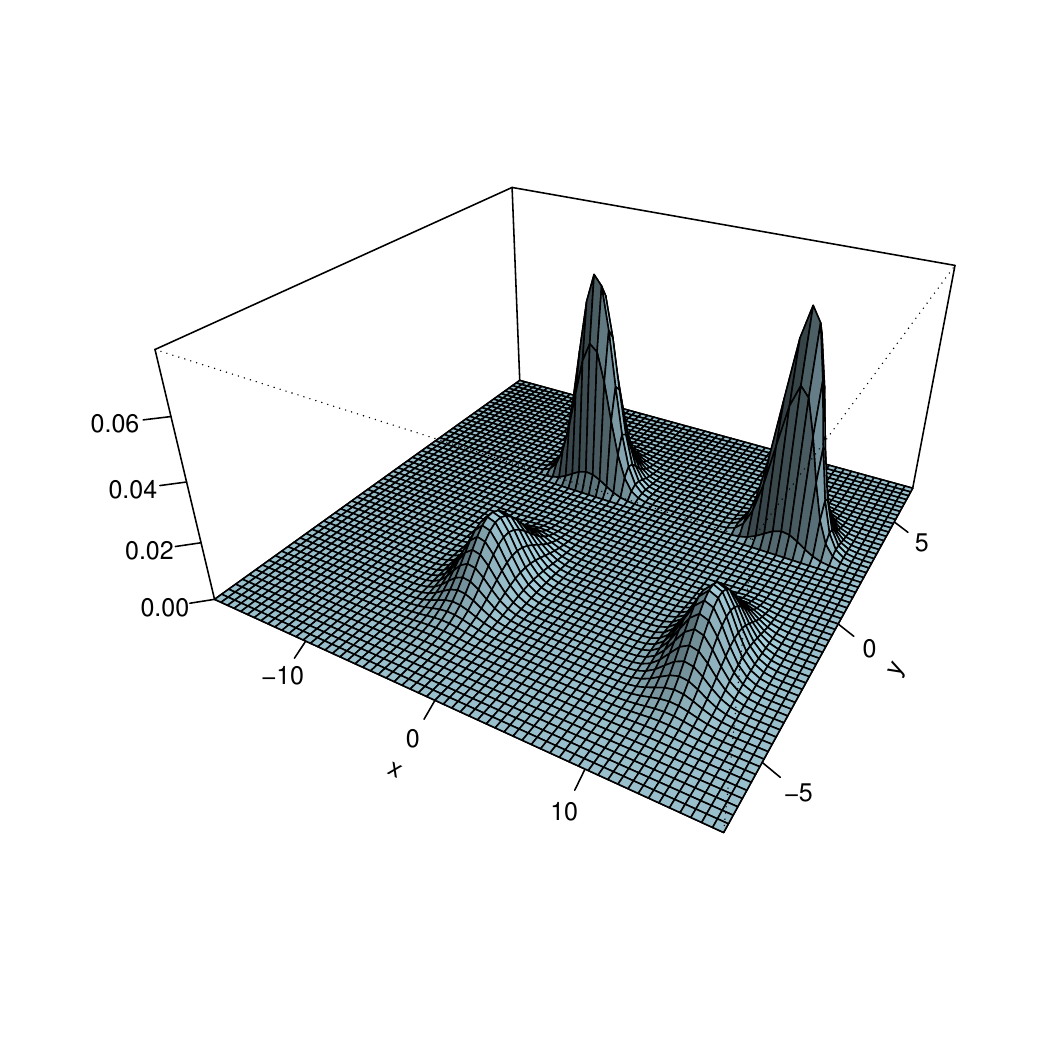}\\
\includegraphics[trim=10mm 37mm 15mm 31mm, clip, scale = 0.25]{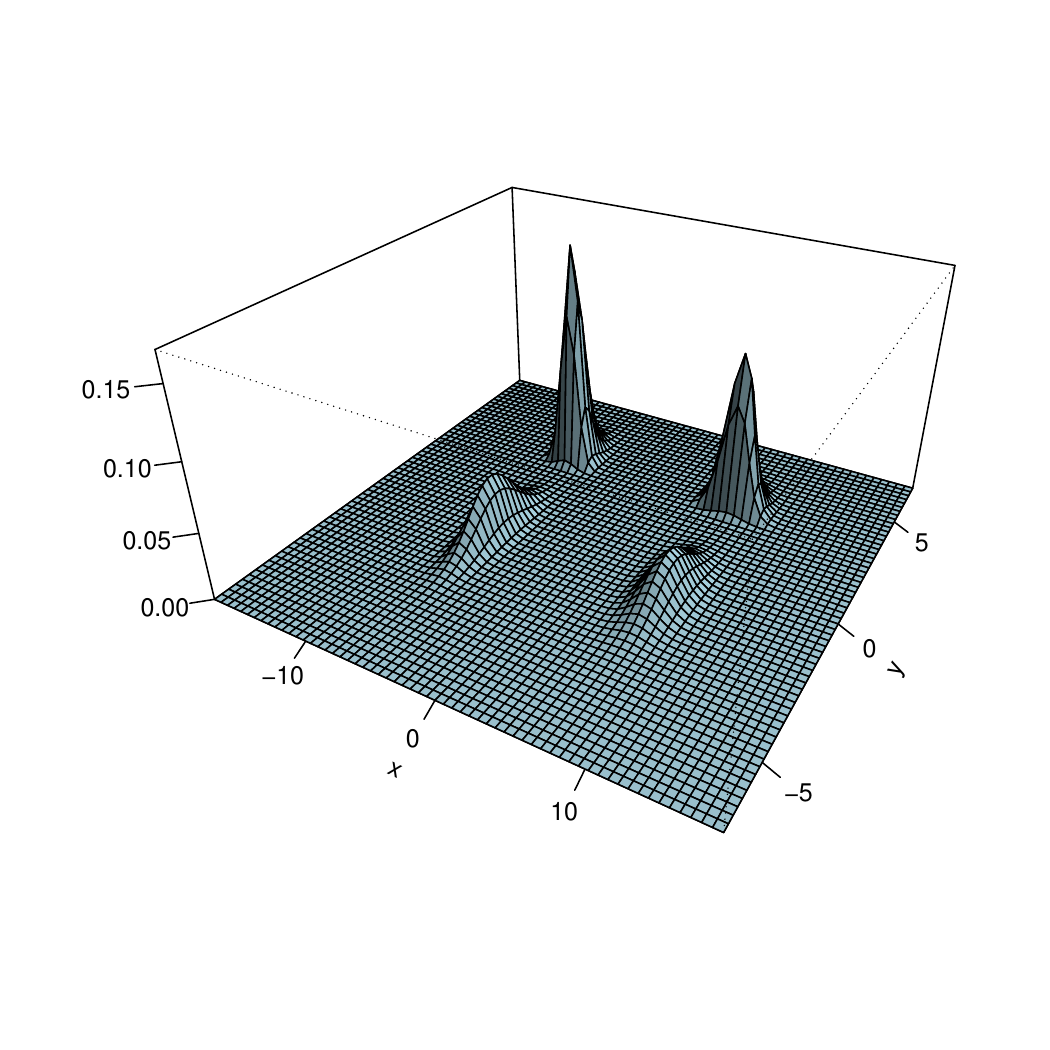}
\includegraphics[trim=10mm 37mm 15mm 31mm, clip, scale = 0.25]{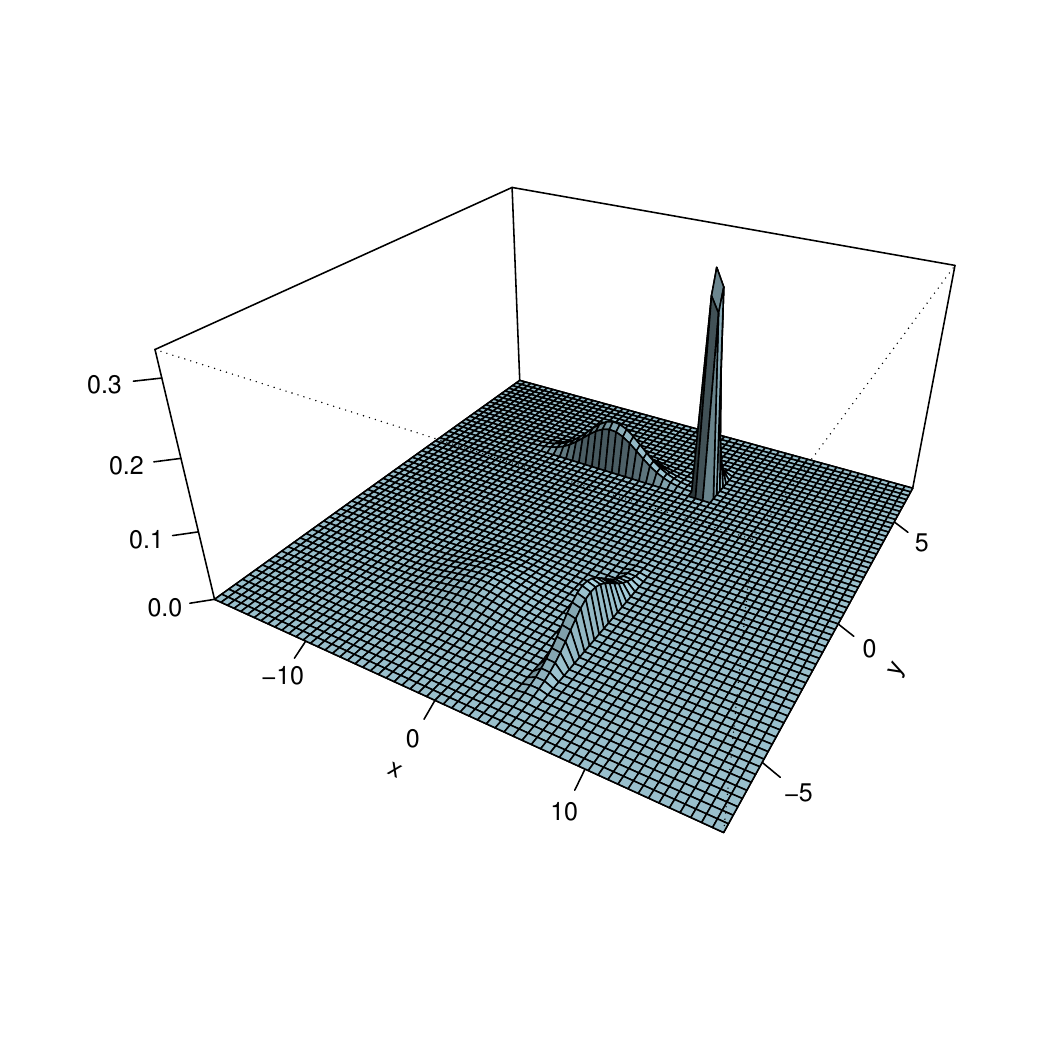}
\caption{Density plots of the four product measures of the measures in Figure~\ref{fig:F-8S80den}.}
\label{fig:indS80}
\end{figure}

\begin{figure}
\includegraphics[trim=10mm 37mm 15mm 31mm, clip, scale = 0.25]{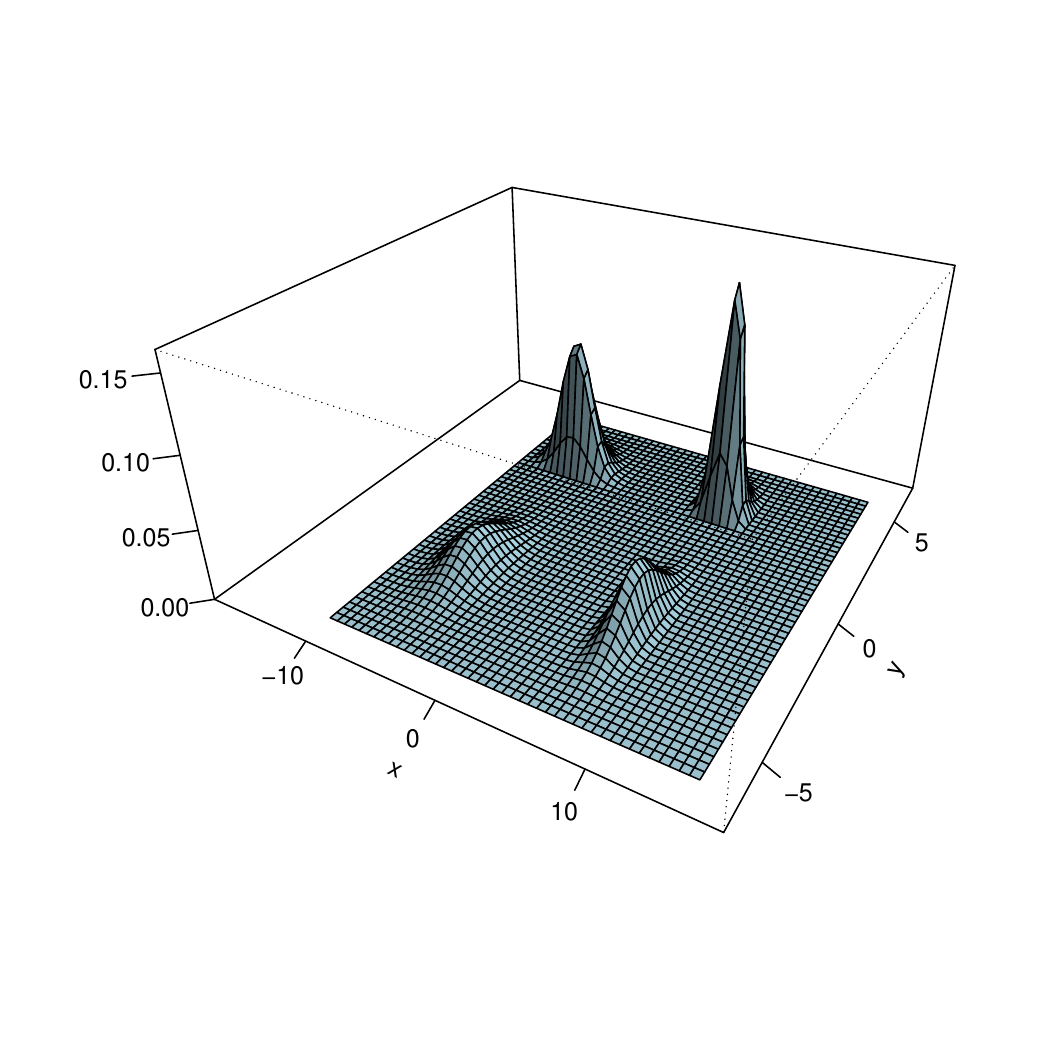}
\caption{Density plot of the Fr\'echet mean of the measures in Figure~\ref{fig:indS80}.}
\label{fig:indS80bar}
\end{figure}

\subsection{Common Copulas}  Let $\mu^i$ be a measure on $\mathbb R^2$ with density
\[
g^i(x,y)
=c(F_X^i(x),F_Y^i(y))f_X^i(x)f_Y^i(y),
\]
where $f_X^i$ and $f_Y^i$ are random densities on the real line with distribution functions $F_X^i$ and $F_Y^i$, and $c$ is a copula density.  Figure~\ref{fig:F-8S80} shows the density plot of $N=4$ such measures, with $f_X^i$ generated as in \eqref{eq:bigaussden}, $f_Y^i$ as in \eqref{eq:bigaussgammaden}, and $c$ is the Frank($-8$) copula density, while Figure~\ref{fig:F-8S80bar} plots the density of the Fr\'echet mean obtained.  (For ease of comparison we use the same realisations of the densities that appear in Figure~\ref{fig:F-8S80den}.)  The Fr\'echet mean can be seen to preserve the shape of the density, having four clearly distinguished peaks.  Figure \ref{vector_fields}(b), depicting the resulting Procrustes maps, allows for a clearer interpretation:  for instance the leftmost plot (in black) shows more clearly that the map splits the mass around $x=-2$ to a much wider interval;  and conversely a very large amount mass is sent to $x\approx 2$.  This rather extreme behaviour matches the peak of the density of $\mu^1$ located at $x=2$.

\begin{figure}
\includegraphics[trim=10mm 37mm 15mm 31mm, clip, scale = 0.25]{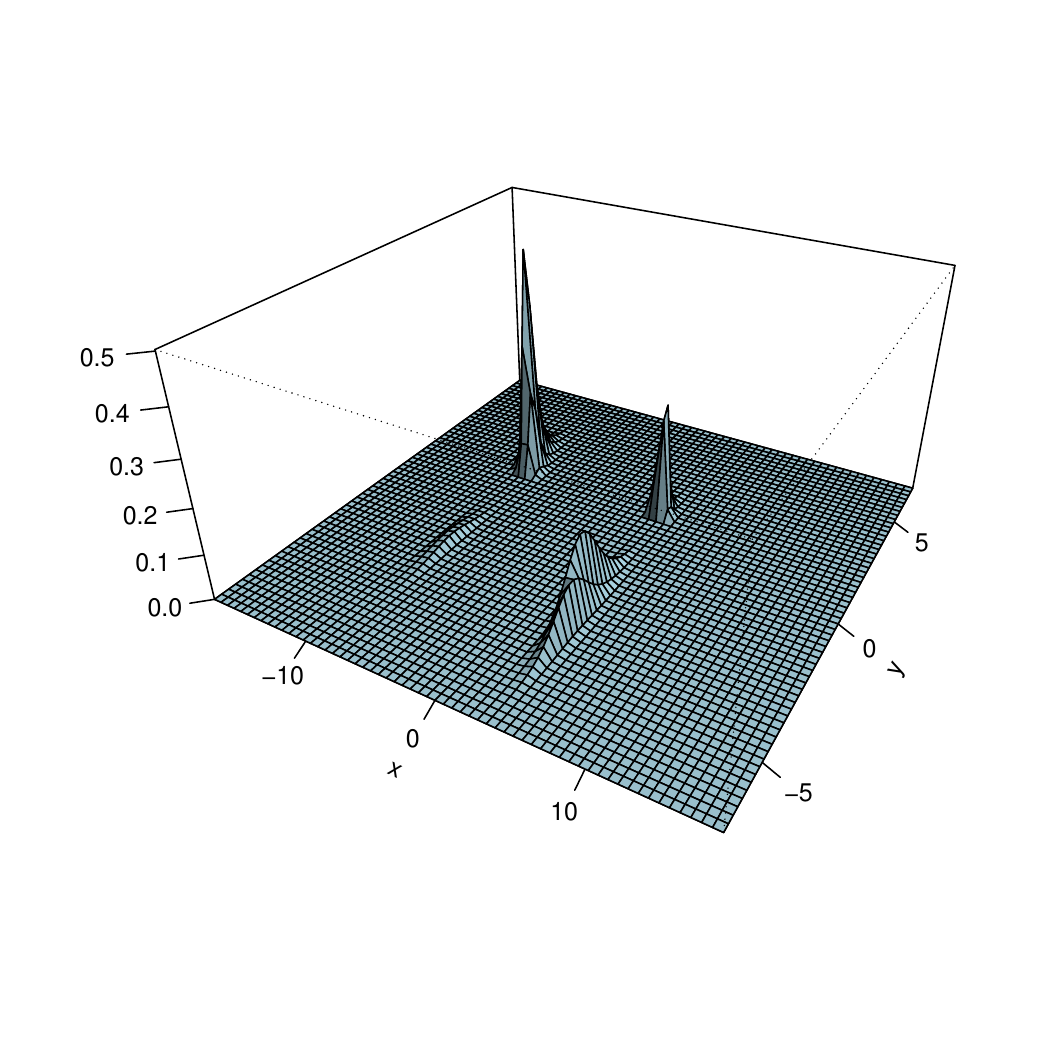}
\includegraphics[trim=10mm 37mm 15mm 31mm, clip, scale = 0.25]{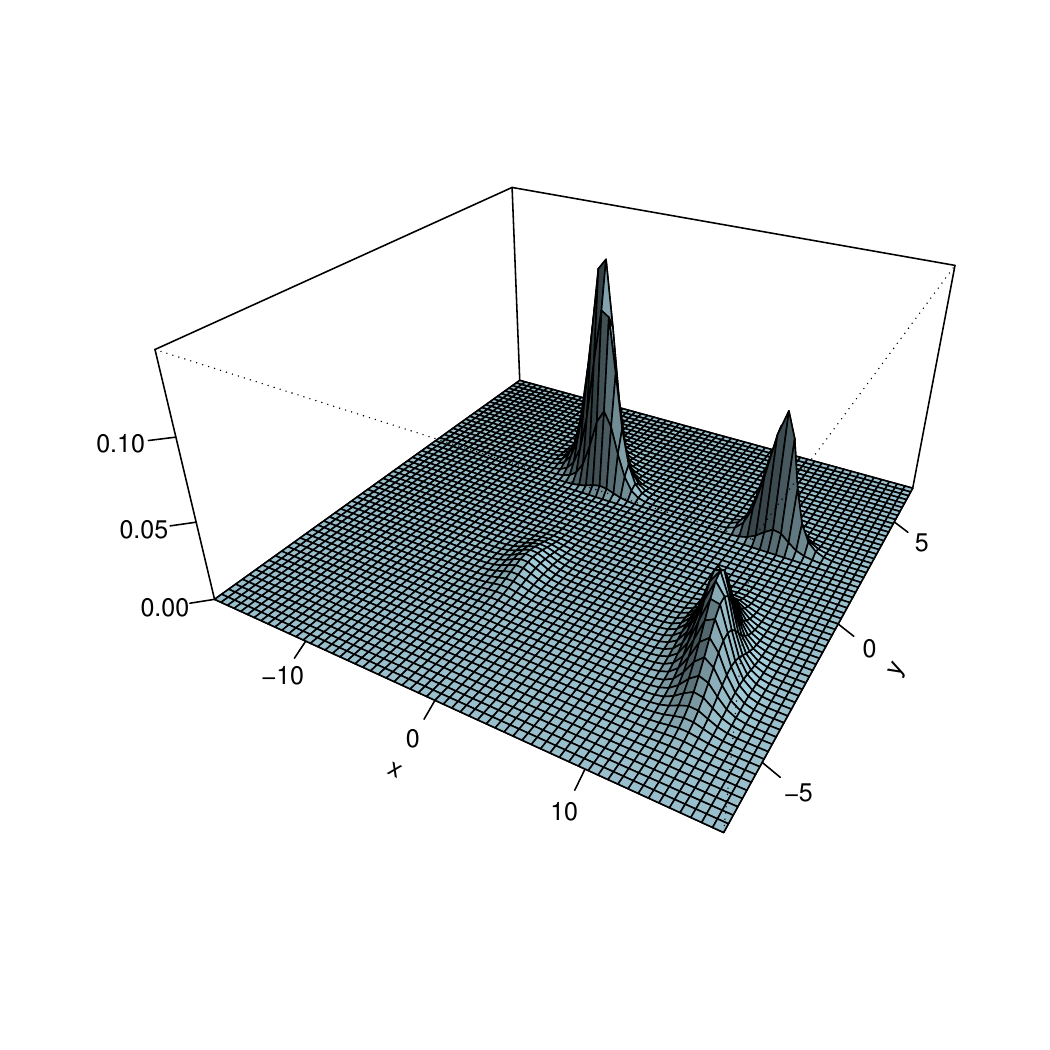}\\
\includegraphics[trim=10mm 37mm 15mm 31mm, clip, scale = 0.25]{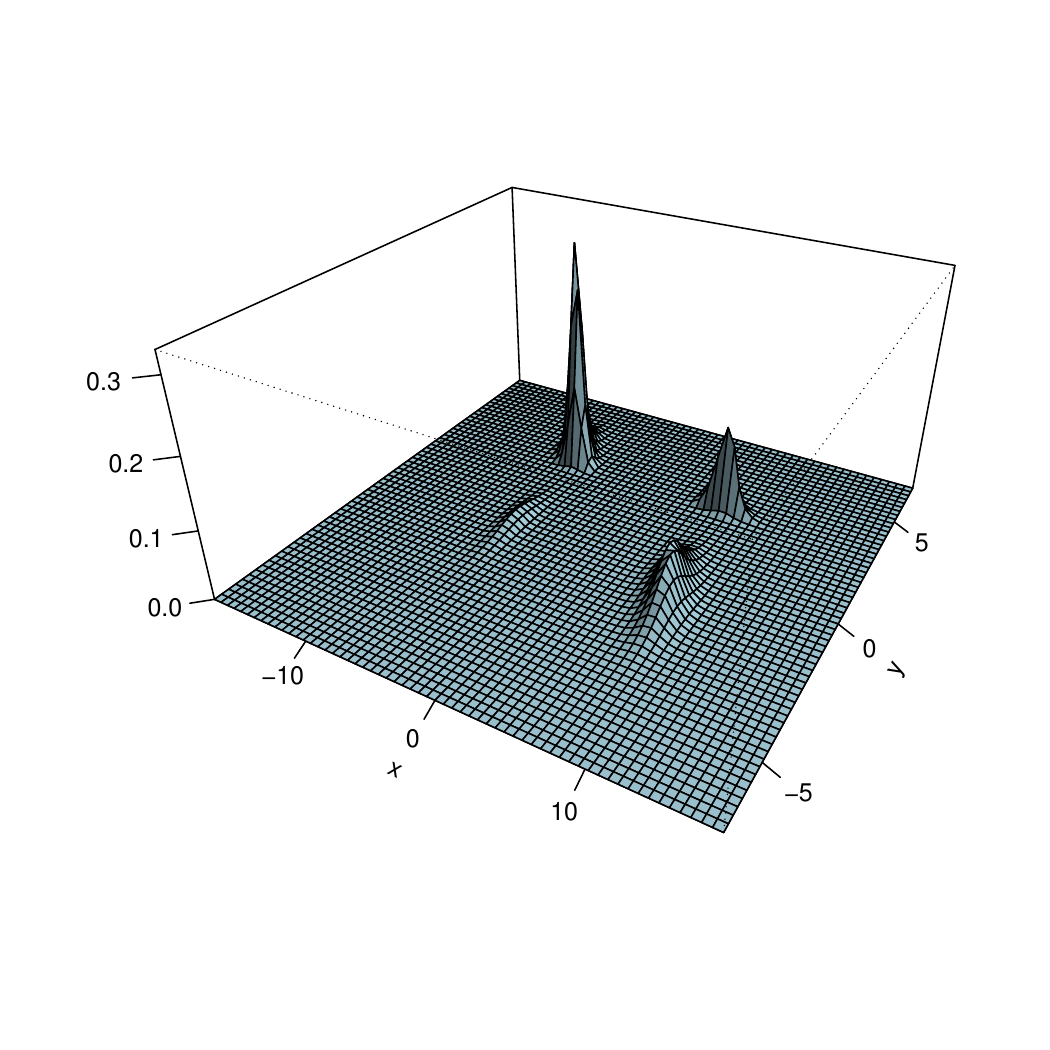}
\includegraphics[trim=10mm 37mm 15mm 31mm, clip, scale = 0.25]{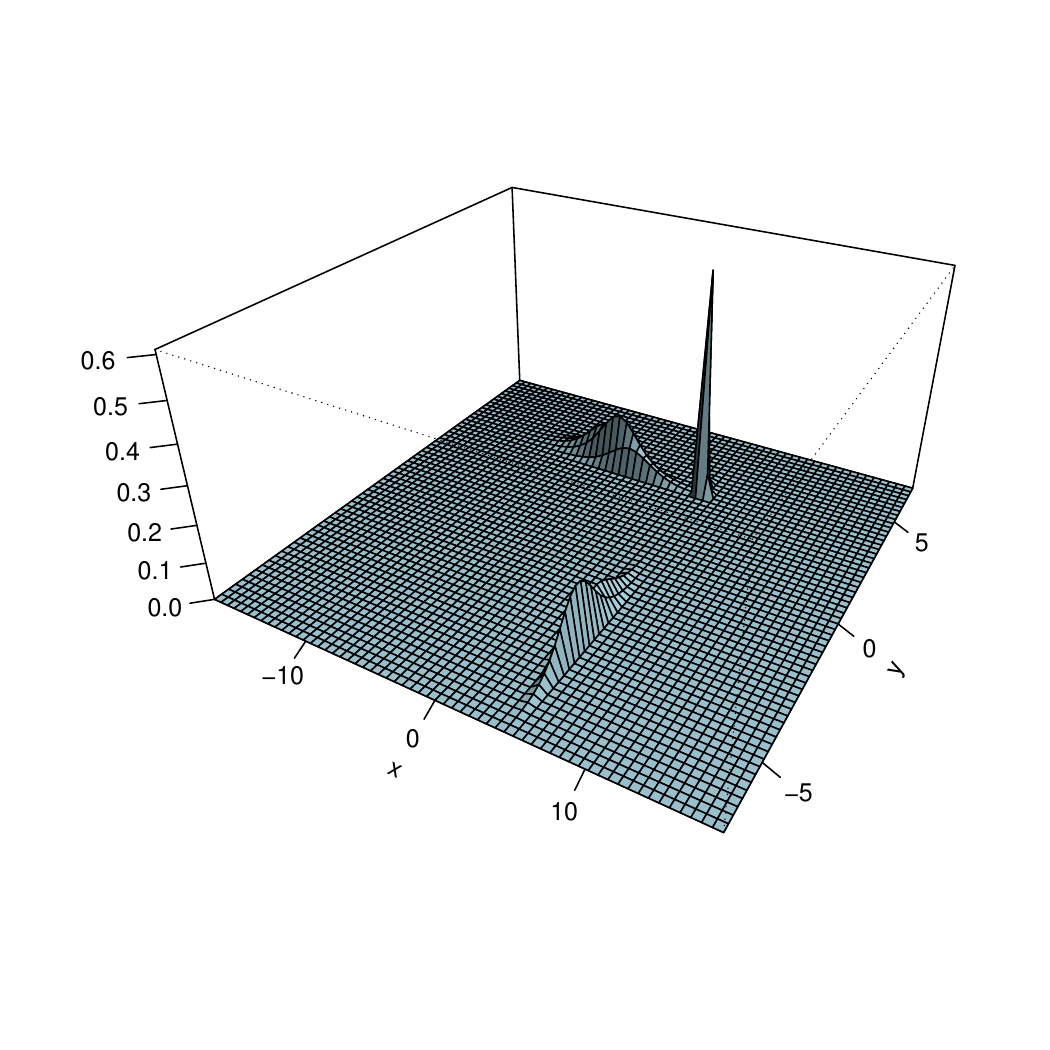}
\caption{Density plots of four measures in $\mathbb R^2$ with Frank copula of parameter $-8$.}
\label{fig:F-8S80}
\end{figure}

\begin{figure}
\includegraphics[trim=10mm 37mm 15mm 31mm, clip, scale = 0.25]{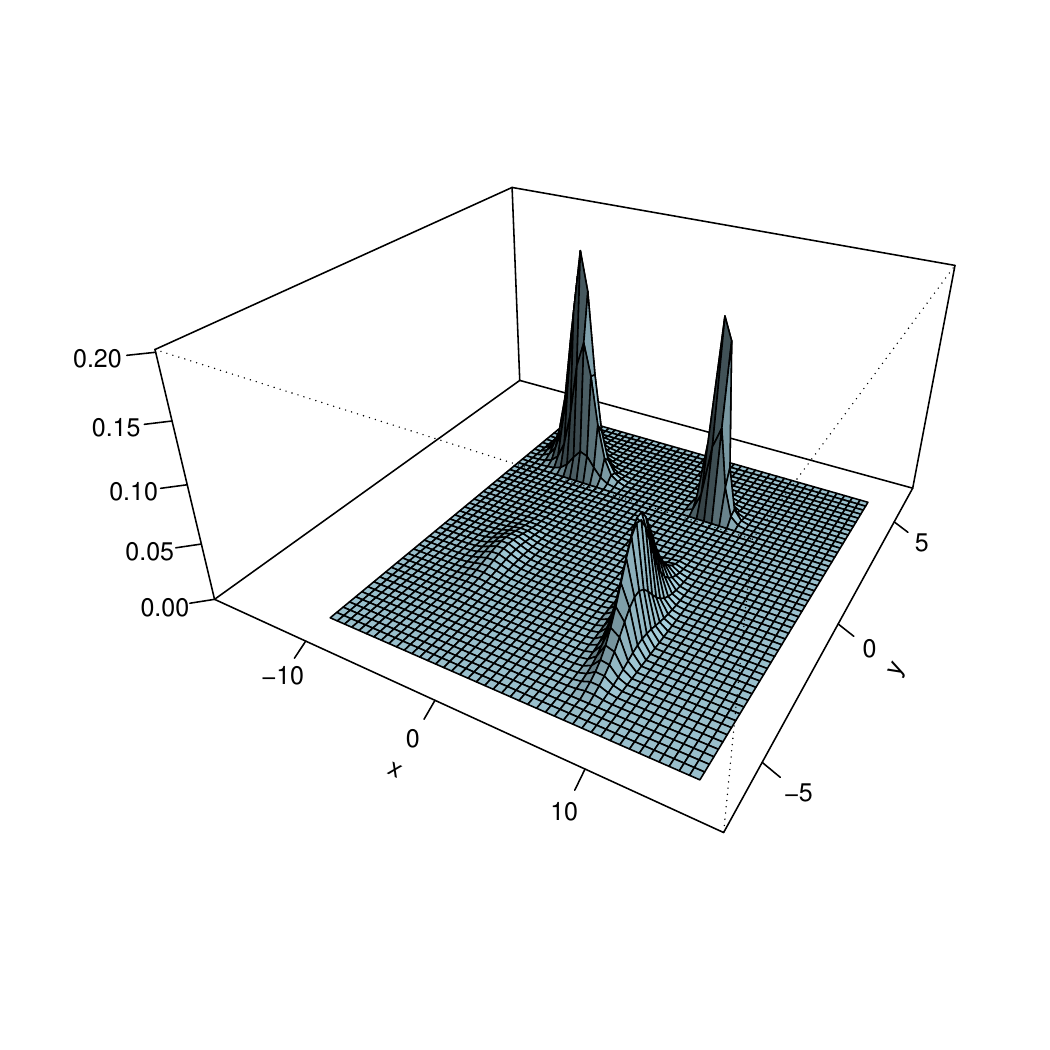}
\caption{Density plot of the Fr\'echet mean of the measures in Figure~\ref{fig:F-8S80}.}
\label{fig:F-8S80bar}
\end{figure}

The first three scenarios are examples of situations where the measures $\{\mu^i\}$ are \textit{compatible with each other} in the sense that $\mathbf t^{\mu^k}_{\mu^j}\circ \mathbf t_{\mu^i}^{\mu^j}=\mathbf t_{\mu^i}^{\mu^k}$.  Boissard et al.\ \cite{boissard2015distribution} tackle the problem of finding the Fr\'echet mean in such a setting, by means of the \textit{iterated barycentre}.  In the supplementary material (Section~\ref{zempan16supp}) we show that Algorithm~\ref{algo} will always converges to the Fr\'echet mean, provided the initial point $\gamma_0$ is compatible with $\{\mu^i\}$ (for instance, if $\gamma_0=\mu^i$).  In fact, we show that convergence is established after a single iteration of the algorithm. Since optimal maps are gradients of convex potentials, they must have positive definite derivatives.  Under regularity conditions, compatibility is essentially equivalent to the commutativity of the $d\times d$ matrices $\nabla\mathbf t_{\mu^j}^{\mu^k}(\mathbf t_{\mu^i}^{\mu^j}(x))$ and $\nabla\mathbf t_{\mu^i}^{\mu^j}(x)$ for $\mu^i$-almost any $x$.  We next discuss examples where this condition fails.

\subsection{Gaussian measures}  Suppose that each $\mu^i$ follows a non-degenerate multivariate Gaussian distribution with mean $0$ and covariance matrix $S_i$.  The optimal maps are known to be linear and admit the explicit formula (Dowson \& Landau \cite{dowson1982frechet};  Olkin \& Pukelsheim \cite{olkin1982distance})
\[
\mathbf t_i^j
=S_j^{1/2}[S_j^{1/2}S_iS_j^{1/2}]^{-1/2}S_j^{1/2}.
\]
If the initial point $\gamma_0$ is another Gaussian measure with covariance matrix $\Gamma_0$, then by the linearity of the maps one sees that $\gamma_k\sim\mathcal N(0,\Gamma_k)$ for some positive definite $\Gamma_k$.  Thus, one can calculate the optimal maps at each iteration;  in the supplement (Section~\ref{zempan16supp}) we prove that $\gamma_k$ must converge to the unique Fr\'echet mean, which is also a Gaussian measure. This example is also studied
independently  in \'Alvarez-Esteban et al. \cite[Section 4]{alvarez2016fixed}, where an alternative proof can be found. Our proof is shorter and arguably simpler, but the proof in \cite{alvarez2016fixed} shows the additional property that the traces of the matrix iterates are monotonically increasing.

Notice that the Gaussian measures $\{\mu^i\}$ will be compatible if $S_iS_j=S_jS_i$, but they might well fail to be. Thus, the algorithm does not converge in one step.  We observed, however, rapid convergence of the iterates of Algorithm~\ref{algo} to the Fr\'echet mean, even for rather large values of $N$ and $d$. Figure~\ref{fig:normalsS22} shows density plots of $N=4$ centred Gaussian measures on $\mathbb R^2$ with covariances $S_i\sim \mathrm{Wishart}(I_2,2)$, and Figure~\ref{fig:normalbarS22} shows the density of the resulting Fr\'echet mean.  In this particular example, the algorithm needed 11 iterations starting from the identity matrix. The corresponding Procrustes registration maps are displayed in Figure~\ref{vector_fields}(c).  It is apparent from the figure that these maps are linear, and after a more careful reflection one can be convinced that their average is the identity.  The four plots in the figure are remarkably different, in accordance with the measures themselves having widely varying condition numbers and orientations;  $\mu^3$ and more so $\mu^4$ are very concentrated, so the registration maps ``sweep" the mass towards zero.  In contrast, the registration maps to $\mu^1$ and $\mu^2$ spread the mass out away from the origin.

\begin{figure}
\includegraphics[trim=10mm 37mm 15mm 31mm, clip, scale = 0.25]{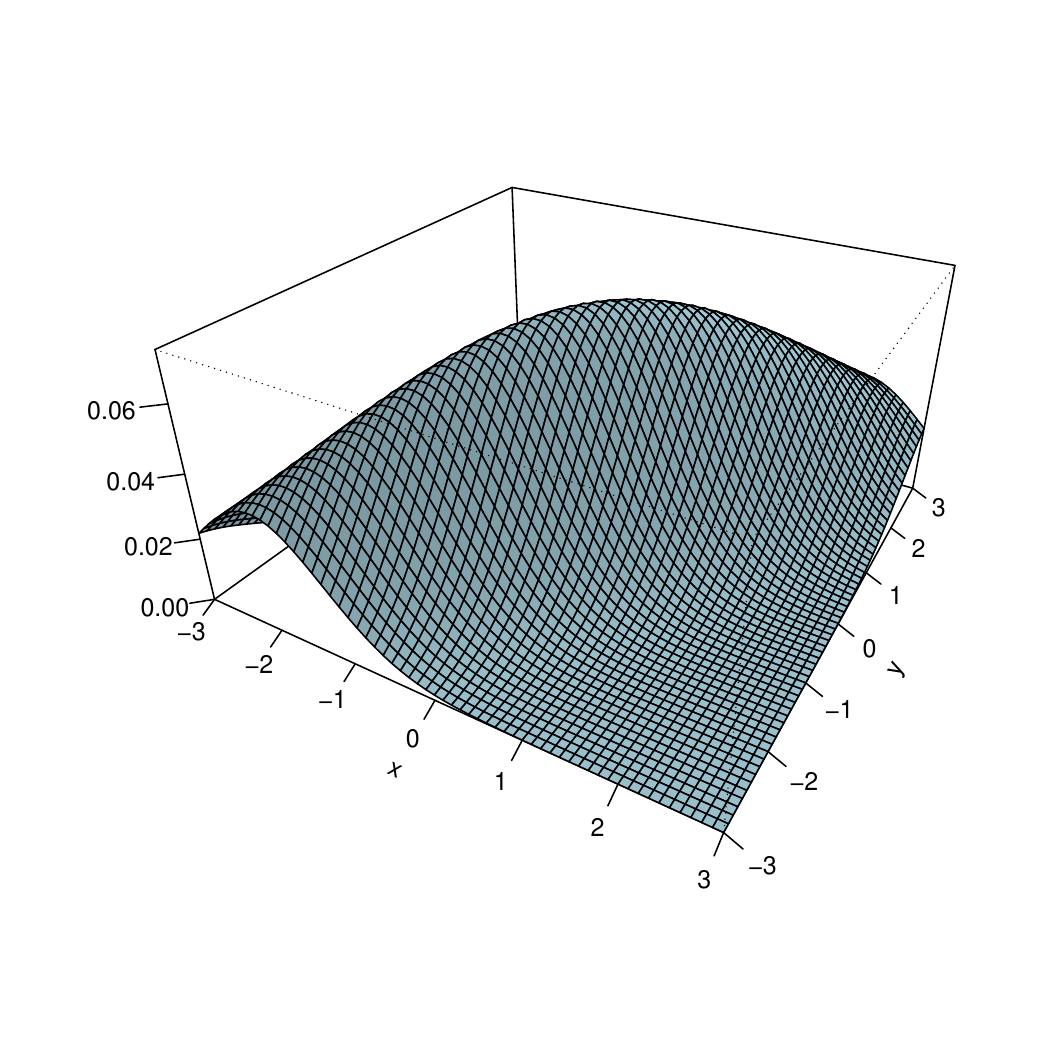}
\includegraphics[trim=10mm 37mm 15mm 31mm, clip, scale = 0.25]{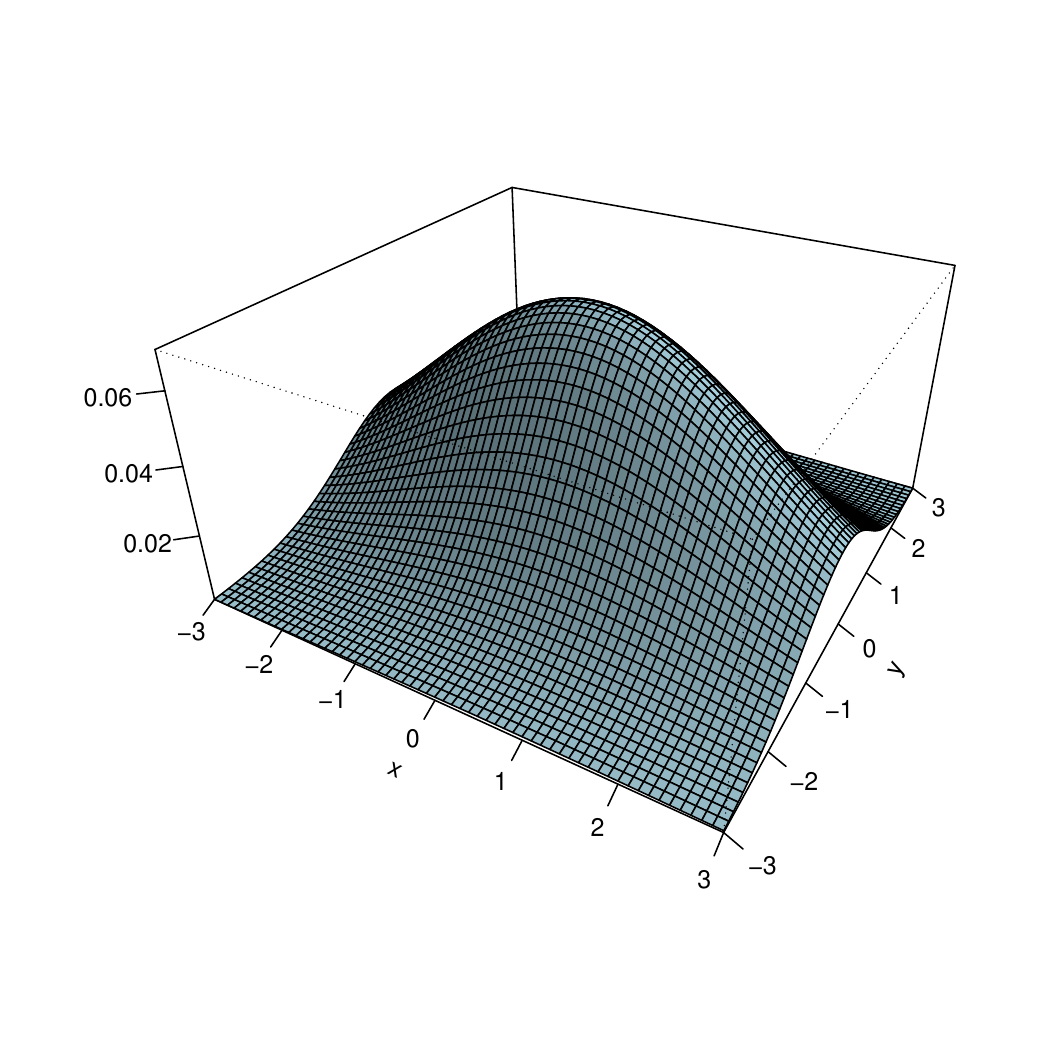}\\
\includegraphics[trim=10mm 37mm 15mm 31mm, clip, scale = 0.25]{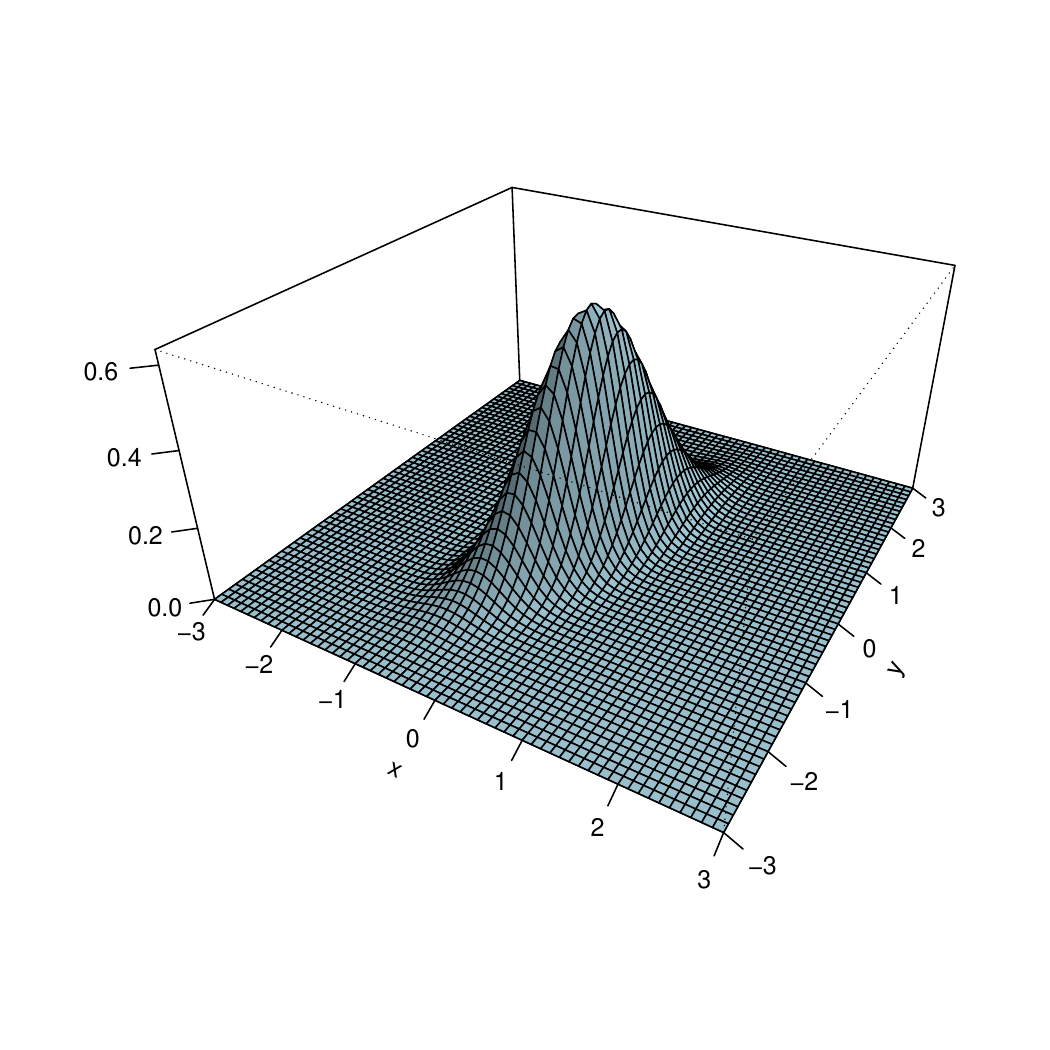}
\includegraphics[trim=10mm 37mm 15mm 31mm, clip, scale = 0.25]{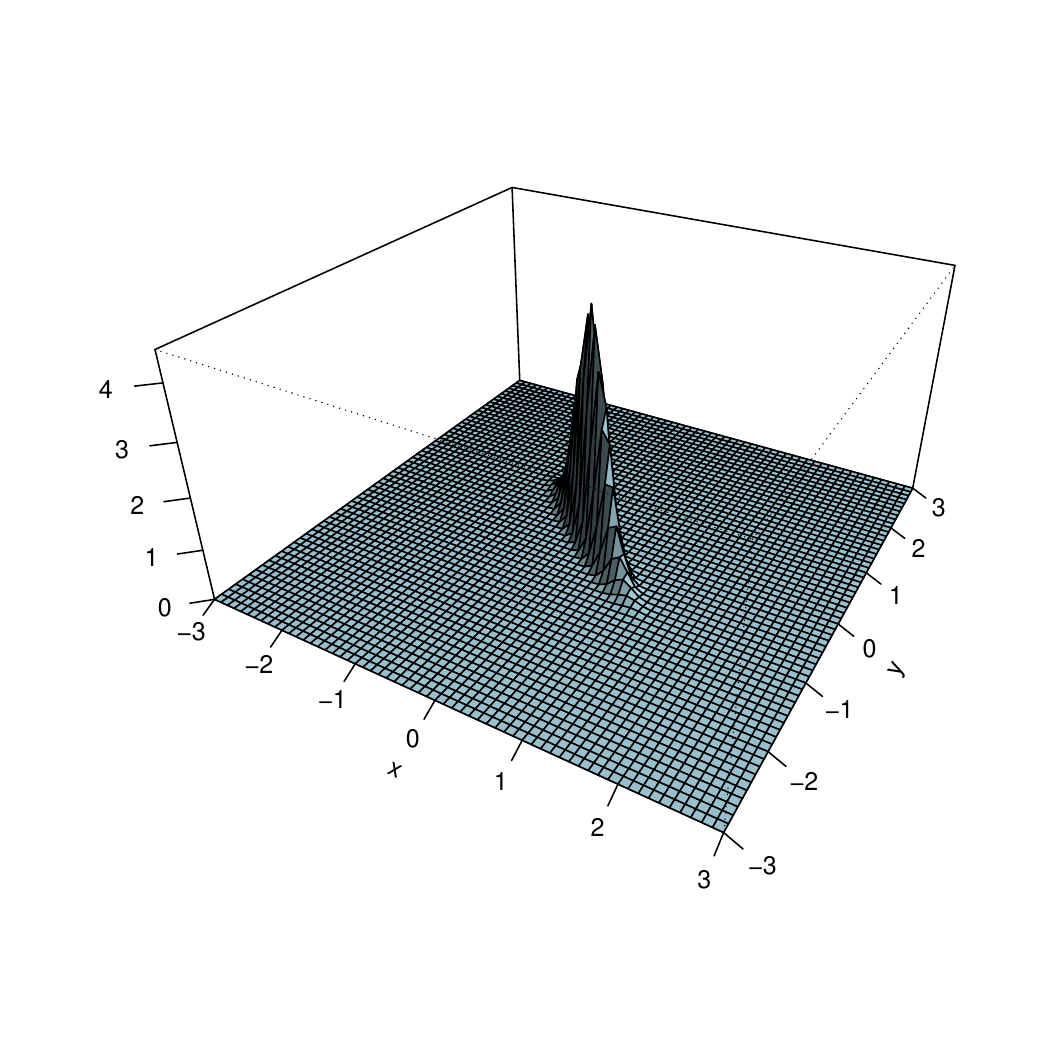}
\caption{Density plot of four Gaussian measures in $\mathbb R^2$.}
\label{fig:normalsS22}
\end{figure}

\begin{figure}
\includegraphics[trim=10mm 37mm 15mm 31mm, clip, scale = 0.25]{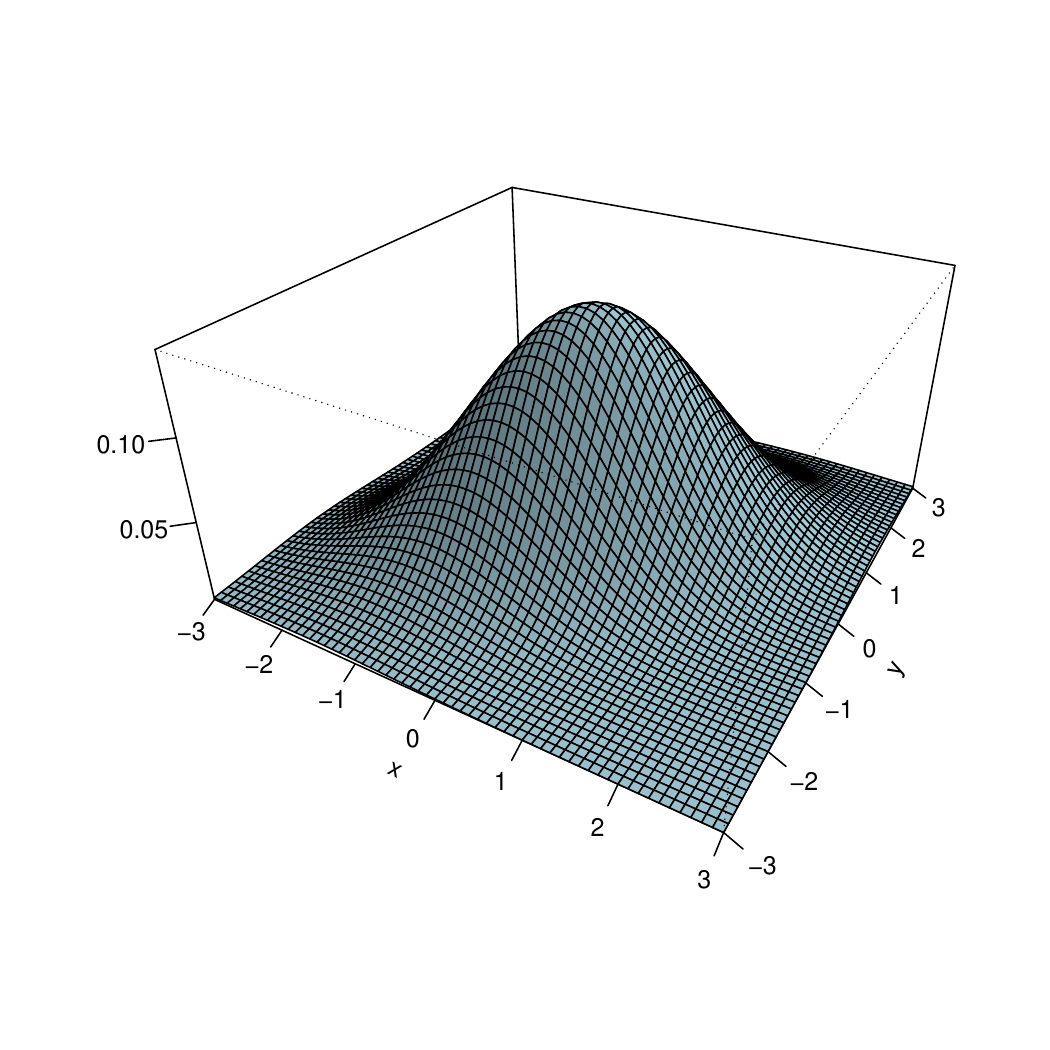}
\caption{Density plot of the Fr\'echet mean of the measures in Figure~\ref{fig:normalsS22}.}
\label{fig:normalbarS22}
\end{figure}

\begin{figure}
\centering

\subfigure[One-dimensional example: Procrustes registration maps $\mathbf t_{\bar\mu}^{\mu^i}$ from the Fr\'echet mean $\bar\mu$ to the four measures $\{\mu^i\}$ in Figure~\ref{fig:F-8S80den}. The left plot corresponds to the bimodal Gaussian mixture, and the right plot to the Gaussian/gamma mixture.]{
\includegraphics[trim=10mm 16mm 10mm 20mm, clip, scale = 0.42]{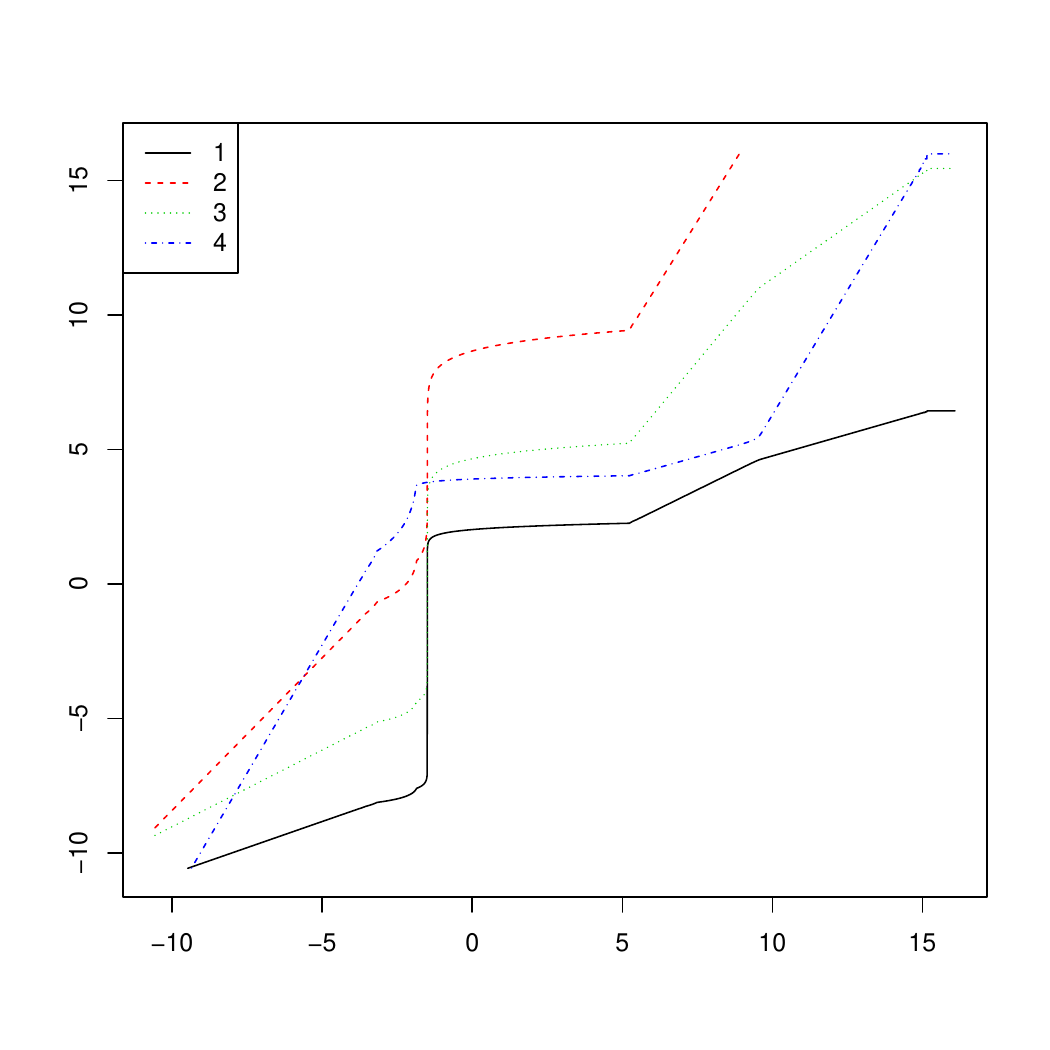}
\includegraphics[trim=10mm 16mm 10mm 20mm, clip, scale = 0.42]{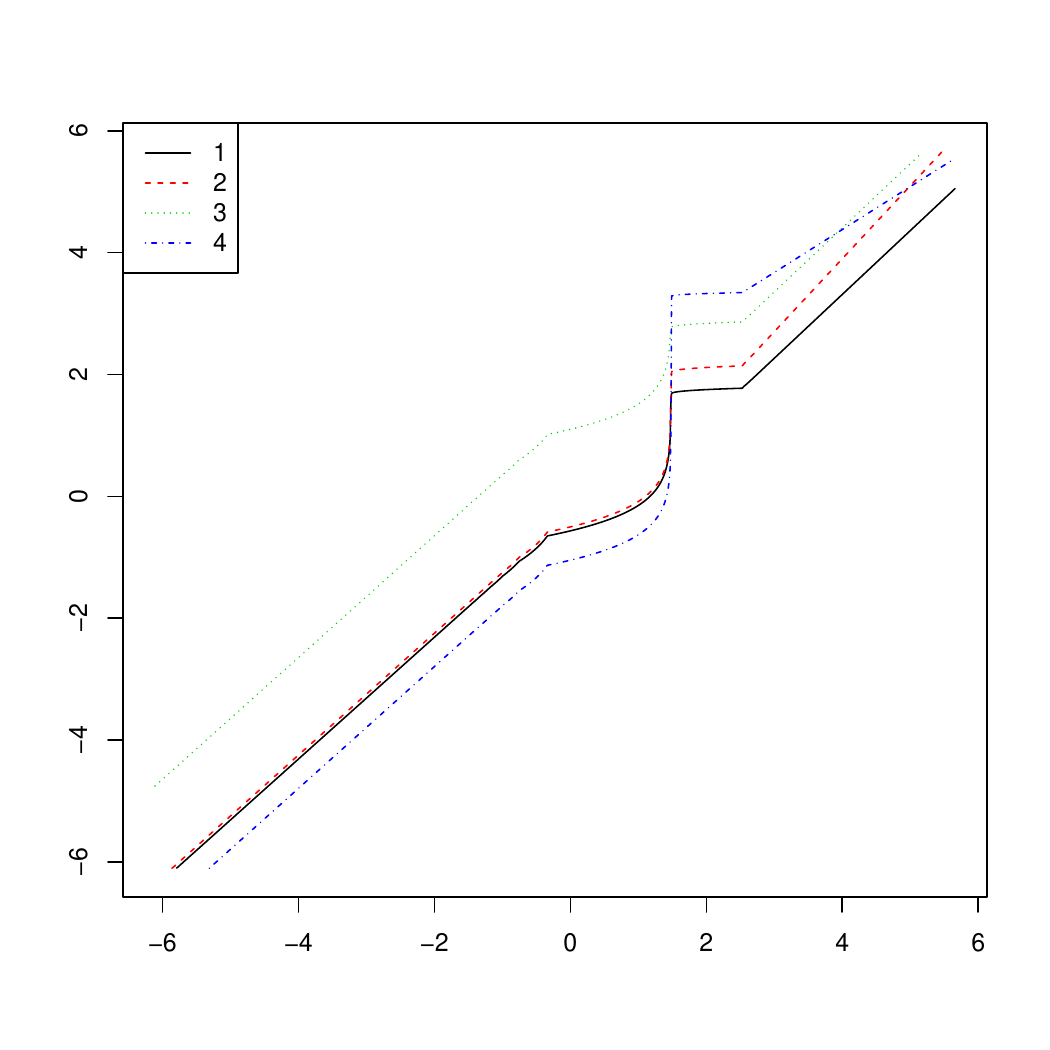}
}


\bigskip
\medskip

\subfigure[Common copula example: Procrustes registration maps $\mathbf t_{\bar\mu}^{\mu^i}$ (depicted as a vector field $\{\mathbf t_{\bar\mu}^{\mu^i}(x)-x:x\in\mathbb{R}^2\}$) from the Fr\'echet mean $\bar\mu$ of Figure~\ref{fig:F-8S80bar} to the four measures $\{\mu^i\}$ of Figure~\ref{fig:F-8S80}.  The colours match those of Figure~\ref{fig:F-8S80den}.]{
\includegraphics[trim=10mm 17mm 15mm 17mm, clip, scale = 0.25]{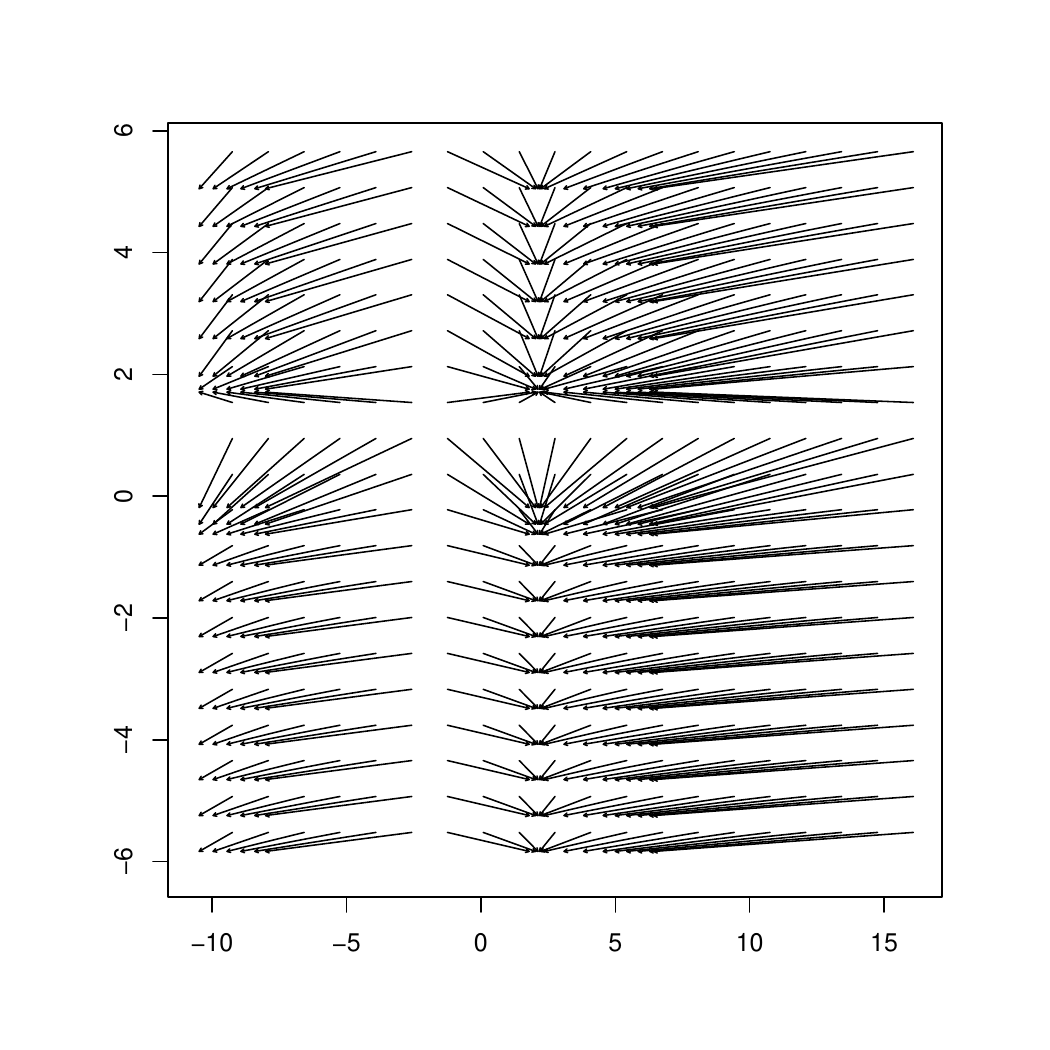}
\includegraphics[trim=10mm 17mm 15mm 17mm, clip, scale = 0.25]{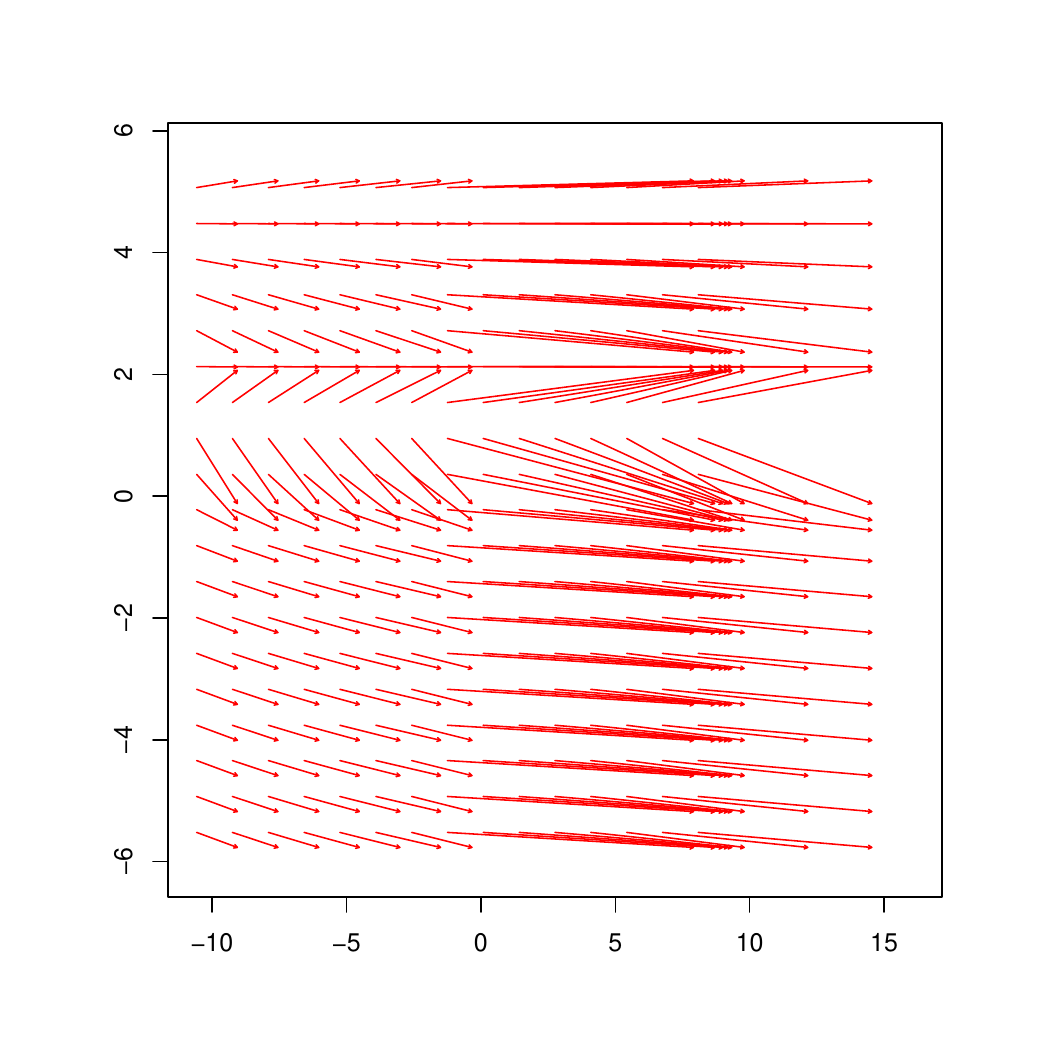}
\includegraphics[trim=10mm 17mm 15mm 17mm, clip, scale = 0.25]{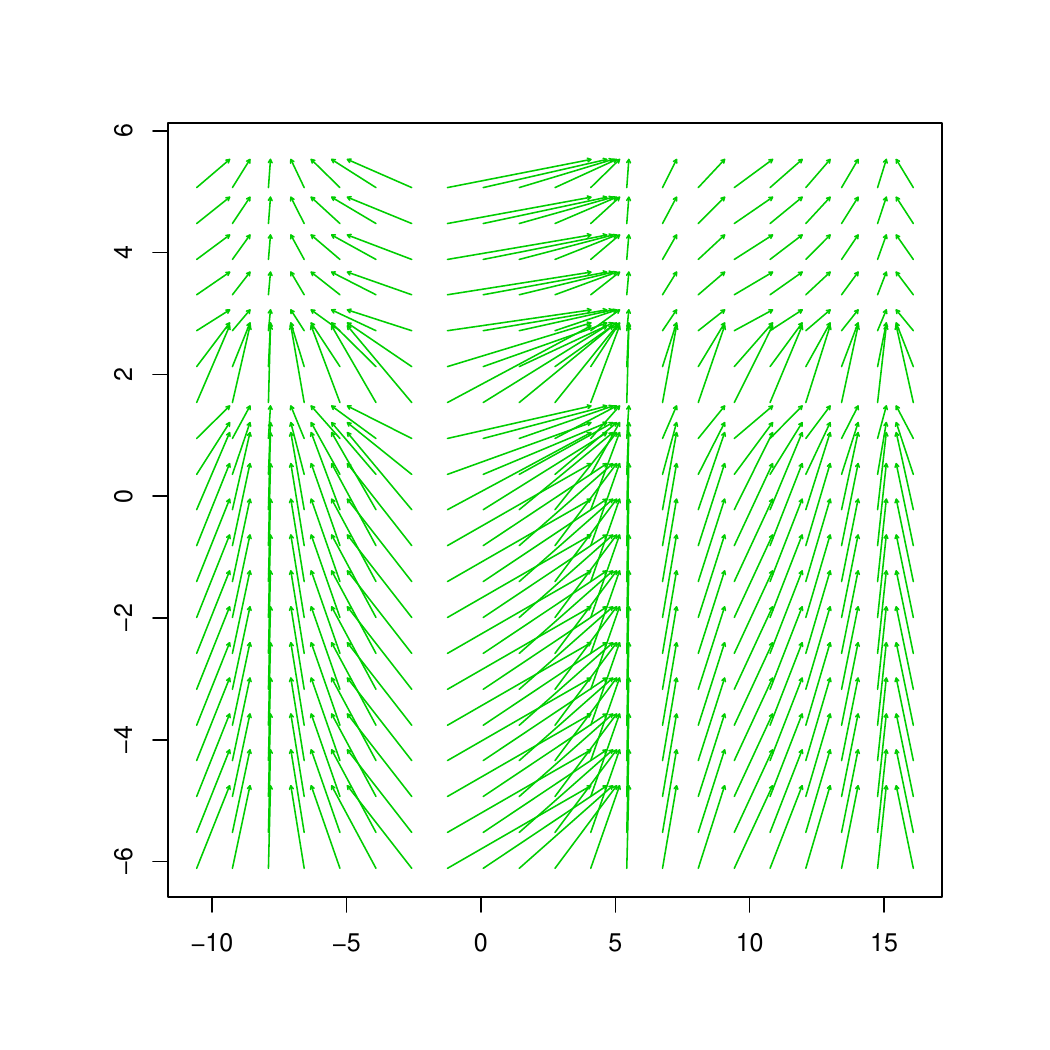}
\includegraphics[trim=10mm 17mm 15mm 17mm, clip, scale = 0.25]{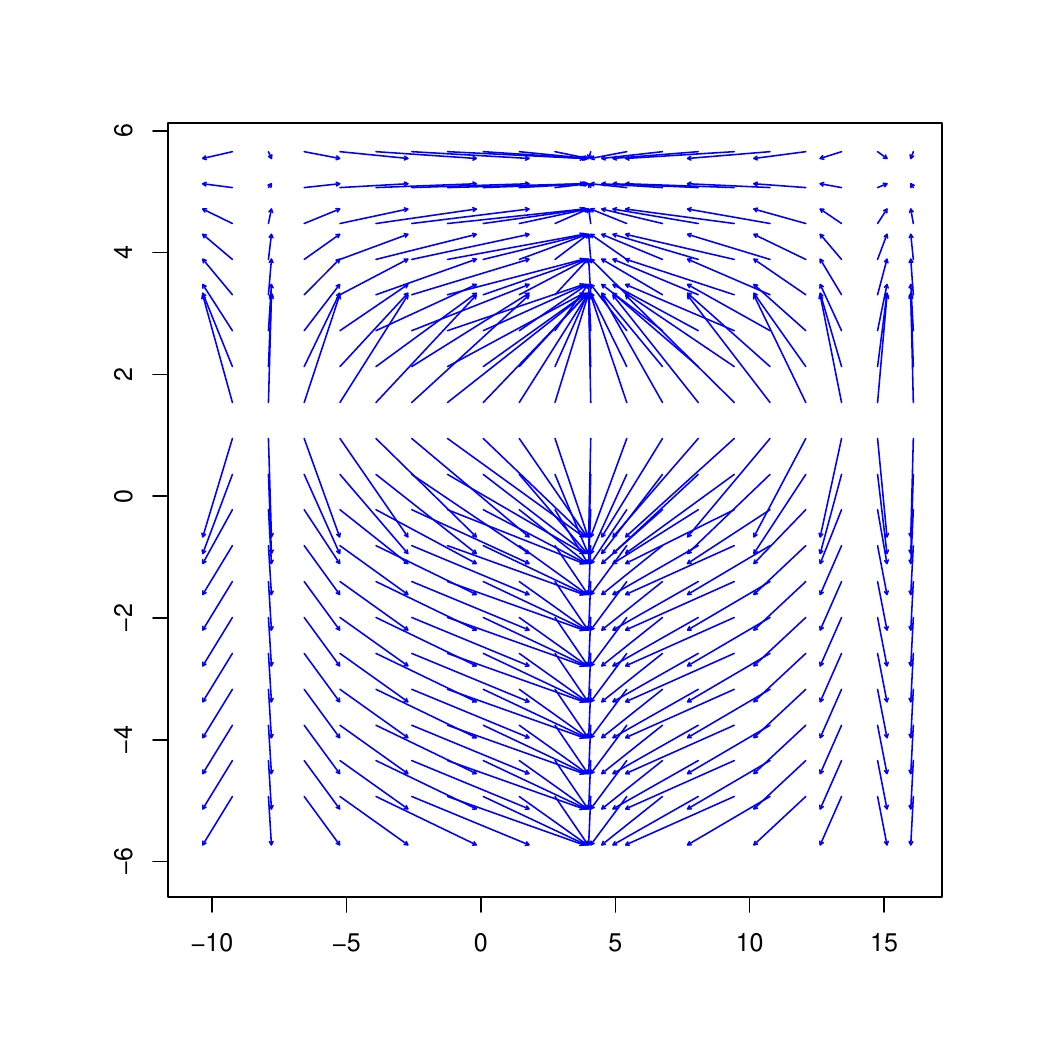}
}

\bigskip
\medskip
\subfigure[Gaussian example: Procrustes registration maps $\mathbf t_{\bar\mu}^{\mu^i}$ (depicted as a vector field $\{\mathbf t_{\bar\mu}^{\mu^i}(x)-x:x\in\mathbb{R}^2\}$) from the Fr\'echet mean $\bar\mu$ of Figure~\ref{fig:normalbarS22} to the four measures $\{\mu^i\}$ of Figure~\ref{fig:normalsS22}. The order corresponds to that of Figure~\ref{fig:normalsS22} (left to right and top to bottom).]{
\includegraphics[trim=10mm 17mm 15mm 17mm, clip, scale = 0.25]{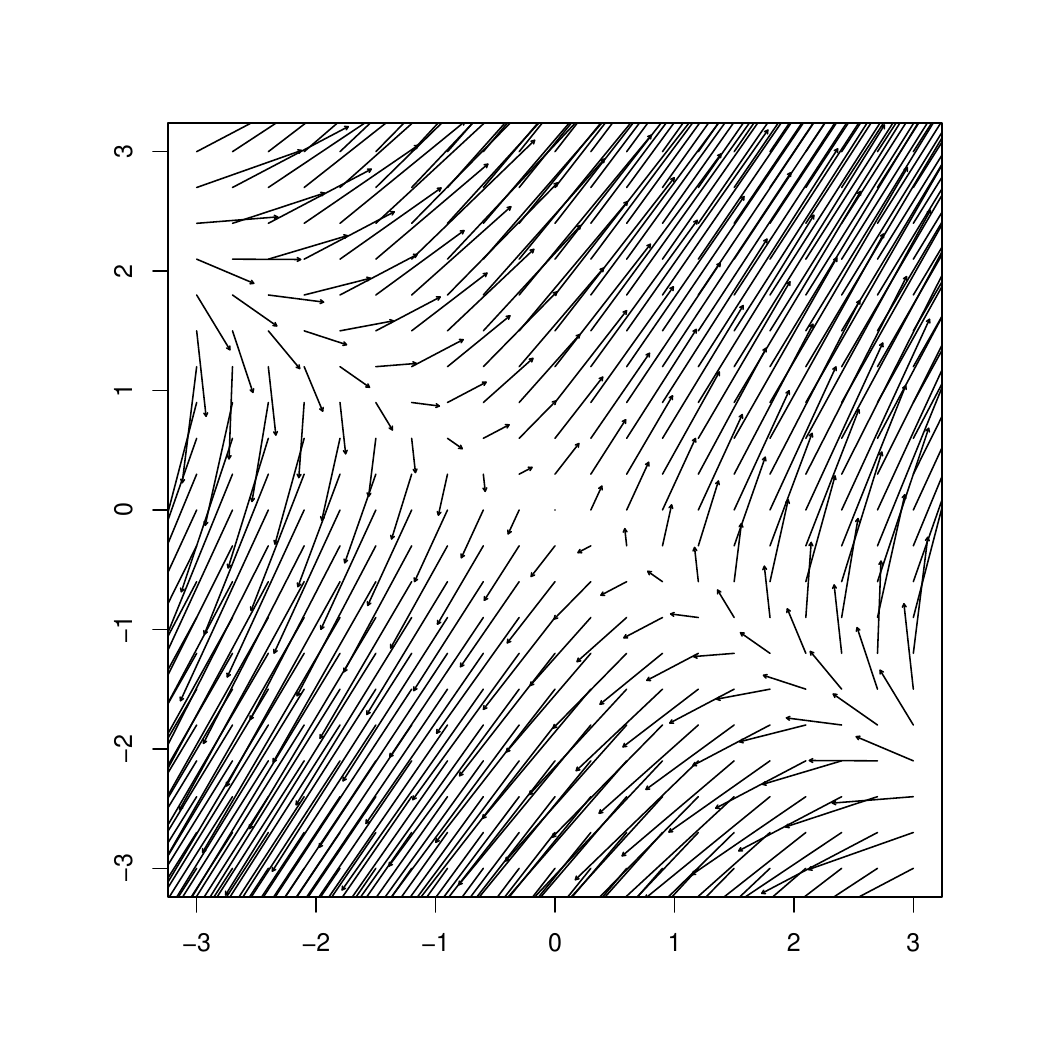}
\includegraphics[trim=10mm 17mm 15mm 17mm, clip, scale = 0.25]{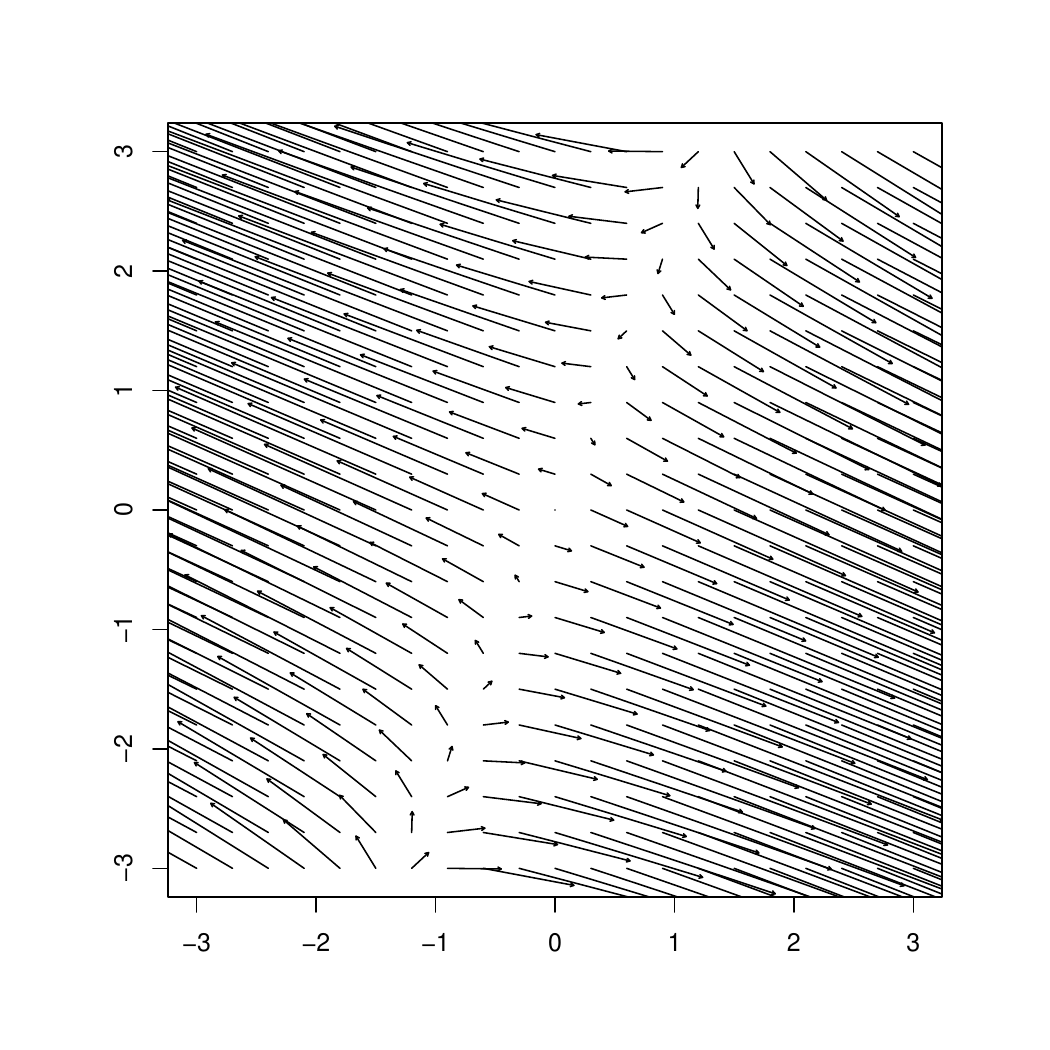}
\includegraphics[trim=10mm 17mm 15mm 17mm, clip, scale = 0.25]{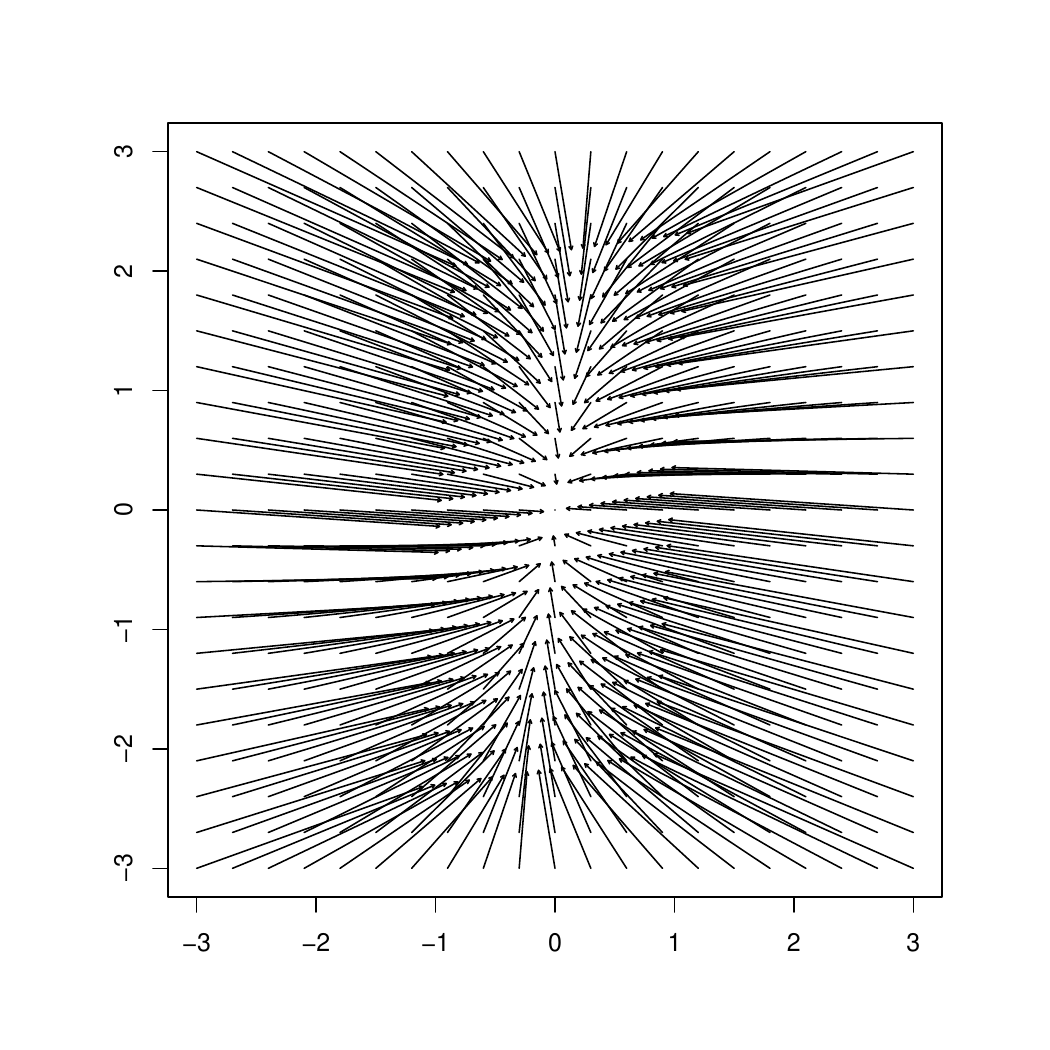}
\includegraphics[trim=10mm 17mm 15mm 17mm, clip, scale = 0.25]{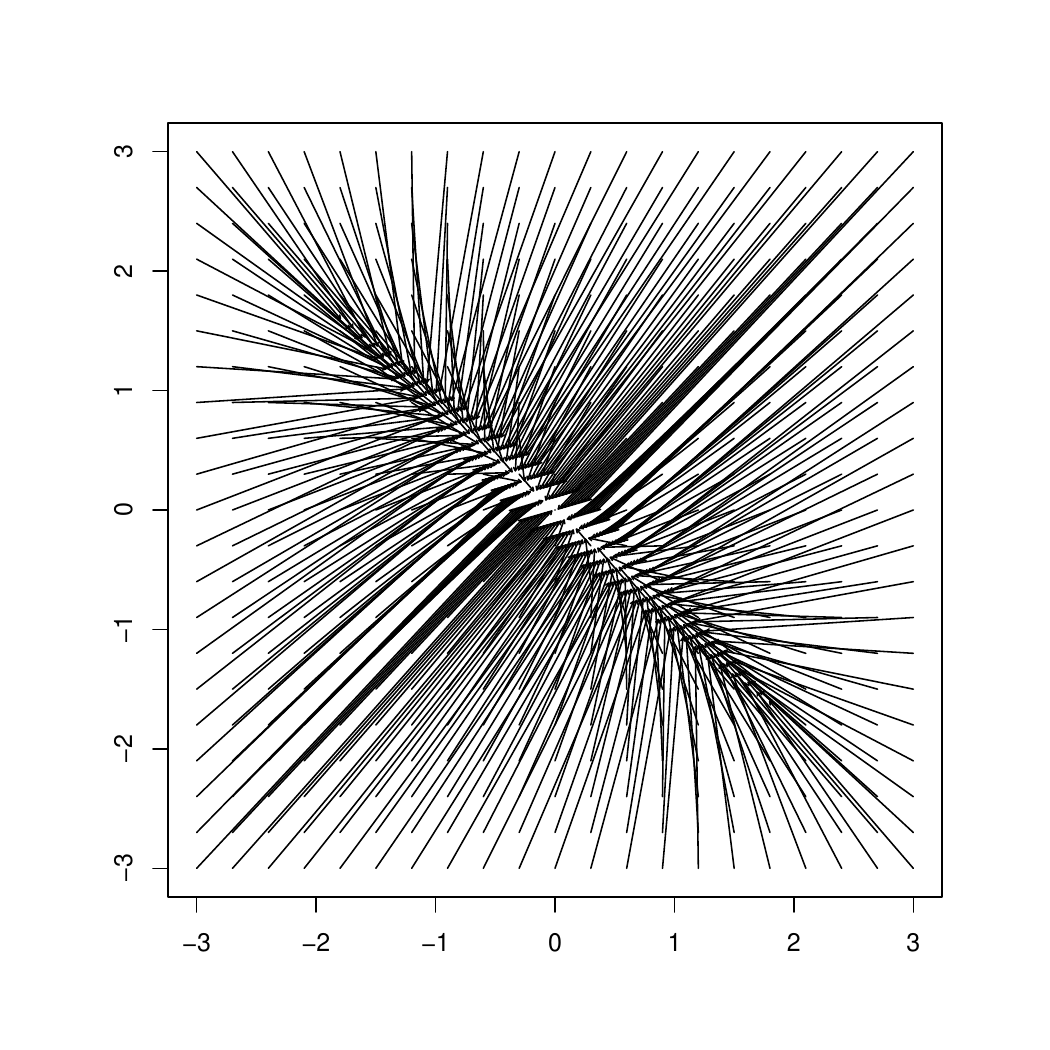}
}

\caption{Procrustes registration maps for the one-dimensional, common copula, and Gaussian examples.}\label{vector_fields}
\end{figure}

\subsection{Partially Gaussian Trivariate Measures}  We now apply Algorithm~\ref{algo} in a situation that entangles two of the previous settings.  Let $U$ be a $3\times 3$ real orthogonal matrix with columns $U_1$, $U_2$, $U_3$ and let $\mu^i$ have density
\[
g^i(y_1,y_2,y_3)
=g^i(y)
=f^i(U_3^ty)
\frac1{2\pi\sqrt{\mathrm{det}S^i}}     \exp\left[-\frac{(U_1^ty, U_2^ty)(S^i)^{-1}\binom{U_1^ty}{U_2^ty}}2\right],
\]
with $f^i$ bounded density on the real line and $S^i\in \mathbb R^{2\times 2}$ positive definite.  We simulated $N=4$ such densities with $f^i$ as in \eqref{eq:bigaussden} and $S^i\sim \mathrm{Wishart}(I_2, 2)$.   We apply Algorithm~\ref{algo} to this collection of measures and find their Fr\'echet mean (in Section~\ref{zempan16supp} we provide precise details on how the optimal maps were calculated).  Figure~\ref{fig:levelsetsS22L3} shows level set of the resulting densities for some specific values.  The bimodal nature of $f^i$ implies that for most values of $a$, $\{x:f^i(x)=a\}$ has four elements.  Hence the level sets in the figures are unions of four separate parts, with each peak of $f^i$ contributing two parts that form together the boundary of an ellipsoid in $\mathbb R^3$ (see Figure~\ref{fig:levelsetsS22L334col}).  The principal axes of these ellipsoids and their position in $\mathbb R^3$ differ between the measures, but the Fr\'echet mean can be viewed as an average of those in some sense.

In terms of orientation (principal axes) of the ellipsoids, the Fr\'echet mean is most similar to $\mu^1$ and $\mu^2$, whose orientations are similar to one another.

\begin{figure}
\includegraphics[trim=10mm 0mm 10mm 15mm, clip, scale = 0.55]{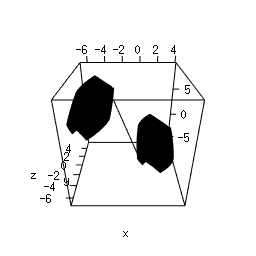}
\includegraphics[trim=10mm 0mm 10mm 15mm, clip, scale = 0.55]{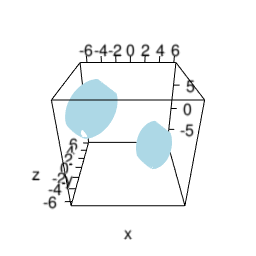}
\includegraphics[trim=10mm 0mm 10mm 15mm, clip, scale = 0.55]{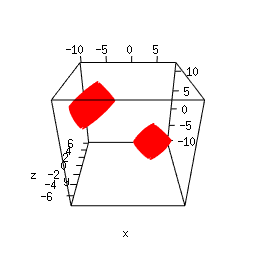}\\
\includegraphics[trim=10mm 0mm 10mm 15mm, clip, scale = 0.55]{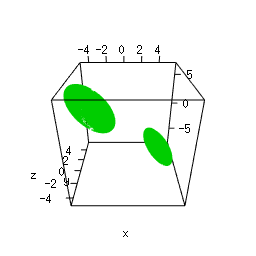}
\includegraphics[trim=10mm 0mm 10mm 15mm, clip, scale = 0.55]{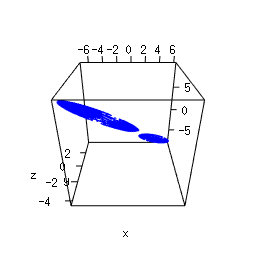}
\caption{The set $\{v\in\mathbb R^3:g^i(v)=0.0003\}$ for $i=1$ (black), the Fr\'echet mean (light blue), $i=2,3,4$ in red, green and dark blue respectively.}
\label{fig:levelsetsS22L3}
\end{figure}

\begin{figure}
\includegraphics[trim=10mm 0mm 10mm 15mm, clip, scale = 0.55]{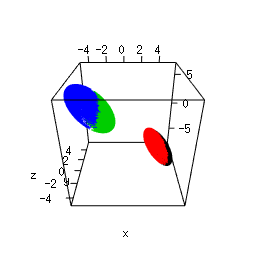}
\includegraphics[trim=10mm 0mm 10mm 15mm, clip, scale = 0.55]{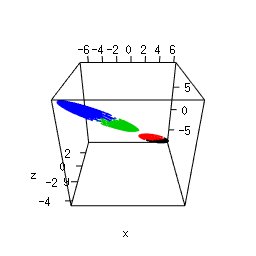}
\caption{The set $\{v\in\mathbb R^3:g^i(v)=0.0003\}$ for $i=3$ (left) and $i=4$ (right), with each of the four different inverses of the bimodal density $f^i$ corresponding to a colour.}
\label{fig:levelsetsS22L334col}
\end{figure}

In the most general examples, one might not be able to analytically obtain the optimal maps at each iteration.  In such situations, one needs to resort to numerical schemes such as Benamou \& Brenier \cite{benamou2000computational}, Haber et al.\ \cite{haber2010efficient} or Chartrand et al.\ \cite{chartrand2009gradient} to obtain the $N$ optimal maps at each iteration {(see the concluding remarks for further discussion about numerical issues)}.  Usually such schemes are iterative themselves, so one must take care in managing propagation of errors resulting from using approximate rather than exact transport maps.

\section{Concluding Remarks}\label{sec:conclusions}
While the algorithm and the convergence analysis in this work were discussed in the context of absolutely continuous measures, it is worth mentioning the possibility of applying it to discrete measures in some special cases.  Specifically, suppose that each measure $\mu^i$ is uniform on a set of $M$ distinct points, $\{x^i_m\}_{m=1}^M$.  Define as in Anderes et al.\ \cite{anderes2015discrete} the set
\[
S=
\frac1N\left\{x^1_{m_1} + \dots + x^N_{m_N}:
1\le m_i\le M,\quad i=1,\dots,N
\right\}
\]
of averages of choices of points from the supports of $\{\mu^i\}$.  Let $\gamma_0$ be an initial measure, uniform on $M$ distinct points as well.  There exist optimal maps (not necessarily unique) from $\gamma_0$ to each $\mu^i$, and they can be averaged to yield $\gamma_1$.  If $|S|=M^N$ (that is, the collection $\{x^i_m\}$ satisfies a general-position-type condition), then $\gamma_1$ will be concentrated on $M$ points as well, and one may carry out further iterations. A conceptual problem with this application is that the Fr\'echet functional is not differentiable at discrete measures, so Algorithm~\ref{algo} can no longer be viewed as gradient descent (but can still be seen as Procrustes averaging).  Also, the Fr\'echet mean itself may fail to be unique. In simulations we observed very rapid convergence of this iteration to a Karcher mean, but the specific limit depended quite heavily on the initial point, and was usually not a Fr\'echet mean.  For problems of moderate size, one can recast the problem of minimising the Fr\'echet functional as a linear program \cite{anderes2015discrete} and find an exact Fr\'echet mean.  In fact, Anderes et al.\ \cite{anderes2015discrete} treat the more general problem where the measures are supported on a different number of points and are not constrained to be uniform on their supports.

An important issue more generally is that of efficient approximate numerical schemes for calculating Fr\'echet means in Wasserstein space. This is a very active field of research with a rapidly-growing literature (both in numerical analysis and in computer science), and a detailed survey is far beyond the scope of this paper. If one is content with an \emph{approximate} solution, then there are several approaches suggested in the literature. Indicatively, let us mention Bonneel et al.\ \cite{bonneel2015sliced} who use a tomographic perspective to reduce the problem to 1-dimensional computations; Carlier, Oberman \& Oudet \cite{carlier2015numerical} who use nonsmooth optimisation techniques to solve a discretised version of the dual problem;  Oberman \& Ruan \cite{oberman2015efficient} exploit the sparsity of optimal plans to reduce the size of the linear program to a tractable one.

Another line of research involves \emph{entropic regularisation}, where one adds an entropy term to the definition of the Wasserstein distance.  This leads to a strictly convex problem that is far better behaved than the original problem. Though its solution no longer yields the actual mean, it can be thought of as a regularised surrogate Fr\'echet mean.  In this direction, Cuturi \& Doucet \cite{cuturi2014fast} employ differentiability properties and carry out what could be thought of as a ``gradient descent", a discrete analogue of Algorithm~\ref{algo};   Benamou et al.\  \cite{benamou2015iterative} exploit the structure of the constraints as an intersection of convex sets by means of iterating Bregman projections that can be evaluated efficiently.  Solomon et al.\ \cite{solomon2015convolutional} extend this idea to the manifold setup, by convoluting with a heat kernel;  and Cuturi \& Peyr\'e \cite{cuturi2016smoothed} employ the regularisation at the level of the dual, rather than the primal, problem.  Recently, Rolet, Cuturi \& Peyr\'e \cite{rolet2016fast} employed this technique in the context of dictionary learning;  and Bonneel, Peyr\'e \& Cuturi \cite{bonneel2016wasserstein} define a sort of ``barycentric convex hull" of given histograms and show how to project a new histogram onto that convex hull.

{{\footnotesize

\begin{center}
\textsc{Acknowledgements}
\end{center}
This research was supported by a European Research Council Starting Grant Award to Victor M.~Panaretos. Part of this work grew out of work presented at the Mathematical Biosciences Institute (Ohio State University), during the \href{http://mbi.osu.edu/event/?id=162}{``Statistics of Time Warping and Phase Variation"} Workshop, November 2012. We wish to acknowledge the stimulating environment offered by the Institute. We wish to warmly thank Prof.\ Cl\'{e}ment Hongler for several useful discussions. We are also very thankful to two reviewers and an associate editor for their detailed and constructive feedback.
}}

\section{Supplementary Material}\label{zempan16supp}
This section contains material supplementing the main article.  The first section contains the proof that no further requirement except for finite second moments is needed for the convergence of the algorithm presented in the article. Next, we provide further details and theoretical results pertaining to the simulation scenarios described in Section~\ref{SEC:EX}.  Finally, we provide all the proofs not included in the main body for tidiness, as well as additional technical details.

\section*{A complete proof of Lemma~\ref{lem:seqcompact}}
In this section we show that condition \eqref{eq:3mom} is not needed for \eqref{eq:uniformboundedness} to hold.  The idea is that \eqref{eq:uniformboundedness} only requires a tiny bit more than finite second moments, and that is provided in Lemma~\ref{lem:nohighestmom}.  Throughout this section, all functions are assumed nonnegative (possibly infinite-valued) and defined on $[0,\infty)$ unless explicitly stated otherwise.  We write $f(x)\in\omega(g(x))$ or $f\in\omega(g)$ if $f(x)/g(x)\to\infty$ as $x\to\infty$.

\begin{lem}\label{lem:intfun}
Let $f$ be integrable.  Then there exists a continuous nondecreasing function $g\in\omega(1)$ such that $fg$ is integrable.
\end{lem}
\begin{proof}  Set $F(x)=\ownint x\infty {f(t)}t$ and $g(x)=[F(x)]^{-1/2}$.  Then a change of variables gives
\[
\ownint 0\infty {f(x)g(x)}x
=\ownint 0\infty {f(x)[F(x)]^{-1/2}}x
=\ownint 0{F(0)} {u^{-1/2}}u
=2\sqrt{\|f\|_1}
<\infty,
\]
and $g(x)\to\infty$ because $F(x)\to0$ as $x\to\infty$ by dominated convergence.
\end{proof}

\begin{lem}\label{lem:nohighestmom}
Let $X$ be a random variable with $\E X^2<\infty$.  Then there exists a convex nondecreasing function $H(x)\in\omega(x^2)$ such that $\E H(X)<\infty$.
\end{lem}
\begin{proof}  Since
\[
\infty>
\E X^2
=\ownint 0\infty {\P(X^2>t)}t,
\]
there exists a function $g$ as in Lemma~\ref{lem:intfun} such that
\[
\infty
>\ownint 0\infty {\P(X^2>t)g(t)}t
=\ownint 0\infty {\P(X^2>G^{-1}(u))}u
=\ownint 0\infty {\P(G(X^2)>u)}u
=\E G(X^2),
\]
where $G$ is the primitive of $g$ and $G(0)=0$.  The properties of $g$ imply that $G$ is convex and invertible, and that for $y<x$,
\[
G(x)
\ge \ownint yx{g(t)}t
\ge \ownint yx{g(y)}t
=(x-y)g(y),
\]
which, combined with $g(y)\to\infty$ as $y\to\infty$, yields
\[
\liminf_{x\to\infty}\frac{G(x)}x
\ge g(y)
\to \infty,
\qquad y\to\infty,
\]
so that $G(x)\in\omega(x)$.  The function $H(x)=G(x^2)$ then has all the desired properties.
\end{proof}

\begin{prop}
Equation~\eqref{eq:uniformboundedness} holds if merely
\[
\ownint {\R^d}{}{\|x\|^2}{\mu^i(x)}
<\infty
,\qquad i=1,\dots,N.
\]
\end{prop}
\begin{proof}  Let $X_i=\|Z^i\|$ where $Z^i\sim \mu^i$.  Then there exist functions $g^i$ as in Lemma~\ref{lem:intfun} with
\[
\ownint 0\infty {\P(X^2_i>t)g^i(t)}t
<\infty
,\qquad i=1,\dots,N.
\]
The same holds with $g^i$ replaced by $g=\min_i g^i$, which is still continuous, nondecreasing and divergent.  Setting $H$ as in Lemma~\ref{lem:nohighestmom}, we see that $H(x)\in\omega(x^2)$ and
\[
M^i
=\mathbb EH(X^i)
=\ownint {\R^d}{}{H(\|x\|)}{\mu^i(x)}
<\infty,
\qquad i=1,\dots,N.
\]
Convexity of $H$ and $\|\cdot\|$ combined with monotonicity of $H$ yield
\begin{align*}
\ownint {\R^d}{}{H(\|x\|)}{\gamma_j(x)}
&=\ownint {\R^d}{}{H\left(\left\|\frac1N \sum_{i=1}^N \mathbf t_{\gamma_{j-1}}^{\mu^i} (x)\right\|\right)}{\gamma_{j-1}(x)}\\
&\le \frac1N\sum_{i=1}^N \ownint {\R^d}{}{H(\|\mathbf t_{\gamma_{j-1}}^{\mu^i}(x)\|)}{\gamma_{j-1}(x)}
=\frac1N\sum_{i=1}^N \ownint {\R^d}{}{H(\|x\|)}{\mu^i(x)}
\le M,
\end{align*}
where $M=\sum_{i=1}^N M^i/N$.  This implies that for any $R>0$ and any $j>0$,
\[
\ownint{\{x:\|x\|>R\}}{}{\|x\|^2}{\gamma_j(x)}
\le \sup_{y>R}\frac {y^2}{H(y)}\ownint{\{x:\|x\|>R\}}{}{H(\|x\|)}{\gamma_j(x)}
\le M\sup_{y>R}\frac {y^2}{H(y)},
\]
and \eqref{eq:uniformboundedness} follows because $H(y)\in\omega(y^2)$.
\end{proof}

\section*{Details for the illustrative examples in Section~\ref{SEC:EX}}
In this section we provide further details for finding the optimal maps in the examples of Section~\ref{SEC:EX} and theoretical results about the Fr\'echet mean and the behaviour of the algorithm.  Throughout this section, $\mu^1,\dots,\mu^N$ are given measures and $\gamma_0$ is the initial point of Algorithm~\ref{algo}.  We begin with two  lemmas regarding compatibility of the measures as defined in Section~\ref{SEC:EX}.

\begin{lem}[Compatibility and Convergence]
If $\mathbf t_{\mu^1}^{\mu^i}\circ \mathbf t_{\gamma_0}^{\mu^1}=\mathbf t_{\gamma_0}^{\mu^i}$ and $\mathbf t_{\mu^1}^{\mu^j}\circ \mathbf t_{\mu^i}^{\mu^1}=\mathbf t_{\mu^i}^{\mu^j}$ (in the relevant $L^2$ spaces) for all $i$ and all $j$, then Algorithm~\ref{algo} converges after a single step.
\end{lem}
\begin{proof}
For all $i$, $j$ and $k$ we have $\mathbf t_{\mu^j}^{\mu^k}\circ \mathbf t_{\mu^i}^{\mu^j}=\mathbf t_{\mu^i}^{\mu^k}$, so that the optimal maps are admissible, and
\[
\gamma_1
=\left[\frac1N\sum_{i=1}^N  \mathbf t_{\gamma_0}^{\mu^i} \right]\#\gamma_0
=\left[\frac1N\sum_{i=1}^N  \mathbf t_{\mu^1}^{\mu^i}\circ \mathbf t_{\gamma_0}^{\mu^1}\right]\#\gamma_0
=\left[\frac1N\sum_{i=1}^N   \mathbf t_{\mu^1}^{\mu^i}\right]\#\mu^1.
\]
Boissard et al.\ \cite{boissard2015distribution} show that this is indeed the Fr\'echet mean.
\end{proof}
When $d=1$, all (diffuse) measures are compatible with each other, and Algorithm~1 converges after one step.  Generally, the algorithm requires the calculation of $N$ pairwise optimal maps, and this can be reduced to $N-1$ if $\gamma_0=\mu^1$.  This is the same computational complexity as the calculation of the iterated barycentre proposed in \cite{boissard2015distribution}.

Measures on $\mathbb R^d$ that have a common dependence structure are compatible with each other.  More precisely, we say that $C:[0,1]^d\to[0,1]$ is a \textit{copula} if there exists a random vector $U$ with $U[0,1]$ margins and such that
\[
\mathbb P(U_1\le u_1,\dots,U_d\le u_d)
=C(u_1,\dots,u_d),
\qquad u_i\in[0,1].
\]
In other words, a copula is the restriction to $[0,1]^d$ of the probability distribution function of some $d$-dimensional random variable with uniform margins.  See, for example, Nelsen \cite{nelsen2013introduction} for an overview.  Given a measure $\mu$ on $\mathbb R^d$ with distribution function $G$ and marginal distribution functions $G_j$, the copula associated with $\mu$ is a copula such that
\[
G(a_1,\dots,a_d)
=\mu((-\infty,a_1]\times\dots\times (-\infty,a_d])
=C(G_1(a_1),\dots,G_d(a_d)).
\]
This equation defines $C$ uniquely if each marginal $G_i$ is continuous, which we shall assume for simplicity.  (If some $G_i$ is discontinuous then $C$ might not be unique, but it always exists, see \cite[Chapter~2]{nelsen2013introduction}.)

\begin{lem}[Compatibility and Copulae]\label{lem:compcop}
Let $\mu,\nu\in\mathcal P_2(\mathbb R^d)$ be regular.  Then $\mu$ and $\nu$ have the same associated copula if and only if $\mathbf t_\mu^\nu$ takes the separable form
\begin{equation}\label{eq:optimalsep}
\mathbf t_\mu^\nu(x_1,\dots,x_d)
=(T_1(x_1),\dots,T_d(x_d)),
\qquad T_i:\mathbb R\to\mathbb R.
\end{equation}
\end{lem}
The result can be obtained as a corollary of Cuesta-Albertos et al. \cite[Theorem~2.9]{cuesta1993optimal}, but here is an alternative direct proof.
\begin{proof}
If $\mu$ and $\nu$ have the same copula then
\[
G(G_1^{-1}(u_1),\dots,G_d^{-1}(u_d))
=C(u_1,\dots,u_d)
=F(F_1^{-1}(u_1),\dots,F_d^{-1}(u_d)),
\]
where $G_j^{-1}(u_j)$ is any number satisfying $G_j(G_j^{-1}(u_j))=u_j$ (such numbers exist because $G_j$ is surjective), and similarly for $F_j^{-1}$.  Consequently, $F(x_1,\dots,x_d)=G(T_1(x_1),\dots,T_d(x_d))$ with $T_j=G_j^{-1}\circ F_j$.  It follows that $\nu=(T_1,\dots,T_d)\#\mu$, and this map is optimal, hence equals $\mathbf t_\mu^\nu$, because the $T_j$'s are nondecreasing.

One proves the converse implication similarly:  if $\mathbf t_\mu^\nu$ takes this form, then each $T_j$ needs to be nondecreasing.  Since it must push $F_j$ forward to $G_j$, we have $T_j=G_j^{-1}\circ F_j$, and this yields the above equality for the copula.
\end{proof}
It is easy to see that if the optimal maps between each $\mu^i$ and each $\mu^j$ are of the form \eqref{eq:optimalsep}, then $\{\mu^i\}$ are compatible with other.  This follows from this property holding for each marginal, and the possibility of working with the marginals separately;  it has already been observed by Boissard et al.\ \cite[Proposition~4.1]{boissard2015distribution}.  This explains why the algorithm converges in one iteration for the example with the Frank copula.

Next, we give a convergence analysis for the Gaussian example.
\begin{thm}[Convergence in Gaussian case]\label{gaussian_convergence}
Let $\mu^i\sim\mathcal N(0,S_i)$ for $S_i$ positive definite, and let the initial point $\gamma_0=\mathcal N(0,\Gamma_0)$ for positive definite $\Gamma_0$.  Then the sequence of iterates generated by Algorithm~\ref{algo} converges to the unique Fr\'echet mean of $(\mu^1,\dots,\mu^N)$.
\end{thm}
\begin{proof}
We first observe that for any centred measure $\mu$ with covariance matrix $S$,
\[
d^2(\mu,\delta_0)
=\mathrm{tr}S,
\]
where $\delta_0$ is a dirac mass at the origin.  (This follows from the singular value decomposition of $S$.)  Next, each iteration stays (centred) Gaussian, say $\mathcal N(0,\Gamma_k)$, because the optimal maps are linear;  and since the iterates are absolutely continuous (Lemma~\ref{iterate_regularity}), each $\Gamma_k$ is nonsingular.

Proposition~\ref{prop:Cmu} implies that $\mathrm{det}\Gamma_k$ is bounded below uniformly;  on the other hand,
\[
0\le
\mathrm{tr}\Gamma_k
=d^2(\gamma_k,\delta_0)
\]
is bounded uniformly, because $\{\gamma_k\}$ stays in a Wasserstein-compact set by Lemma~\ref{lem:seqcompact}.  Let $C_1=\inf_k\mathrm{det}\Gamma_k>0$ and $C_2=\sup_k\mathrm{tr}\Gamma_k<\infty$.  Then each eigenvalue $\lambda$ of $\Gamma_k$ is nonnegative, bounded above by $C_2$, and satisfies
\[
C_1\le \mathrm{det}\Gamma_k
\le \lambda C_2^{d-1}
\qquad\Longrightarrow\quad \lambda\ge C_1C_2^{1-d}=C_3>0.
\]
The matrices $\Gamma_k$ stay in a bounded set, and each limit point $\Gamma$ is positive definite because $x^t\Gamma x\ge C_3\|x\|^2$ for all $x\in\mathbb R^d$.  Each limit point $\gamma$ of $\gamma_k$ is a Karcher mean by Theorem~\ref{thm:convtosta}, and the limit must follow a $\mathcal N(0,\Gamma)$ distribution with $\Gamma$ (nonsingular) limit point of $\Gamma_k$ (e.g., by Lehmann--Scheff\'e's theorem).   Since $F'(\gamma)=0$ everywhere on $\mathbb R^d$, $\gamma$ is the Fr\'echet mean by the discussion after Corollary~\ref{thm:minisstationary}.  Every limit of $\gamma_k$ is the Fr\'echet mean and the sequence is compact, so $\gamma_k$ must converge to the Fr\'echet mean.
\end{proof}
\begin{rem}
During the review process, a referee asked whether the result in Theorem \ref{gaussian_convergence} is related to the iteration $\Sigma\mapsto N^{-1}\sum_{i=1}^N (\Sigma^{1/2} \Sigma_i\Sigma^{1/2})^{1/2}$, introduced by Knott \& Smith \cite{knott1994generalization} and later considered in Xia et al.\ \cite{xia2014synthesizing} for the Gaussian case. That iteration, however, is distinctly different from Algorithm \ref{algo}, which Theorem \ref{gaussian_convergence} concerns (for instance, the latter involves inversion operations, which the former does not).  As pointed out by R\"uschendorf \&  Uckelmann \cite[p.\ 6]{ruschendorf2002n}, the scheme $\Sigma\mapsto N^{-1}\sum_{i=1}^N (\Sigma^{1/2} \Sigma_i\Sigma^{1/2})^{1/2}$ is not known to converge, and indeed Theorem \ref{gaussian_convergence} does not furnish any additional insight on the matter.
\end{rem}

In order to deal with the last example of Section~\ref{SEC:EX}, we need two more results.  The first involves coupling measures of dimensions greater than one, while the second shows the equivariance of the Fr\'echet mean with respect to rotations.

Invoking the \textit{independence copula} $C(u_1,\dots,u_d)=u_1\dots u_d$, a special case of Lemma~\ref{lem:compcop} above is when the marginals of $\mu$ and $\nu$ are independent.  In this independence case, it is possible in fact to replace the marginals by measures of arbitrary dimension:
\begin{lem}\label{lem:Frechetindep}
Let $\mu^1,\dots,\mu^N$ and $\nu^1,\dots,\nu^N$ be regular measures in $\mathcal P_2(\mathbb R^{d_1})$ and $\mathcal P_2(\mathbb R^{d_2})$ with (unique) Fr\'echet means $\mu$ and $\nu$ respectively.  Then the independent coupling $\mu\otimes\nu$ is the Fr\'echet mean of $\mu^1\otimes\nu^1,\dots,\mu^N\otimes\nu^N$.
\end{lem}
By induction (or a straightforward modification of the proof), one can show that the Fr\'echet mean of $(\mu^i\otimes \nu^i\otimes\rho^i)$ is $\mu\otimes\nu\otimes\rho$, and so on. While we are confident this result should already be known, we could not find a reference, and thus we provide a full proof for completeness.
\begin{proof}
Agueh \& Carlier \cite[Proposition~3.8]{bary} show that there exist convex lower semicontinuous potentials $\psi_i$ on $\mathbb R^{d_1}$ and $\varphi_i$ on $\mathbb R^{d_2}$ whose gradients push $\mu$ forward to $\mu^i$ and $\nu$ to $\nu^i$ respectively, and such that
\[
\frac1N\sum_{i=1}^N\psi_i^*(x)
\le \frac{\|x\|^2}2
,\quad x\in\mathbb R^{d_1};
\qquad
\frac1N\sum_{i=1}^N\varphi_i^*(y)
\le \frac{\|y\|^2}2
,\quad y\in\mathbb R^{d_2},
\]
with equality $\mu$- and $\nu$-almost surely respectively.  It is easy to see that the extensions $\tilde\psi_i(x,y)=\psi_i(x)$ and $\tilde\varphi_i(x,y)=\varphi_i(y)$ defined on $\mathbb R^{d_1+d_2}$ are convex lower semicontinuous functions whose sum $\phi_i$ is a convex function satisfying
\[
\phi_i^*(x,y)
=(\tilde\psi_i+\tilde\varphi_i)^*(x,y)
=\psi_i^*(x)+\varphi_i^*(y).
\]
Clearly $\nabla\phi_i\#(\mu^i\otimes\nu^i)=\mu\otimes\nu$ and
\[
\frac1N\sum_{i=1}^N\phi_i^*(x,y)
\le \frac{\|x\|^2}2 + \frac{\|y\|^2}2
= \frac{\|(x,y)\|^2}2
,\quad (x,y)\in\mathbb R^{d_1+d_2},
\]
with equality $\mu\otimes\nu$-almost surely.  By the same Proposition~3.8 in \cite{bary}, $\mu\otimes\nu$ is the Fr\'echet mean.
\end{proof}

\begin{lem}\label{lem:Frechetortho}
If $\mu$ is the Fr\'echet mean of the regular measures $\mu^1,\dots,\mu^N$, one with bounded density, and $U$ is orthogonal, then $U\#\mu$ is the Fr\'echet mean of $U\#\mu^1,\dots,U\#\mu^N$.
\end{lem}
\begin{proof}
Bonneel et al.\ sketch a proof of this statement in \cite[Proposition~1]{bonneel2015sliced}, and it also appears implicitly in Boissard et al.\ \cite[Proposition~4.1]{boissard2015distribution}; we give an alternative argument here.

If $x\mapsto\varphi(x)$ is convex, then $x\mapsto\varphi(U^{-1}x)$ is convex with gradient $U\nabla\varphi(U^{-1}x)$ at (almost all) $x$ and conjugate $x\mapsto\varphi^*(U^{-1}x)$.  If $\varphi_i$ are convex potentials with $\nabla\varphi_i\#\mu=\mu^i$, then $\nabla(\varphi_i\circ U^{-1})$ pushes $U\#\mu$ forward to $U\#\mu^i$ and by \cite[Proposition~3.8]{bary}
\[
\frac1N\sum_{i=1}^N (\varphi_i\circ U^{-1})^*(Ux)
=\frac1N\sum_{i=1}^N \varphi_i^*(x)
\le \frac{\|x\|^2}2
=\frac{\|Ux\|^2}2
\]
with equality for $\mu$-almost any $x$.  A change of variables $y=Ux$ shows that the set of points $y$ such that $\sum(\varphi_i\circ U^{-1})^*(y)< N\|y\|^2/2$ is $(U\#\mu)$-negligible, completing the proof.
\end{proof}

We apply these results in the context of the simulated example in Section~\ref{SEC:EX}.  If $Y=(y_1,y_2,y_3)\sim \mu^i$, then the random vector $(x_1,x_2,x_3)=X=U^{-1}Y$ has joint density
\[
f^i(x_3)\exp\left[-\frac{(x_1,x_2)(\Sigma^i)^{-1}\binom{x_1}{x_2}}2\right]
\frac1{2\pi\sqrt{\mathrm{det}\Sigma^i}},
\]
so the probability law of $X$ is $\rho^i\otimes\nu^i$ with $\rho^i$ centred Gaussian with covariance matrix $\Sigma^i$ and $\nu^i$ having density $f^i$ on $\mathbb R$.  By Lemma~\ref{lem:Frechetindep}, the Fr\'echet mean of $(U^{-1}\#\mu^i)$ is the product measure of that of $(\rho^i)$ and that of $(\nu^i)$;   by Lemma~\ref{lem:Frechetortho}, the Fr\'echet mean of $(\mu^i)$ is therefore
\[
U\#(\mathcal N(0,\Sigma)\otimes f),
\qquad f=F',
\quad F^{-1}(q)
=\frac1N\sum_{i=1}^NF_i^{-1}(q),
\quad F_i(x)
=\ownint{-\infty}x{f^i(s)}s,
\]
where $\Sigma$ is the Fr\'echet--Wasserstein mean of $\Sigma_1,\dots,\Sigma_N$.

Starting at an initial point $\gamma_0=U\#(\mathcal N(0,\Sigma_0)\otimes \nu_0)$, with $\nu_0$ having continuous distribution $F_{\nu_0}$, the optimal maps are $U\circ \mathbf t_0^i\circ U^{-1}=\nabla(\varphi_0^i\circ U^{-1})$ with
\[
\mathbf t_0^i(x_1,x_2,x_3)
=\binom{\mathbf t_{\Sigma_0}^{\Sigma^j}(x_1,x_2)}{F_j^{-1}\circ F_{\nu_0} (x_3)}
\]
the gradients of the convex function
\[
\varphi_0^i(x_1,x_2,x_3)
=(x_1, x_2)\mathbf t_{\gamma_0}^{\Sigma^i}\binom{x_1}{x_2}
+\ownint 0{x_3}{F_j^{-1}(F_{\nu_0}(s))}s,
\]
where we identify $\mathbf t_{\gamma_0}^{\Sigma^i}$ with the positive definite matrix $(\Sigma^i)^{1/2}[(\Sigma^i)^{1/2}\Sigma_0(\Sigma^i)^{1/2}]^{-1/2}(\Sigma^i)^{1/2}$ that pushes forward $\mathcal N(0,\Sigma_0)$ to $\mathcal N(0,\Sigma^i)$.  Due to the one-dimensionality, the algorithm finds the third component of the rotated measures after one step, but the convergence of the Gaussian component requires further iterations.

\section*{Proofs and Details Omitted from the Main Article}

\subsection*{Proofs of statements from Section~\ref{population_to_sample}}

\begin{proof}[Proof of Corollary~\ref{thm:minisstationary}, Section \ref{characterisation_mean}]
The characterisation of Karcher means follows immediately from Theorem \ref{gradient_definition}. Now suppose that $\mu\in\mathcal P_2(\mathbb R^d)$ is regular and $F'(\mu)\ne0\in L^2(\mu)$.  The function $S=N^{-1}\sum_{i=1}^N\mathbf t_{\mu}^{\mu^i}$ is a gradient of a convex function and
\[
\lim_{\nu\to\mu}\frac{F(\nu)- F(\mu) + \innprod{S - \mathbf i}{\mathbf t_\mu^\nu - \mathbf i}_{L^2(\mu)}}
{\|\mathbf t_\mu^\nu - i\|_{L^2(\mu)}}
=\lim_{\nu\to\mu}\frac{F(\nu)- F(\mu) + \ownint {}{}{\innprod{S(x)-x}{\mathbf t_\mu^\nu(x)-x}}{\mu(x)}}
{d(\nu,\mu)}
=0.
\]
By assumption $W=S-\mathbf i\ne0\in L^2(\mu)$.  The measure $\nu_s=[\mathbf i + s(W-\mathbf i)]\#\mu$ with $s\in(0,1)$ is such that $d(\nu_s,\mu)=s\|W\|_{L^2(\mu)}$ and
\[
0=
\lim_{s\to 0^+}\frac{F(\nu_s)- F(\mu) + \ownint {}{}{\innprod{W(x)}{sW(x)}}{\mu(x)}}
{s\|W\|_{L^2(\mu)}}
=\lim_{s\to 0^+}\frac{F(\nu_s)- F(\mu) }{s\|W\|_{L^2(\mu)}} + \|W\|_{L^2(\mu)}.
\]
This means that when $s$ is small enough, $F(\nu_s)<F(\mu)$, so $\mu$ cannot be the minimiser of $F$.  Since $\bar\mu$ has to be regular \cite[Proposition 5.1]{bary}, necessity of $F'(\bar\mu)=0$ is proven.
\end{proof}

\subsection*{Proofs of statements from Section~\ref{sec:population}}
\begin{proof}[Additional Details on the Proof of Theorem~\ref{thm:meanTid}]
Write $M(\gamma)=\mathbb{E}[d^2(\Lambda,\gamma)]$.  We wish to show that $M$ has a unique minimiser $\gamma$ and that $\gamma$ is supported on $K$.  We first establish (weak) convexity of $M$. Indeed, for given measures $\gamma$ and $\rho$ and $0<t<1$,
\begin{align*}
tM_\omega(\gamma) + (1-t)M_\omega(\rho)
&=t\ownint{\mathbb R^d\times\mathbb R^d}{}{ (x-y)^2}{\pi_{\omega,\gamma} (x,y)} + (1-t)\ownint{\mathbb R^d\times\mathbb R^d}{}{ (x-y)^2}{\pi_{\omega,\rho}(x,y)}\\
&=\ownint{\mathbb R^d\times\mathbb R^d}{}{ (x-y)^2}{[t\pi_{\omega,\gamma} + (1-t)\pi_{\omega,\rho}]},
\end{align*}
where $\pi_{\omega,\gamma}$ is the optimal coupling between $\Lambda=\Lambda(\omega)$ and $\gamma$.  The measure $t\pi_{\omega,\gamma} + (1-t)\pi_{\omega,\rho}$ is a coupling between $\Lambda$ and $t\gamma+(1-t)\rho$, and this shows that $M_\omega$ is convex without any regularity assumptions on $\Lambda$. To upgrade to strict convexity when $\Lambda$ is regular, observe firstly that $M$ is finite on the set of probability measures supported on $K$.  If $\Lambda$ is regular, then optimal measures are supported on graphs of functions:
\begin{align*}
\pi_{\omega,\gamma}(A\times B)
&=\Lambda(A\cap T_1^{-1}(B))\\
\pi_{\omega,\rho}(A\times B)
&=\Lambda(A\cap T_2^{-1}(B))\\
\pi_{\omega,t\gamma+(1-t)\rho}(A\times B)
&=\Lambda(A\cap T_3^{-1}(B))\\
[t\pi_{\omega,\gamma}+(1-t)\pi_{\omega,\rho}](A\times B)
&=t\Lambda(A\cap T_1^{-1}(B))
+(1-t)\Lambda(A\cap T_2^{-1}(B)).
\end{align*}
The measure $t\pi_{\omega,\gamma}+(1-t)\pi_{\omega,\rho}$ is supported on the graph of two functions, $T_1$ and $T_2$.  It can only be optimal if it is supported on the graph of one function, and this will only happen if $T_1=T_2$, $\Lambda$-almost surely, that is, if $\gamma=\rho$.  (See \cite[Corollary 2.9]{alvarez2011} for a rigorous proof.)
We can thus conclude that
\[
\Lambda \textrm{ regular}
\quad\Longrightarrow\quad M \textrm{ strictly convex}.
\]
Since $M$ was already shown to be weakly convex in any case, it follows that
\[
\mathbb P(\Lambda \textrm{ regular})>0
\quad\Longrightarrow\quad M\textrm{ strictly convex}.
\]
Now we turn to the existence of a solution (once existence is established, uniqueness will follow from strict convexity). Let $\mathrm{proj}_K:\mathbb R^d\to K$ denote the projection onto the set $K$, which is well-defined since $K$ is closed and convex, and of course satisfies
\[
\|x - y\|
\geq \|x - \mathrm{proj}_K(y)\|
,\qquad x\in K,\quad y\in\mathbb R^d.
\]
Since $\Lambda$ is concentrated on $K$, the above inequality holds $\Lambda$-almost surely with respect to $x$. Let $T$ be the optimal map from $\Lambda$ to $\gamma$ (a proper map almost surely, as argued above).  Observe that
\[
d^2(\Lambda,\gamma)
=\ownint{K}{}{\|T(x) - x\|^2}\Lambda
\ge \ownint{}{K} {\| \mathrm{proj}_K(T(x)) - x \|^2}\Lambda
\ge d^2(\Lambda,\mathrm{proj}_K\#\gamma),
\]
since $(\mathrm{proj}_K\circ T)\#\Lambda=\mathrm{proj}_K\#(T\#\Lambda)=\mathrm{proj}_K\#\gamma$.  This measure is concentrated on $K$, and taking expectations gives $M(\gamma)\ge M(\mathrm{proj}_K\#\gamma)$.  Hence, the infimum of $M$ equals the infimum of $M$ on $\mathcal P(K)$, the collection of probability measures supported on $K$ (or else, we could project all the remaining mass to $K$ to reduce the total cost further).  The restriction of $M$ to $\mathcal P(K)$ is a continuous functional on a compact set (measures whose support is contained in a common compactum are a compact set in Wasserstein space), and existence follows.

In order to establish that $\lambda$ minimises $M$, we need to justify the following facts:
\begin{align*}
d^2(T\#\lambda,\theta)\qquad\textrm{is measurable for all }\theta\in P(K);\\
\E \ownint {\mathbb R^d}{}{\left(\frac12\|x\|^2 - \phi(x)\right)}{\theta(x)}
= \ownint {\mathbb R^d}{}{\left(\frac12\|x\|^2 - \E\phi(x)\right)}{\theta(x)};\\
\E \phi(x)
=\|x\|^2 /2 + C
\end{align*}
The space
\[
C_b(K,\mathbb R^d)
=\{f:K\to \mathbb R^d;f \textrm{ continuous}\}
\]
(endowed with the supremum norm $\|f\|_\infty=\sup_{x\in K}\|f(x)\|$) is a separable Banach space and therefore any random element $T:(\Omega,\mathcal F,\mathbb P)\to C_b(K,\mathbb R^d)$ is Bochner measurable.  Clearly
\[
d(T\#\lambda,R\#\lambda)
\le \sqrt{\ownint {\mathbb R^d}{}{\|T(x) - R(x)\|^2}{\lambda(x)}}
\le \sup_{x\in K}\|T(x) - R(x)\|
=\| T - R\|_\infty
\]
so that $\Lambda=T\#\lambda$, when viewed as a random measure in the Wasserstein space, is a continuous function of $T$, and hence measurable.  Again by continuity, $d^2(\Lambda,\theta):(\Omega,\mathcal F,\mathbb P)\to \mathbb R$ is measurable.

Let us now show the remaining two Fubini-type assertions.  For simplicity, we assume that $K$ includes the origin.  Then the convex potential $\phi$ of $T$ can be recovered as the line integral
\[
\phi(x)
=\phi_T(x)
=\ownint01{\innprod {T(sx)}x}s
\]
and the Legendre transform $\phi^*:K\to\R$ of $\phi$ by
\[
\phi^*(y)
=\sup_{x\in K}\innprod xy - \phi(x).
\]
One can then verify the following properties (invoking the uniform continuity of $T$), which imply in particular that $\phi$ and $\phi^*$ are measurable:
\begin{enumerate}
\item $\phi$ is a continuous and bounded, hence an element of $C_b(K)$;
\item the same holds for $\phi^*$;
\item the map $T\mapsto\phi$ from $C_b(K,\mathbb R^d)$ to $C_b(K)$ is Lipschitz;
\item the map $\phi\mapsto\phi^*$ from $C_b(K)$ to itself is Lipschitz.
\end{enumerate}
Indeed, we write
\[
\phi(x) - \phi(y)=
\ownint01{\innprod {T(sx)}x}s
-\ownint01{\innprod {T(sy)}y}s
=\ownint01{\innprod {T(sx) - T(sy)}x}s
+\ownint01{\innprod {T(sy)}{x-y}}s
\]
and notice that the first integral vanishes as $y\to x$ by uniform continuity.  The last integral also vanishes by the Cauchy--Schwarz inequality, because $T$ is bounded (as a continuous function on $K$).

Again by the Cauchy--Schwarz inequality,
\[
|\phi_T(x) - \phi_R(x)|
\le \|T - R\|_\infty \|x\|
\le \|T - R\|_\infty \sup_{x\in K}\|x\|
<\infty,
\]
so $\phi:C_b(K,\mathbb R^d)\to C_b(K)$ is Lipschitz with constant $d_K=\sup_{x\in K}\|x\|$.  It is also obvious that $\|\phi^*_T - \phi^*_R\|_\infty \le \|\phi_T - \phi_R\|_\infty$ and that $\phi^*$ is bounded on $K$ because $\phi$ is bounded and $K$ are bounded.  Uniform continuity can be verified directly as follows.  Fix $\delta>0$ and $y,z\in K$ with $\|z-y\|<\delta$.  Then for any $\epsilon>0$ we can pick some $x\in K$ such that
\[
\phi^*(z)\le \innprod xz-\phi(x)+\epsilon
= \innprod xy - \phi(x) + \epsilon + \innprod x{z-y}
\le \phi^*(y) + \epsilon + \delta\sup_{x\in K}\|x\|.
\]
Letting $\epsilon\to0$ and since the supremum is finite, we see that $\phi^*(z) - \phi^*(y)\le d_K \|y-z\|$;  interchanging the roles of $y$ and $z$ above shows that in fact $|\phi^*(z) - \phi^*(y)|\le d_K\|y-z\|$, so $\phi^*$ is even Lipschitz.

It now remains to show that
\[
\E\phi_{T}(x)
=\ownint 01{\innprod {\E T(sx)}x}s,
\qquad x\in K,
\qquad\textrm{and}\quad
\E\ownint K{}{\phi}\theta
=\ownint K{}{\E\phi}\theta
\qquad \forall\theta\in P(K).
\]
Both equalities hold when $T$ is a simple function.  Since $C_b(K,\mathbb R^d)$ is separable, any $T$ can be approximated by simple functions $T_n$ such that $\|T_n\|\le 2\|T\|$ almost surely.  The assumption that $T$ takes values in $K$ almost surely implies that $\E T_n\to \E T$ in the Bochner sense, which means that $\|\E T_n - \E T\|_\infty\to0$.  Let $\phi_n$ be the convex potential of $T_n$.  Then $\|\phi_n - \phi\|_\infty\le \|T_n - T\|_\infty d_K$ and $\|\phi_n\|_\infty\le \|T_n\|_\infty d_K$, which is integrable.  It follows that $\E \phi_n\to \E\phi$ in the Bochner sense in $C_b(K)$, and in particular $\E \phi_n(x)\to \E\phi(x)$ for all $x\in K$, proving the first equality.  The second equality is proven by a similar approximation argument.
\end{proof}

\begin{proof}[Proof of Lemma~\ref{lem:smoothing}]
It is assumed that $\psi(z)=\psi_1(\|z\|)$ with $\psi_1$ non-increasing, strictly positive and
\[
\ownint{\mathbb R^d}{}{\psi(z)}z
=1
=\ownint{\mathbb R^d}{}{\|z\|^2\psi(z)}z.
\]
Let $\Psi(A)=\ownint A{}{\psi(x)}x$ be the corresponding probability measure and recall that $\psi_\sigma(x)=\sigma^{-d}\psi(x/\sigma)$ for $\sigma>0$.

For $y\in K$ set $\tilde\mu_y=\delta\{y\}\ast\psi_\sigma$ and its restricted renormalized version $\mu_y=(1/\tilde \mu_y(K))\tilde \mu_y|_K$, so that $\widehat\Lambda_i=(1/m)\sum_{j=1}^m\mu_{x_j}$, and it is assumed that $m\ge1$ and $x_j\in K$ (because $\Lambda_i(K)=1$).

One way (certainly not optimal, unless $m=1$) to couple $\widehat\Lambda_i$ with $\widetilde\Pi_i/m$ is to send the $1/m$ mass of $\mu_{x_j}$ to $x_j$.  This gives
\[
d^2(\widehat\Lambda_i,\tilde \Pi_i/m)
\le \frac1m\sum_{j=1}^md^2(\mu_{x_j},\delta\{x_j\})
=\frac1m\sum_{j=1}^m \frac 1{\tilde\mu_{x_j}(K)}\ownint K{}{\|x-x_j\|^2\psi_\sigma(x-x_j)}x.
\]
However, for an arbitrary $y\in K$,
\[
\frac 1{\tilde\mu_y(K)}\ownint K{}{\|x-y\|^2\psi_\sigma(x-y)}x
=\frac 1{\tilde\mu_y(K)}\sigma^2\ownint {(K-y)/\sigma}{}{\|z\|^2\psi(z)}z.
\]
The last displayed integral is bounded by 1.  Hence, we seek a lower bound, uniformly in $y$ and $\sigma$, for
\[
\tilde\mu_y(K)
=\ownint K {}{\psi_\sigma(x-y)}x
=\ownint{(K-y)/\sigma}{}{ \psi(x)}x
=\Psi\left(\frac{K-y}\sigma\right).
\]
Since $K-y$ is a convex set that contains the origin, the collection of sets $\{\sigma^{-1}(K-y)\}_{\sigma>0}$ is increasing as $\sigma\searrow0$.  Consequently, if $\sigma\le1$,
\[
\Psi\left( \frac{K-y}\sigma\right)
\ge \Psi(K-y)
=\ownint {K-y}{}{\psi(x)}x
\ge \ownint {K-y}{}{\psi_1(d_K)}x
=\psi_1(d_K)\mathrm{Leb}(K)
>0.
\]
Here $d_K=\sup\{\|x-y\|:x,y\in K\}$ is the (finite) diameter of $K$, and we have used the monotonicity of $\psi_1$.

It follows that for $C_{\psi,K}=[\psi_1(d_K)\mathrm{Leb}(K)]^{-1}<\infty$ (depending only on $\psi$ and $K$),
\begin{equation}\label{eq:upperboundsmooth}
d^2(\mu_y,\delta_y)
\le C_{\psi,K}\sigma^2
,\qquad y\in K
,\quad \sigma\le1.
\end{equation}
Since the bound is uniform in $y$, the proof is complete.  In the context of Remark~\ref{rem:nonisotropic}, one simply needs to replace the term $\psi_1(d_K)$ in $C_{\psi,K}$ by $\delta_\psi(d_K)$.
\end{proof}

\begin{proof}[Proof of Theorem~\ref{thm:consistency}]
Convergence in probability in part (1) follows as in \cite[pp.\ 793--794]{panzem15}, using \eqref{eq:smoothbound}. For convergence almost surely, let $a=(a_1,\dots,a_d)\in\mathbb R^d$.  A straightforward generalisation of the argument in \cite[pp.\ 794--795]{panzem15} gives
\[
\mathbb P\left( \frac{\widetilde \Pi_i((-\infty,a])} {\tau_n} - \Lambda_i((-\infty,a]) \to 0  \right)
= 1,
\]
where for $-\infty\le a_i\le b_i\le\infty$ ($i=1,\dots,d$) we denote
\[
(a,b] = (a_1,b_1]\times\dots\times(a_d,b_d].
\]
Consequently
\[
\mathbb P\left( \frac{\widetilde \Pi_i((-\infty,a])} {\tau_n} - \Lambda_i((-\infty,a]) \to 0
\textrm{ for any }  a \in \mathbb Q^d  \right)
= 1.
\]
For a general $a\in\mathbb R^d$ there exist sequences $a^k\nearrow a\swarrow b^k$ with $a^k,b^k\in\mathbb Q^d$ (that is, $a_i^k\nearrow a_i\swarrow b_i^k$ for any coordinate $i$).  Since for any $k$
\[
\frac{\widetilde \Pi_i((-\infty,a])} {\tau_n} - \Lambda_i((-\infty,a])
\le \frac{\widetilde \Pi_i((-\infty,b^k])} {\tau_n} - \Lambda_i((-\infty,b^k]) + \Lambda_i((-\infty,b^k]) - \Lambda_i((-\infty,a]),
\]
it follows that with probability one
\[
\limsup_{n\to\infty}
\frac{\widetilde \Pi_i((-\infty,a])} {\tau_n} - \Lambda_i((-\infty,a])
\le 
  \Lambda_i\left(  (-\infty,b^k] \setminus (-\infty,a]  \right)
  \to0,
  \qquad k\to\infty,
\]
as the sequence of sets at the right-hand side converges monotonically to the empty set.

Similarly, with probability one
\[
\liminf_{n\to\infty}
\frac{\widetilde \Pi_i((-\infty,a])} {\tau_n} - \Lambda_i((-\infty,a])
\ge  \Lambda_i\left(  (-\infty,a] \setminus (-\infty,a^k]  \right)
  \to\Lambda_i((-\infty,a] \setminus (-\infty,a)),
\]
and the right-hand side vanishes because $\Lambda_i$ is assumed absolutely continuous (the set $(-\infty,a] \setminus (-\infty,a)$ is union of $d$ $d-1$-dimensional rays).  Specifying $a=\infty$ shows that almost surely $\widetilde\Pi_i(K)/\tau_n\to1$ and we conclude that almost surely $\widetilde\Pi_i/\widetilde\Pi_i(K)\to \Lambda_i$ weakly.  Further, $d(\widehat\Lambda_i,\widetilde\Pi_i/N_i)\to0$ since $\sigma_n\to0$ by Lemma~\ref{lem:smoothing}.

We sketch the main ideas of the proof of (2);  more details can be found in \cite[pp.\ 795--797]{panzem15}.  We wish to show that
\[
\widehat M_n(\gamma)
=\frac1n\sum_{i=1}^n d^2(\widehat\Lambda_i,\gamma)
\to \mathbb Ed^2(\Lambda_i,\gamma)
= M(\gamma)
,\quad\textrm{uniformly in }\gamma.
\]
In order to do this we write
\[
\widehat M_n(\gamma)  -  M(\gamma)
= \left[\widehat M_n(\gamma) - M_n(\gamma)\right]
+ \left[M_n(\gamma) - M(\gamma)\right],
\]
where we introduce the empirical Fr\'echet functional
\[
M_n(\gamma)
=\frac1n\sum_{i=1}^n d^2(\Lambda_i,\gamma).
\]

Since for any three probability measures on $K$ it holds that
\begin{align*}
d(\mu,\nu)
\le \sqrt{\sup_{\gamma\in P(K^2)}\ownint {K^2}{}{ \|x-y\|^2}{\gamma(x,y)}}
\le \sqrt{\sup_{x,y\in K}\|x-y\|^2}
= d_K
<\infty;\\
|d^2(\mu,\rho) - d^2(\nu,\rho)|
=|d(\mu,\rho) + d(\nu,\rho)|  |d(\mu,\rho) - d(\nu,\rho)|
\le 2 d_K d(\mu,\nu),
\end{align*}
we see that
\[
\sup_{\gamma\in P(K)} |\widehat M_n(\gamma) - M_n(\gamma)|
\le \frac{2d_K}n  \sum_{i=1}^n  d\left(  \widehat\Lambda_i, \Lambda_i\right)
=  \frac{2d_K}n \sum_{i=1}^n X_{ni}
= 2d_K \overline X_n.
\]
Each $X_{ni}$ is a function of $T_i$, $\Pi_i^{(n)}$ and $\sigma_i^{(n)}$, and $0\le X_{ni}\le d_K$.  If $\sigma_i^{(n)}$ is a function of $\widetilde\Pi_i^{(n)}=T_i\#\Pi_i^{(n)}$ only, then $X_{ni}$ are iid across $i$.  Part (1) shows that $X_{n1}\to0$ in probability and by the bounded convergence theorem $\mathbb E\overline X_n=\mathbb EX_{n1}\stackrel p\to0$ and therefore the above expression converges to 0 in probability.  In general, $L^1$-convergence of random variables does not guarantee convergence almost surely.  As we deal with averages, however, almost sure convergence can be established:  let $Y_{ni}=X_{ni} - \mathbb EX_{ni}\in[-d_K,d_K]$.  Then $Y_{ni}$ are mean zero iid random variables, so that
\[
\mathbb P\left(  \left| \overline X_n - \mathbb E\overline X_n \right| > \epsilon \right)
=\mathbb P\left(  \overline Y_n^4  > \epsilon^4\right)
\le \frac{n\mathbb E\left[Y_{n1}^4\right]  + 3n(n-1)\mathbb E\left[Y_{n1}^2\right]} {\epsilon^4n^4}
\le  \frac {3 \max(d_K^4,d_K^2)}  {\epsilon^4 n^2}.
\]
By the Borel--Cantelli lemma, $|\overline X_n - \mathbb E\overline X_n|\stackrel{as}\to0$, hence $\overline X_n\stackrel{as}\to0$.

If the smoothing is not carried out independently across trains, then $X_{ni}$ may be correlated across $i$.  In that case, one can introduce the functional $M_n^*(\gamma)=n^{-1}\sum_{i=1}^n d^2\left(\widetilde\Pi_i/N_i,\gamma\right)$ and proceed as in \cite{panzem15}.  For $M_n^*$ to be well-defined one may use Lemma~\ref{lem:numpoints} and that requires $\tau_n/\log n\to\infty$.

Finally, observe that by the strong law of large numbers $M_n(\gamma)\stackrel{as}\to M(\gamma)$ for all $\gamma\in P(K)$.  That the convergence is uniform follows from the equicontinuity of the collection $\{M_n\}_{n=1}^\infty$ (they are $2d_K$-Lipschitz).  We have thus established
\[
\sup_{\gamma\in P(K)} |\widehat M_n(\gamma) - M(\gamma)|
\stackrel{as}\to 0,
\qquad n\to\infty.
\]
By standard arguments, the minimiser $\widehat\lambda_n$ of $\widehat M_n$ converges to the minimiser $\lambda^*$ of $M$, since the latter is unique by Theorem~\ref{thm:meanTid}.  But $\lambda^*=\lambda$ by the hypothesis.
\end{proof}

\subsection*{Proofs of statements from Section~\ref{appendix}}

\begin{proof}[Proof of Lemma~\ref{lem:densityodist}]  For any $1>\epsilon>0$ there exists $0<t_\epsilon$ such that for $t<t_\epsilon$,
\[
\frac {\mathrm{Leb}(B_t(x_0)\cap G)} {\mathrm{Leb}(B_t(x_0))}
> 1 - \epsilon^d.
\]
Fix $z$ such that $t = t(z) = \|z - x_0\|<t_\epsilon$.  The intersection of $B_t(x_0)$ with $B_{2\epsilon t}(z)$ includes a ball of radius $\epsilon t$ centred at $y=x_0+(1-\epsilon)(z-x_0)$, so that
\[
\frac {\mathrm{Leb}(B_t(x_0)\cap B_{2\epsilon t}(z))}  {\mathrm{Leb}(B_t(x_0))}
\ge \frac{\mathrm{Leb}(B_{\epsilon t}(y)) } {\mathrm{Leb}(B_t(x_0))}
=\epsilon ^d.
\]
It follows that $G\cap B_{2\epsilon t}(z)$ is nonempty.  In other words:  for any $\epsilon>0$ there exists $t_\epsilon$ such that if $\|z-x_0\|<t_\epsilon$, then there exists $x\in G$ with $\|z - x\|\le 2\epsilon t(z) = 2\epsilon \|z - x_0\|$.  This means precisely that $\delta(z)=o(\|z - x_0\|)$ as $z\to x_0$.
\end{proof}

\begin{proof}[Proof of Lemma~\ref{lem:linfchull}]  Assume $\epsilon<\rho$ (there is nothing to prove otherwise).  Take a corner of the $\ell_\infty$ ball of radius $\rho'<\rho$ around $x_0$,
\[
y = x_0 + \rho'(e_1,\dots,e_d)
,\qquad e_d\in\{\pm1\},
\]
and write $y=\sum a_iz_i$ as a (finite) convex combination of elements of $Z$.  Then $\tilde y=\sum a_i\tilde z_i\in \mathrm{conv}(\tilde Z)$ is such that $\|\tilde y - y\|_\infty\le \epsilon$.  It follows that $\tilde y$ lies at the same quadrant as $y$ with each coordinate larger in absolute value than $\rho'-\epsilon$.  In other words, $\tilde y$ is ``more extreme" than the corner
\[
x_0 + (\rho' - \epsilon)(e_1,\dots,e_d)
\]
of the $\ell_\infty$-ball $B_{\rho'-\epsilon}^\infty(x_0)$.  Since this is true for all the corners, $\mathrm{conv}(\tilde Z)\supseteq B_{\rho'-\epsilon}(x_0)$ for any $\rho'<\rho$.  Now let $\rho'\nearrow \rho$ to conclude.
\end{proof}

\end{document}

%% file: macros.tex
\newtheorem{thm}{Theorem}[section]

\newtheorem{lem}[thm]{Lemma}
\newtheorem{prop}[thm]{Proposition}

\newtheorem{rem}{Remark}

\def\clap#1{\hbox to 0pt{\hss#1\hss}}

\newcommand{\ownint}[4]{{\int_{#1}^{#2} \! #3 \, \mathrm{d}#4}}

\newcommand{\innprod}[2]{\langle#1,#2\rangle}

\newcommand {\E}{\mathbb{E}}
\renewcommand {\P}{\mathbb{P}}
\newcommand {\R}{\mathbb{R}}